\documentclass{amsart}[12pt]
\usepackage{amssymb,amsmath,amscd,graphicx,epsfig,amsthm}
\usepackage[all,poly,knot,dvips]{xy}
\usepackage{txfonts}
\numberwithin{equation}{section}

\def\wh{\widehat}

\def\C{{\mathbb C}}
\def\R{{\mathbb R}}
\def\RR{{\mathcal R}}
\def\EE{{\mathcal E}}
\def\N{{\mathbb N}}
\def\Z{{\mathbb Z}}

\def\FFF{\mathfrak F}
\def\P{\mathcal P}

\def\H{\mathcal H}

\def\L{\mathfrak L}

\def\C{{\mathbb C}}
\def\L{\mathcal L}
\def\R{{\mathbb R}}
\def\N{{\mathbb N}}

\def\A{\mathcal A}
\def\B{\mathbf B}
\def\C{\mathbb C}

\def\D{{\mathcal D}}

\def\FF{{\mathcal F}}

\def\HH{\mathbf  H}
\def\L{\mathcal L}

\def\N{\mathbf N}
\def\O{\mathcal O}
\def\P{\mathcal P}
\def\PP{\mathbf P}
\def\R{\mathbb R}

\def\W{\mathcal W}
\def\Z{\mathbb Z}

\def\ii{{\bf i}}
\def\jj{{\bf j}}
\def\kk{{\bf k}}

\def\k{\mathfrak k}
\def\n{\mathfrak n}
\def\b{\mathfrak b}
\def\x{\mathbf x}
\def\y{\mathbf y}
\def\z{\mathbf z}

\DeclareMathSymbol{\DDelta}{\mathalpha}{letters}{"01}
\def\1{{\bf 1}}
\def\one{\mathbf 1}

\makeatletter
\@addtoreset{equation}{section}
\makeatother

\def\bx{\begin{example}}
\def\ex{\end{example}}

\newtheorem{theorem}{Theorem}[section]
\newtheorem{examp}{Example}[section]
\newtheorem{coroll}{Corollary}[section]

\newtheorem{lemma}{Lemma}[section]
\newtheorem{remark}{Remark}[section]

\newtheorem{proposition}{Proposition}[section] 
\newtheorem{definition}{Definition}[section]
\def\br{\begin{remark}}
\def\er{\end{remark}}
\def\bt{\begin{theorem}}
\def\et{\end{theorem}}
\def\bc{\begin{coroll}}
\def\ec{\end{coroll}}
\def\bl{\begin{lemma}}
\def\el{\end{lemma}}
\def\bd{\begin{definition}}
\def\ed{\end{definition}}
\def\bp{\begin{proposition}}
\def\ep{\end{proposition}}
\def\be{\begin{equation}}
\def\ee{\end{equation}}
\def\bea{\begin{eqnarray}}
\def\eea{\end{eqnarray}}
\def\beas{\begin{eqnarray*}}
\def\eeas{\end{eqnarray*}}

\def\one{\mathbf 1}
\def\Id{\operatorname{Id}}

\def\Id{{\operatorname {Id}}}
\def\Jac{{\operatorname {Jac}}}

\def\Poi{{\{\cdot,\cdot\}}}

\def\Span{{\operatorname{span}}}

\def\Trace{\operatorname{tr}}

\def\codim{\operatorname{codim}}
\def\corank{\operatorname{corank}}
\def\ddt{\frac{d}{dt}}
\def\deg{{\operatorname{deg}}}
\def\diag{\operatorname{diag}}

\def\rank{\operatorname{rank}}

\def\sign{\mbox{sign}}
\def\tr{\operatorname{tr}}

\date{}
\begin{document}
                                                                  
\title[Coxeter--Toda flows from a cluster algebra perspective]
{Generalized B\"acklund--Darboux transformations for Coxeter--Toda flows from a cluster algebra perspective}

\author{Michael Gekhtman}
\address{Department of Mathematics, University of Notre Dame, Notre Dame,
IN 46556}
\email{mgekhtma@nd.edu}

\author{Michael Shapiro}
\address{Department of Mathematics, Michigan State University, East Lansing,
MI 48823}
\email{mshapiro@math.msu.edu}

\author{Alek Vainshtein}
\address{Department of Mathematics AND Department of Computer Science, University of Haifa, Haifa,
Mount Carmel 31905, Israel}
\email{alek@cs.haifa.ac.il}

\date\today

\begin{abstract} 
We present the third in the series of papers describing Poisson 
properties of planar directed networks in the disk or in the annulus. 
In this paper we concentrate on special networks $N_{u,v}$ in the disk 
that correspond to the choice of a pair $(u,v)$ of  Coxeter elements 
in the symmetric group $S_n$ and the corresponding networks $N_{u,v}^\circ$ 
in the annulus. Boundary measurements for $N_{u,v}$ 
represent elements of the Coxeter double Bruhat cell 
$G^{u,v}\subset GL_n$. The Cartan subgroup $\HH$ acts on ${G}^{u,v}$ by conjugation.  
The standard Poisson structure on the space of weights of  $N_{u,v}$
induces a Poisson structure on $G^{u,v}$, and hence on the quotient $G^{u,v}/\HH$, 
which makes the latter into the phase space for an appropriate Coxeter--Toda lattice. 
The boundary measurement for $N_{u,v}^\circ$ is a rational function that 
coincides up to a nonzero factor with the Weyl function for the 
boundary measurement for $N_{u,v}$. The corresponding Poisson bracket on the space of 
weights of $N_{u,v}^\circ$  induces a Poisson bracket on the certain space
$\RR_n$ of rational functions, which appeared previously in the context of Toda flows.

Following the ideas developed in our previous papers,  we introduce a cluster algebra $\A$ 
on $\RR_n$ compatible with the obtained Poisson bracket. 
Generalized B\"acklund--Darboux transformations map solutions of one Coxeter--Toda
lattice to solutions of another preserving the corresponding Weyl function.
Using network representation, we construct generalized B\"acklund--Darboux transformations 
as appropriate sequences of cluster transformations in $\A$.
\end{abstract}

\maketitle

\section{Introduction}
This is the third in the series of papers in which we investigate Poisson geometry of directed networks. 
In \cite{GSV3, GSV4}, we studied Poisson structures associated with weighted directed networks
in a disk and in an annulus. The study was motivated in part by Poisson properties of cluster algebras.
In fact, it was shown in \cite{GSV3} that 
if a universal  Poisson bracket on the space of edge
weights of a directed network in a disk satisfy an analogue of the Poisson--Lie property with respect to  concatenation, then the
Poisson structure induced by this bracket on the corresponding Grassmannian is compatible with the cluster  
algebra structure in the homogeneous coordinate ring of the Grassmannian.
In this paper we deal with  an example that ties together objects and concepts from
the theory of cluster algebras and directed networks with the theory of integrable systems.

Integrable systems in question are the {\em Toda flows\/} on $GL_n$. These are commuting Hamiltonian
flows generated by conjugation-invariant functions on $GL_n$ with respect to  the standard Pois\-son--Lie structure.
Toda flows (also known as {\em characteristic Hamiltonian systems\/} \cite{R}) are defined for an arbitrary standard 
semi-simple Poisson--Lie group, but we will concentrate on the $GL_n$ case, 
where as a maximal algebraically independent family of conjugation-invariant functions  one can choose 
$F_k : GL_n \ni X \mapsto \frac{1}{k}
\tr X^ k$, $k=1,\ldots, n-1$. The equation of motion generated by $F_k$ has a 
{\em Lax form}:
\be
\label{Lax_intro}
\ddt X = \left [ X,\  - \frac{1}{2} \left ( \pi_+(X^k) - \pi_-(X^k)\right ) \right ],
\ee
where $\pi_+(A)$ and $\pi_-(A)$ denote strictly upper and lower parts of a matrix $A$.

Any double Bruhat cell  $G^{u,v}$, $u,v \in S_n$, is a regular Poisson submanifold  in $GL_n$ invariant under the 
right and left multiplication by elements of the maximal torus (the subgroup of diagonal matrices) $\HH \subset GL_n$. 
In particular,  $G^{u,v}$ is invariant under the conjugation by elements of $\HH$. The standard Poisson--Lie structure 
is also invariant under the conjugation action of $\HH$ on $GL_n$. This means that Toda
flows defined by (\ref{Lax_intro}) induce commuting Hamiltonian flows on $G^{u,v} / \HH $ where  $\HH$ acts on 
$G^{u,v}$ by conjugation.
In the case when $v=u^{-1}=(n\ 1\ 2 \ldots  n-1)$,
$G^{u,v}$ consists of tridiagonal matrices with nonzero off-diagonal entries,  $G^{u,v} / \HH $  can be conveniently described
as the set $\Jac$ of {\em Jacobi matrices} of the form
\be
\label{Jac}
L =  \left( \begin{array}{ccccc}
b_{1} & 1 & 0 & \cdots & 0 \\
a_{1}&b_{2}& 1 &\cdots&0\\
&\ddots&\ddots&\ddots&\\
&&&b_{n-1}& 1 \\
0&&&a_{n-1}&b_{n}
\end{array}\right), \quad  a_1 \cdots a_{n-1} \ne 0, \ \det L \ne 0.
\ee
Lax equations \eqref{Lax_intro} then become the equations of the {\em finite nonperiodic Toda hierarchy\/}:
$$
\ddt L=[L,\pi_-(L^k)],
$$
the first of which, corresponding to $k=1$, is the celebrated {\it Toda lattice} 
\begin{eqnarray*}
\ddt a_j&=&a_j(b_{j+1}-b_j), \quad j=1,\dots , n-1,\\
\nonumber
\ddt b_j&=&(a_j-a_{j-1}), \quad j=1,\dots ,n,
\end{eqnarray*}
 with the boundary conditions $a_0=a_n=0$. Recall that $\det L$ is a Casimir function for the
 standard Poisson--Lie bracket.
 The level sets of the function $\det L$  foliate $\Jac$ into $2(n-1)$-dimensional symplectic
 manifolds, and the Toda hierarchy defines a completely integrable system on every symplectic leaf. 
Note that although Toda flows on an arbitrary double Bruhat cell $G^{u,v}$ can be exactly solved via the so-called
 {\em factorization method\/} (see, e.g. \cite{r-sts}), in most cases the dimension of symplectic leaves in 
 $G^{u,v} / \HH $ exceeds $2(n-1)$, which means that conjugation-invariant functions do not
 form a Poisson commuting family rich enough to ensure Liouville complete integrability.

 An important role in the study of Toda flows is played by the
{\it Weyl function}
\begin{equation}
m(\lambda)=m(\lambda;X)=((\lambda\one-X)^{-1} e_1,e_1)=\frac{q(\lambda)}{p(\lambda)},
\label{weyl}
\end{equation}
where $p(\lambda)$ is the characteristic polynomial of $X$ and $q(\lambda)$ is the characteristic polynomial 
of the $(n-1)\times(n-1)$ submatrix of $X$ formed by deleting the first row and column 
(see, e.g., \cite{dlnt,moser,bf}). Differential equations that describe the evolution of $m(\lambda;X)$ induced by Toda flows do not
depend on the initial value $X(0)$ and are easy to solve: though nonlinear,  they 
are also induced by {\em linear differential equations with constant coefficients\/} on the space
\be
\label{Rat}
\left \{ M(\lambda) =\frac{Q(\lambda)}{P(\lambda)} 
\ : \deg P = n,\  \deg Q = n-1,\  \text{$P, Q$ are coprime,  $P(0) \ne 0$} 
\right \}   
\ee
 by the map $ M(\lambda) \mapsto m(\lambda) = -\frac{1}{H_0} M(-\lambda)$, where 
 $H_0=\lim_{\lambda\to \infty}\lambda M(\lambda)\ne 0$. 
 
 It is easy to see that $m(\lambda;X)$ is invariant under the action of $\HH$ on $G^{u,v}$ by conjugation. Thus we have
 a map from  $G^{u,v} / \HH $ into the space 
$$
\W_n=\left \{ m(\lambda) =\frac{q(\lambda)}{p(\lambda)}
\ : \deg p = n,\  \deg q = n-1,\  \text{$p, q$ are monic and coprime,  $p(0) \ne 0$}
\right \}.
$$
 In the tridiagonal  case, this map, sometimes called the {\em Moser map}, is invertible: it is  a classical 
 result in the theory of moment problems that matrix entries of an element in $\Jac$ can be restored from 
its Weyl function $m(\lambda;X)$ via determinantal formulas for matrix entries
of $X$ in terms of Hankel determinants built from the coefficients of the Laurent expansion of $m(\lambda;X)$. 
These formulas  go back to the work of Stieltjes on continuous fractions \cite{st} (see, e.g.  \cite{akh} for details). 
 
In this paper, we study  double Bruhat cells $G^{u,v}$ that share  common features with the tridiagonal case: 

(i)  the Toda hierarchy defines a completely integrable system on level sets of the determinant in $G^{u,v} / \HH$, and 

(ii) the Moser map $m_{u,v}:G^{u,v} / \HH\to\W_n$ defined in the same way as in the tridiagonal case is invertible.  

\noindent We will see that double Bruhat cells $G^{u,v}$ associated with any pair 
  of Coxeter elements $u,v \in S_n$ enjoy these properties.
(Recall that a Coxeter element in $S_n$ is a product 
of $n-1$ distinct elementary transpositions.) Double Bruhat cells of this kind has previously appeared 
(for an arbitrary simple Lie group) in \cite{Reshetikhin&Co} in the context of integrable systems and 
in \cite{CAIII, YZ} in connection with cluster algebras of finite type. We will call any such 
double Bruhat cell  a {\em Coxeter double Bruhat cell}. 
Integrable equation induced on $G^{u,v} / \HH $ by Toda flows will be called {\em Coxeter--Toda lattices}. This term was first
used in \cite{Reshetikhin&Co} in the case $u=v$ for an arbitrary simple Lie group, which generalizes the 
{\em relativistic Toda lattice\/} that corresponds to the choice 
$u=v=s_{n-1} \cdots s_1$ in $GL_n$. In \cite{FayGekh1, FayGekh2}, the corresponding integrable systems for 
$v=s_{n-1} \cdots s_1$ and an arbitrary Coxeter element $u$  were called {\em elementary Toda lattices}. In the latter case, 
$G^{u,v} / \HH$ can be described as a subset of Hessenberg matrices subject to certain rank conditions on submatrices. 
The  tridiagonal case corresponds to the choice $v=s_{n-1} \cdots s_1$, $u=s_1 \cdots s_{n-1}$. 

Since Coxeter--Toda flows associated with different choices of $(u,v)$  lead to the same evolution of the Weyl 
function, and the corresponding Moser maps are invertible, one can construct transformations between  different 
$G^{u,v} / \HH$ that preserve the corresponding Coxeter--Toda flows and thus serve as {\em generalized B\"acklund--Darboux 
transformations\/}  between them. 

Our goal is to describe these transformations from the cluster algebra point 
of view. To this end, we construct a cluster algebra of rank $2n-2$ associated with an extension of the space (\ref{Rat})
\[
\mathcal{R}_n =\left\{\frac{Q(\lambda)}{P(\lambda)}\ : \deg P = n,\  \deg Q < n,\ \text{$P, Q$ are coprime, $P(0)\ne 0$} \right\}.   
\]
(Note that $\W_n$ is embedded into $\RR_n$ as a codimension~$1$ subspace.)
Distinguished clusters $\x_{u,v}$ in this algebra correspond to Coxeter double Bruhat cells, and are
formed by certain collections of Hankel determinants built out of coefficients of the Laurent expansion 
of an element in $\RR_n$. Sequences of cluster transformations connecting  these distinguished
clusters are then used as the main ingredient in the construction of
generalized B\"acklund--Darboux transformations.

The insight necessary to implement this construction is drawn from two sources: 

(i) the procedure for the inversion of the Moser map, that can be viewed as a generalization of the inverse moment problem, and

(ii) interpretation of functions in $\mathcal{R}_n$ as boundary measurement functions associated with a particular 
kind of networks in an annulus.

Before discussing the organization of the paper, we would like to make two remarks. First, birational
transformations between $G^{u,s_{n-1} \cdots s_1}/\HH $ and $G^{u',s_{n-1} \cdots s_1}/\HH $ for two different 
Coxeter elements $u, u'$ that serve as  generalized B\"acklund--Darboux transformation between the corresponding elementary 
Toda lattices were first studied in \cite{FayGekh1}. Second, a cluster algebra closely related to the one we 
considered here recently appeared in \cite{kedem} and was a subject of a detailed combinatorial study in 
paper \cite{dk}, where cluster mutations along the edges of a certain subgraph of its exchange graph were shown 
to describe an evolution of an {\it  $A_n$ type Q-system}  -- a discrete evolution that
arises in the analysis of the XXX-model, which is an example of a {\em quantum\/} integrable model.
In \cite{dk}, solutions of the Q-system are represented as Hankel determinants built from coefficients
of a certain generating function, that turns out to be rational and can be represented as a matrix element
of a resolvent of an appropriate linear operator. 

The paper is organized as follows. 

In Section 2 we go over the necessary background information on double Bruhat cells, Toda flows, cluster algebras, networks and 
associated Poisson structures. We then proceed, in Section 3, to describe  a parametrization of a Coxeter double 
Bruhat cell. This is a particular case of the Berenstein-Fomin-Zelevinsky parametrization
\cite{BFZ, FZ1}: for a generic element $X$ in  $G^{u,v}$, we consider a factorization of $X$ into elementary bidiagonal factors
consistent with the Gauss factorization of $X$, that is $X=X_- X_0 X_+$, where $X_0$ is  the diagonal matrix 
$\diag(d_1, \ldots, d_n)$, 
$X_+$ is the product of $n-1$ elementary upper bidiagonal factors $E_{i}^+(c_i^+)$,  $i=1,\ldots, n-1$,
with the order of factors in the product prescribed by $v$, and 
$X_-$ is the product of $n-1$ elementary lower bidiagonal factors $E_{i}^-(c_i^-)$,  $i=1,\ldots, n-1$, 
with the order of factors in the product prescribed by $u$.  We also give an intrinsic characterization  of a double Bruhat cell.

Elements $G^{u,v}/\HH$ are parametrized by $d_i$ and $c_i=c_i^+ c_i^-$,  $i=1,\ldots, n-1$. In Section 4 we show that 
these parameters can be restored as monomial expressions in terms of an appropriately chosen collection 
of Hankel determinants built from the coefficients of the Laurent expansion of the Weyl 
function $m(\lambda)$. (In \cite{FayGekh3}, a similar inverse
problem was solved for the case   $v=s_{n-1} \cdots s_1$, $u$  arbitrary.)  Both the choice of Hankel determinants 
and exponents entering monomial expressions for $d_i$, $c_i$ are uniquely determined by
the pair $(u, v)$. 

In Section 5, the map $X \mapsto m(\lambda; X)$ is  given a combinatorial interpretation in terms of weighted directed 
planar networks. To an elementary bidiagonal factorization of  $X\in G^{u,v}$ there corresponds a network  $N_{u,v}$ in a 
square (disk) with $n$ sources located on one side of the square and $n$ sinks located at the opposite side, both 
numbered bottom to top (see, e.g.  \cite{FZ1, FZ_Intel, Fallat}).
By gluing opposite sides of the square containing sinks and sources in such a way that each sink is glued to the 
corresponding source and adding two additional edges, one incoming and one outgoing, 
one obtains a weighted directed network in an annulus (outer and inner boundary circles of the annulus are formed by 
the remaining two sides of the square).
Networks in an annulus were studied in \cite{GSV4}. The  network we just described, $N_{u,v}^\circ$, has one sink and one 
source on the outer boundary of an annulus and, according to  \cite{GSV4}, the boundary measurement that corresponds 
to this network is a rational function $M(\lambda)$ 
in an auxiliary parameter $\lambda$. We show that $-M(-\lambda)$ is equal to 
$m(\lambda; X)$ times the product  of weights
of the incoming and outgoing edges in $N_{u,v}^\circ$. 

The determinantal formulae for 
the inverse of the Moser map are homogeneous of degree 
zero with respect to coefficients of the Laurent expansion, therefore the same formulae applied to $M(\lambda)$ also
recover $c_i$, $d_i$. Thus, we can define a map $\rho_{u,v}:(\C^*)^{2n}\to G^{u,v}/\HH$ in such a way that the through map
$$
G^{u,v}/\HH\stackrel{m_{u,v}}{\longrightarrow}\W_n\hookrightarrow \RR_n\stackrel{\x_{u,v}}{\longrightarrow}(\C^*)^{2n}\stackrel{\rho_{u,v}}{\longrightarrow}G^{u,v}/\HH
$$
is the identity map.

In the remainder of Section 5, we use the combinatorial data determined by the pair $(u,v)$ (or, in a more 
transparent way, by the corresponding network $N_{u,v}^\circ$) to construct a cluster algebra $\A=\A_{u,v}$, 
with the (slightly modified) collection $\x_{u,v}$ serving as the initial cluster.
The matrix $B_{u,v}$  that determines cluster transformations for the initial cluster is closely
related to the incidence matrix of the graph dual to $N_{u,v}^\circ$. To construct $\A_{u,v}$, we start with the 
Poisson structure induced on boundary measurement functions by a so-called {\em standard Poisson bracket\/} 
on the space of face weights associated with $N_{u,v}^\circ$ (this bracket is a particular case
of the general construction for networks in the annulus given in \cite{GSV4}). Initial cluster variables, viewed as functions 
on $\mathcal R_n$ form a 
coordinate system in which this Poisson structure takes a particular simple form: the Poisson bracket of 
logarithms of any two functions in the family is constant. This allows us to  follow
the strategy from \cite{GSV1} to construct $\A_{u,v}$ as a cluster algebra compatible with this Poisson bracket. 
We then show that $\A_{u,v}$ does not depend on the choice of Coxeter elements $u, v$, that is, that for any
$(u',v')$, the initial seed of $\A_{u',v'}$ is a seed in the cluster algebra $\A_{u,v}$. 
Therefore, the change of coordinates $T_{u,v}^{u',v'}: \x_{u,v} \mapsto\x_{u',v'}$  is accomplished by a 
sequence of cluster transformations.
Moreover, the ring of regular functions on $\mathcal{R}_n$ coincides with the localization of the complex form of 
$\A$ with respect to  the stable variables. We complete Section 5 with the discussion of the interplay 
between our results and those of \cite{dk}. In particular, we provide an alternative proof for one of the main results of \cite{dk} 
concerning the Laurent positivity of the solutions of Q-systems.

In the final section,   we 
interpret generalized B\"acklund--Darboux transformations between  Coxeter--Toda lattices corresponding to different 
pairs of Coxeter elements in terms of the cluster algebra $\A$ by observing that the map
\be\label{gbdasclust}
{\sigma}_{u,v}^{u',v'}= \rho_{u',v'} \circ  T_{u,v}^{u',v'} \circ \tau_{u,v}\ : 
G^{u,v}/\HH\to G^{u',v'}/\HH,
\ee
with $\tau_{u,v}$ being the right inverse of $\rho_{u,v}$, preserves flows generated by conjugation-invariant functions  
and makes the diagram
$$
\begin{xy}
*!C\xybox{
\xymatrix{
{G^{u,v}/\HH}\ar[rr]^{{\sigma}_{u,v}^{u',v'}}\ar[dr]_{m_{u,v}}&&{G^{u',v'}/\HH}\ar[dl]^{m_{u',v'}}\\
&{\W_n}&
} }\end{xy}
$$
commutative. We obtain explicit formulas for $\sigma_{u,v}^{u',v'}$
and, as a nice application, present formulas that transform solution of the usual Toda lattice into solutions of the relativistic one.
Besides, we explain how one represents generalized B\"acklund--Darboux transformations
as equivalent transformations of the network $N_{u,v}^\circ$. Finally we show that classical Darboux
transformations are also related to cluster algebra transformations via a formula similar 
to~\eqref{gbdasclust}.

\section{Preliminaries} 

In this section we collect  the necessary background information on double Bruhat cells, Toda flows
and directed networks on surfaces. Though notions and results that we will need on the first two subjects 
can be as easily stated for an arbitrary semisimple group, we will limit ourselves to the $GL_n$ case.

\subsection{Double Bruhat cells}
 Let $\b_+$, $\n_+$, $\b_-$, $\n_-$ be, resp.,
algebras of upper triangular, strictly upper triangular, lower triangular and strictly lower
triangular matrices.

The connected subgroups that correspond to 
$\b_+$, $\b_-$, $\n_+$, $\n_-$ will be
denoted by  $\B_+$, $\B_-$, $\N_+$, $\N_-$. We denote by $\HH$ the maximal torus (the subgroup of diagonal matrices) in $GL_n$.

Every $\xi \in gl_n$ can be uniquely decomposed into
\begin{equation*}
\xi = \xi_- + \xi_0 + \xi_+,
\label{decomposition_algebra}
\end{equation*}
where $\xi_+ \in \n_+$, $\xi_- \in \n_-$ and $\xi_0$ is diagonal.
Consequently,  for every $X$ in an open Zariski dense subset 
of $GL_n$ there exists a unique {\em Gauss factorization}
$$
X = X_- X_0 X_+, \quad X_+ \in \N_+, \ X_- \in \N_-,  \ X_0 \in \HH.
$$

Let $s_i$, $i\in [1,n-1]$, denote the elementary transposition $(i, i+1)$ in the symmetric group $S_n$. 
A {\em reduced decomposition\/} of an element $w\in S_n$ is a
representation of $w$ as a product $w=s_{i_1} \cdots s_{i_l}$ of the
smallest possible length. A reduced decomposition is not unique,
but the number $l$ depends only on $w$ and is called the {\em
length\/} of $w$ and denoted by $l(w)$. The sequence of indices $\ii=(i_1,\ldots,i_l)$ that
corresponds to a given reduced decomposition of $w$ is called a
{\em reduced word\/} for $w$. The unique element of $S_n$ of maximal
length (also called {\em the longest element\/} of $S_n$) is
denoted by $w_0$. 

We will also need need a notion of a reduced
word for an ordered pair $(u,v)$ of elements in $S_n$. It is defined
as follows: if $(i_1,\ldots,i_{l(u)})$ is a reduced word for $u$ and
$(i'_1,\ldots,i'_{l(v)})$ is a reduced word for $v$, then any shuffle
of the sequences $(i_1,\ldots,i_{l(u)})$  and $(-i'_1,\ldots,-i'_{l(v)})$ is
called a reduced word for $(u,v)$.

Let us fix an embedding of $S_n$ into $GL_n$ and denote the
representative of $w\in S_n$ in $GL_n$ by the same letter $w$. The
{\em Bruhat decompositions\/} of $GL_n$ with respect to $\B_+$ and
$\B_-$ are defined, resp., by
\[
{\displaystyle GL_n = \cup_{u\in S_n}\B_+ u \B_+, \qquad GL_n =
\cup_{v\in S_n} \B_-  v \B_-. }
\]
For any $u,v \in S_n$, the {\em double Bruhat cell\/} is
defined as
\[
G^{u,v} = \B_+ u \B_+ \cap \B_- v \B_-.
\]

According to \cite{FZ1},
the variety $G^{u, v}$ is biregularly isomorphic to a Zariski open
subset of $\C^{l(u)+l(v)+n}$. A corresponding birational map
from  $\C^{l(u)+l(v)+n}$ to  $G^{u, v}$
can be constructed quite
explicitly, though not in a unique way. Namely, fix a  reduced word  $\ii$
for the pair $(u,v)$ and consider, in addition, a sequence
$\kk=(k_1,\ldots,k_n)$ obtained as an arbitrary re-arrangement of numbers  $\mathfrak i$ through $\mathfrak i n$.
 Let $\mathbf j=(j_1,\dots,j_{l(u)+l(v)+n})$ be a shuffle of $\kk$  and $\ii$;
we set
$\theta(j_l)$ to $+$ if $j_l>0$, to $-$  if $j_l< 0$, and to $0$
if $j_l \in \kk$. 

Denote by $e_{ij}$ an elementary $n\times n$ matrix
$\left ( \delta_{i\alpha} \delta_{j\beta}\right )_{\alpha,\beta=1}^n$. 
For $t\in \C$, $i,j\in [1,n-1]$, $k\in [1,n]$, let 
\be
\label{elembi}
E^-_i(t) =\one + t e_{i+1, i},\quad 
E^+_j(t) =\one + t e_{j, j+1},\quad 
E^0_k(t) = \one + (t-1) e_{kk}.
\ee
Then the map
$X_{\jj}\ : \C^{l(u)+l(v)+n}\to G^{u, v}$ can be defined  by
\begin{equation}
X_{\jj}(\mathbf t)= 
\prod_{q=1}^{l(u)+l(v)+n}
E^{\theta(j_q)}_{|j_q|}(t_q) .
\label{factorGuv}
\end{equation}
Parameters $ t_{1},\ldots,t_{l(u)+l(v)+n}$ constituting $\mathbf t$ are called {\it
factorization parameters}. 
Explicit formulae for the inverse of the map (\ref{factorGuv}) in terms of the so-called 
{\it twisted generalized minors\/} were found in \cite{FZ1}. 

\subsection{Toda flows}  Next, we review the basic facts about the Toda flows on $GL_n$. 
 
 Recall that the standard Poisson--Lie structure on $GL_n$ is given by
\begin{equation*}
\{ f_1, f_2\}_{SL_n}(X) = \frac{1}{2} ( R(\nabla f_1(X)\ X), 
\nabla f_2(X)\ X ) -  \frac{1}{2} ( R(X\ \nabla f_1(X)), X\ \nabla f_2(X) ),
\end{equation*}
where $(\  ,\ )$ denotes the {\em  trace-form}, $\nabla$ is the gradient defined with respect to  the trace-form, 
and  $R: gl_n\to gl_n$ is the standard R-matrix given by
$$
R(\xi)=\xi_+ - \xi_- = \left ( \sign (j-i) \xi_{ij}\right )_{i,j=1}^n.
$$

Double Bruhat cells are regular Poisson submanifolds of $GL_n$ equipped with the standard Poisson--Lie structure
(see \cite{R, KZ, Y}). Furthermore, 

(i) any symplectic leaf of $GL_n$ is of the form $S^{u,v} a$, where $S^{u,v}\subset G^{u,v}$ is a certain distinguished
symplectic leaf and $a$ is an element of the Cartan subgroup, and 

(ii) the dimension of symplectic leaves in $G^{u,v}$ 
equals $l(u) + l(v) + \corank ( u v^{-1} - \Id)$, see \cite{R, KZ}.

Conjugation-invariant functions on $GL_n$ form a Poisson-commuting family (see, e.g.,  \cite{r-sts}). Any such function
$F$ generates a Hamiltonian flow described by {\em the Lax equation}
\be
\label{Lax}
dX/dt = \left [ X\ ,\  - \frac{1}{2} R \left ( X \nabla F(X) \right ) \right ].
\ee
The resulting family of equations is called {\em the hierarchy of Toda flows\/}
(in \cite{R}, the term {\em characteristic Hamiltonian systems\/} is used). 
If one chooses $F(X) = F_k(X)= \frac{1}{k} \tr(X^k)$, then equation (\ref{Lax}) becomes (\ref{Lax_intro}).  
Functions $F_1, \ldots, F_{n-1}$  form a maximal family of algebraically independent conjugation-invariant functions
on $GL_n$.

For an element $h\in GL_n$, denote by $C_h$ the action of $h$ on $GL_n$ by conjugation: $C_h (X) = h X h^{-1}$. 
For any smooth function $f$ on $GL_n$ we have 
$$
\nabla (f \circ C_h) = Ad_{h^{-1}} ( \nabla f ). 
$$
Furthermore, if $h$ belongs to $\HH$, then it is easy to see that
$$
R( Ad_{h^{-1}} ( \xi )) = Ad_{h^{-1}} ( R(\xi ))
$$ 
for any $\xi \in gl_n$. Together, these observations imply
that for any $h\in \HH$ and any pair of smooth functions $f_1, f_2$ on $GL_n$,
$$
\{ f_1\circ C_h, f_2\circ C_h\} = \{ f_1, f_2\}\circ C_h.
$$
In other words, the action of $\HH$ on $GL_n$ by conjugation is Poisson with respect to  the standard 
Poisson--Lie structure. Since the action preserves double Bruhat cells, the standard Poisson--Lie structure
induces a Poisson structure on $G^{u,v}/\HH$, and the Toda hierarchy induces the family of commuting 
Hamiltonian flows on  $G^{u,v}/\HH$.

\br{\rm (i) The Lax equation (\ref{Lax}) can be solved explicitly 
via {\em the factorization method\/} \cite{r-sts}, 
which we will not review here. 

(ii) Written in terms of matrix entries, equations (\ref{Lax}) have exactly the same form as equations 
of the Toda hierarchy on $gl_n$, where the relevant Poisson structure is the Lie--Poisson structure associated 
with the {R-matrix Lie bracket\/} 
$[\xi, \eta]_R = \frac{1}{2} \left ([ R(\xi), \eta ] + [\xi, R (\eta) ]\right )$.
In fact, viewed as equations on the algebra of $n\times n$ matrices, the Toda hierarchy becomes a family of 
bi-Hamiltonian flows with compatible linear and quadratic Poisson brackets given by, respectively, Lie--Poisson 
and the extension of the Poisson--Lie brackets. However, we will not need the linear Poisson structure in the current paper.
}
\er

\subsection{Cluster algebras and compatible Poisson brackets}
\label{CA&PB}

First, we recall the basics of cluster algebras of geometric type. The definition that we present
below is not the most general one, see, e.g.,
\cite{FZ2, CAIII} for a detailed exposition.
 
The {\em coefficient group\/} $\PP$ is a free multiplicative abelian
group of a finite rank $m$ with generators $g_1,\dots, g_m$.
An {\em ambient field\/}  is
the field $\FFF$ of rational functions in $n$ independent variables with
coefficients in the field of fractions of the integer group ring
$\Z\PP=\Z[g_1^{\pm1},\dots,g_m^{\pm1}]$ (here we write
$x^{\pm1}$ instead of $x,x^{-1}$). It is convenient to think of $\FFF$ as
of the field of rational functions in $n+m$ independent variables
with rational coefficients.

A {\em seed\/} (of {\em geometric type\/}) in $\FFF$ is a pair
$\Sigma=(\x,B)$,
where $\x=(x_1,\dots,x_{n+m})$, $x_1,\dots,x_n$ is a transcendence basis of $\FFF$ over the field of
fractions of $\Z\PP$, $x_{n+i}=g_i$ for $i\in [1,m]$, 
and $B$ is an $n\times(n+m)$ integer matrix
whose principal part (that is, the $n\times n$ submatrix formed by the
columns $1,\dots,n$) is skew-symmetric.
The $(n+m)$-tuple  $\x$ is called a {\em cluster\/}, its elements
$x_1,\dots,x_n$ are called {\em cluster variables\/}, and $x_{n+1},\dots,x_{n+m}$ are {\em stable
variables\/}. 

Given a seed as above, the 
{\em cluster transformation\/} in direction $k\in [1,n]$
is defined by
$$
\x\mapsto\x_k=(\x\setminus\{x_k\})\cup\{\bar x_k\},
$$
where the new cluster variable $\bar x_k$ is given by the {\em exchange relation}
\begin{equation}\label{exchange}
x_k\bar x_k=\prod_{\substack{1\le i\le n+m\\  b_{ki}>0}}x_i^{b_{ki}}+
       \prod_{\substack{1\le i\le n+m\\  b_{ki}<0}}x_i^{-b_{ki}};
\end{equation}
here, as usual, the product over the empty set is assumed to be
equal to~$1$.

We say that $\bar B$ is obtained from $B$ by a {\em matrix mutation\/} in direction $k$ if
\[
\bar b_{ij}=\begin{cases}
         -b_{ij}, & \text{if $i=k$ or $j=k$;}\\
                 b_{ij}+\displaystyle\frac{|b_{ik}|b_{kj}+b_{ik}|b_{kj}|}2,
                                                  &\text{otherwise.}
        \end{cases}
\]

Given a seed $\Sigma=(\x,B)$, we say that a seed
$\bar\Sigma=(\bar\x,\bar B)$ is {\em adjacent\/} to $\Sigma$ (in direction
$k$) if $\bar\x$ is obtained from $\x$ and $\bar B$ is obtained from $B$
by a cluster transformation and a  matrix mutation, respectively, 
in direction $k$.
 Two seeds are {\em mutation equivalent\/} if they can
be connected by a sequence of pairwise adjacent seeds.
The {\em cluster
algebra\/} (of {\em geometric type\/}) $\A=\A(B)$
associated with $\Sigma$ is the $\Z\PP$-subalgebra of $\FFF$
generated by all cluster variables in all seeds mutation
equivalent to $\Sigma$. The {\it complex form\/} of $\A$ is defined as 
$\A$ tensored by $\C$ and is denoted $\A_\C$.

Let $V$ be a Zariski open subset in $\C^{n+m}$, $\A_\C$ be the complex form of a cluster algebra of geometric type. 
We assume that the variables in some extended cluster are identified with
a set of algebraically independent rational functions on $V$. This allows us to identify cluster variables in any cluster with rational functions on $V$ as well, and thus to consider $\A_\C$ as a subalgebra of the field $\C(V)$ of rational functions on $V$. Finally, we denote by $\A^V_\C$ the localization of $\A_C$ with respect to the stable variables that do not vanish on $V$.

\begin{proposition} \label{lostprop}
Let $V$ and $\A$ as above satisfy the following conditions:

{\rm (i)} each regular function on $V$ belongs to $\A^V_\C$;

{\rm (ii)} there exists a cluster $\x=(x_1,\dots,x_{n+m})$ in $\A_\C$ consisting of algebraically independent functions regular on $V$;

{\rm (iii)} any cluster variable $\bar x_k$, $k\in [1,n]$, obtained by the cluster transformation {\rm \eqref{exchange}} applied to $\x$ is regular on $V$.

Then $\A^V_\C$ is isomorphic to the ring $\O(V)$ of regular functions on $V$.
\end{proposition}

\proof All we have to prove is that any element in $\A^V_\C$ is a regular function on $V$. 
The proof follows the proof of a similar statement for double Bruhat cells in \cite{Z} and  consists of three steps.

\begin{lemma}\label{zl33}
Let $\z=(z_1,\dots,z_{n+m})$ be an arbitrary cluster in $\A_\C$. If
a Laurent monomial $M=z_1^{d_1}\cdots z_{n+m}^{d_{n+m}}$ is regular on $V$ then $d_i\ge0$ for $i\in [1,n]$. 
\end{lemma}

\proof Indeed, assume that $d_k<0$ for some $k\in [1,n]$ and consider the cluster $\z_k$. By~\eqref{exchange}, $M$ can be rewritten as 
$M=M_1\bar z_k^{-d_k}/P^{-d_k}$, where $M_1$ is a Laurent monomial in common variables of $\z$ and $\z_k$, and $P$ is the binomial 
(in the same variables) that appears in the right hand side of~\eqref{exchange}. By condition (i) and the Laurent phenomenon
(Theorem~3.1 in \cite{FZ2}), $M$ can be written as a Laurent polynomial
in the variables of $\z_k$. Equating two expressions for $M$, we see that $P^{-d_k}$ times a polynomial in variables of $\z_k$  equals a Laurent monomial in the same variables. This contradicts the algebraic independence of variables in $\z_k$, which follows from the algebraic independence of variables in $\x$.
\endproof

\begin{lemma}\label{zl34}
Let $z$ be a cluster variable in an arbitrary cluster $\z$, and assume that $z$ is a regular function on $V$. Then $z$ is  
irreducible in the ring of regular functions on $V$.
\end{lemma}

\proof
Without loss of generality, assume that $z_{n+1}=x_{n+1},\dots, z_{n+m'}=x_{n+m'}$ do not vanish on $V$, and $z_{n+m'+1}=x_{n+m'+1},\dots,
z_{n+m}=x_{n+m}$ may vanish on $V$. Moreover, assume to the contrary that $z=fg$, where $f$ and $g$ are non-invertible regular 
functions on $V$. By condition (i) and Proposition~11.2 of \cite{CAII}, both
$f$ and $g$ are Laurent polynomials in $z_1,\dots,z_{n+m'}$ whose coefficients are polynomials in $z_{n+m'+1},\dots,z_{n+m}$. Applying the same argument as in the proof of Lemma~\ref{zl33}, we see that both $f$ and $g$ are, in fact, Laurent monomials in $z_1,\dots,z_{n+m}$ and that 
$z_{n+1},\dots,z_{n+m'}$ enter both $f$ and $g$ with a non-negative degree. Moreover, by Lemma~\ref{zl33}, each cluster variable $z_1,\dots,z_n$ enters both $f$ and $g$ with a non-negative degree. This can only happen if one of $f$ and $g$ is invertible in 
$\O(V)$, a contradiction.
\endproof

Denote by $U_0\subset V$ the locus of all $t\in V$ such that $x_i(t)\ne 0$ for all $i\in [1,n]$. Besides, denote by $U_k\subset V$ 
the locus of all $t\in V$ such that $x_i(t)\ne 0$ for all $i\in [1,n]\setminus k$ and $\bar x_k(t)\ne0$. 

\begin{lemma}\label{zl36}
Let $U=\cup_{i=0}^n U_i$, then $\codim V\setminus U\ge 2$.
\end{lemma}

\proof
Follows immediately from Lemma~\ref{zl34} and conditions (ii) and (iii).
\endproof

Assume that there exists $f\in\A^V_\C$ that is not regular on $V$. Recall that $V$ is 
the complement of a finite union of irreducible hypersurfaces $D_i$ in $\C^{n+m}$. Therefore, the divisor of poles of $f$ has
codimension~$1$ in $\C^{n+m}$. Since $f$ is not regular on $V$, this latter divisor does not lie entirely in the union of $D_i$,
and hence its intersection with $V$ has codimension~$1$ in $V$. Therefore, by Lemma~\ref{zl36}, it intersects $U$ nontrivially.
To complete the proof, note that by Proposition~11.2 of \cite{CAII}, any function in $\A^V_\C$ is regular on $U$, a contradiction.
\endproof

Let $\Poi$ be a Poisson bracket on the ambient field $\FFF$. 
We say that it is {\em compatible} with the cluster algebra $\A(B)$ if, for any 
cluster $\x=(x_1,\dots,x_{n+m})$,  one has
$\{x_i,x_j\}=\omega_{ij} x_ix_j,$
where $\omega_{ij}\in\Z$ are
constants for all $i,j\in[1,n+m]$. The matrix
$\Omega^{\x}=(\omega_{ij})$ is called the {\it coefficient matrix\/}
of $\Poi$ (in the basis $\x$); clearly, $\Omega^{ \x}$ is
skew-symmetric. A complete description of Poisson brackets compatible with $\A(B)$
in the case $\rank B=n$ is given in \cite{GSV1}.

\subsection{Networks on surfaces with boundaries} 
\label{networks}
Let $ S$ be a disk with $c\ge 0$ holes,
so that its boundary $\partial S$ has $c+1$ connected components, and let 
$G=(V,E)$ be a directed graph embedded in $ S$ 
with the vertex set $V$ and the edge set $E$. 
Exactly $r$ of its vertices are located on the boundary $\partial S$.
They are denoted $b_1,\dots,b_r$ and
called {\it boundary vertices\/}. Each boundary vertex is labeled
as a source or a sink. A {\it source\/} is
a vertex with exactly one outcoming edge and no incoming edges.
{\it Sinks\/} are defined in the same way, with the direction of the single edge
reversed. The number of sources is denoted by $n$ and the number of sinks by $m=r-n$.
All the internal vertices of $G$ have degree~$3$ and are of two types: either they have exactly one
incoming edge, or exactly one outcoming edge. The vertices of the first type are called
(and shown on figures) {\it white}, those of the second type, {\it black}. 

A pair $(v,e)$, $v\in V$, $e\in E$, is called a {\it flag\/} if $v$ is an endpoint of $e$. To each flag 
$(v,e)$ we assign an independent variable $x_{v,e}$. 
Let $u$ and $v$ be two endpoints of $e$. The {\it edge weight\/} $w_{e}$ is defined by $w_{e}=x_{v,e}x_{u,e}$.
A {\it perfect network\/} $N=(G,w,\rho_{1},\dots,\rho_{c})$ is obtained from $G$ weighted as above 
by adding $c$ nonintersecting oriented curves $\rho_{i}$ (called {\it cuts\/}) in such a way that cutting $ S$ along $\rho_{i}$
makes it into a disk (note that the endpoints of each cut belong to distinct connected components of $\partial S$). 
The points of the {\it space of edge weights\/} $\EE_N=(\R\setminus 0)^{|E|}$ (or $(\C\setminus 0)^{|E|}$) can be 
considered as copies of the  
 graph $G$  with edges weighted by nonzero numbers obtained by specializing the variables $x_{v,e}$ to nonzero values.

Assign an independent variable $\lambda_{i}$ to each cut $\rho_{i}$. The weight of a path $P$ between two boundary 
vertices is defined as the product of the weights of all edges constituting the path times a Laurent monomial in $\lambda_{i}$. 
Each intersection point of $P$ with $\rho_{i}$ contributes to this monomial $\lambda_{i}$ if the oriented tangents to $P$ and $\rho_i$
at this point form a positively oriented basis, and $\lambda_i^{-1}$ otherwise (assuming that all intersection points are transversal).
Besides, the sign of the monomial is defined via the rotation number
of a certain closed curve built from $P$ itself, cuts and arcs of $\partial S$. For a detailed description of the corresponding
constructions, see \cite{GSV3, Postnikov} in the case $c=0$ (networks in a disk, no cuts needed, the path weight is a signed 
product of the edge weights) and \cite{GSV4} in the case $c=1$ (networks in an annulus, one cut $\rho$ and one additional independent
variable $\lambda$ involved, the path weight is a signed product of the edge weights times an integer power of $\lambda$).
The {\it boundary measurement\/} between a source $b_{i}$ and a sink $b_{j}$ is then defined as the sum of path weights over all 
(not necessary simple) paths from $b_{i}$ to $b_{j}$. It is proved in the above cited papers that a boundary measurement is a rational
function in the weights of edges (in case of the disk) or in the weights of edges and $\lambda$ (in case of the annulus).

Boundary measurements are organized in the {\it boundary measurement matrix}, thus giving rise to the {\it boundary measurement map\/}
from $\EE_N$ to the space of $n\times m$ matrices (for $c=0$), or the space of $n\times m$ rational matrix 
functions (for $c=1$). The gauge group acts on $\EE_N$ as follows: for any internal vertex $v$ of $N$ and any Laurent monomial $L$ 
in the weights $w_e$ of $N$, the weights of all edges leaving $v$ are multiplied by $L$, and the weights of all edges entering $v$ are
 multiplied by $L^{-1}$. Clearly, the weights of paths between boundary vertices, and hence boundary measurements, 
 are preserved under this action.
Therefore, the boundary measurement map can be factorized through the space $\FF_N$ defined as the 
quotient of $\EE_N$ by the action of the gauge group. In \cite{GSV4} we explained that $\FF_N$ can be identified 
with the relative cohomology group $H^1(G,G\cap\partial S)$ with
coefficients in the multiplicative group of nonzero real numbers. This gives rise to the representation
$$
\FF_N=H^{1}(G\cup\partial S)/H^{1}(\partial S)\oplus 
H^{0}(\partial S)/H^{0}(G\cup\partial S)=\FF_{N}^{f}\oplus\FF_{N}^{t}.
$$

The space $\FF_N^f$ can be described as follows. The graph $G$ divides $ S$ into a finite number of connected components called
\emph{faces}. The boundary of each face consists of edges of $G$ and, possibly, of several arcs of 
$\partial S$. A face is called {\it bounded\/} if its boundary contains only edges of $G$ and {\it unbounded\/} otherwise. 
Given a face $f$, we define its {\it face weight\/} $y_f$ as the function on $\EE_N$ that assigns to the 
edge weights $w_e$, $e\in E$, the value
\[                      
y_f=\prod_{e\in\partial f}w_e^{\gamma_e},
\]                       
where $\gamma_e=1$ if the direction of $e$ is compatible with the counterclockwise orientation of the 
boundary $\partial f$ and $\gamma_e=-1$ otherwise. Face weights are invariant under the gauge group action, 
and hence are functions on $\FF_N^f$, and, moreover, form a basis in the space of such functions.

In \cite{GSV3, GSV4} we have studied the ways to turn the space of edge weights into a Poisson manifold 
by considering Poisson brackets on the 
space of flag variables satisfying certain natural conditions. We proved that all such Poisson brackets on $\EE_N$ form a 6-parameter
family, and that this family gives rise to a 2-parameter family of Poisson brackets on $\FF_N$. In what follows we are interested in
a specific member of the latter family (obtained by setting $\alpha=1/2$ and $\beta=-1/2$ in the notation of 
\cite{GSV3, GSV4}). For the reasons that will be explained later, we call this bracket the {\it standard\/} 
Poisson bracket on $\FF_N$. The corresponding 4-parameter family of Poisson brackets on $\EE_N$ is called standard as well.

Given a perfect network $N$ as above,
define the {\it directed dual network\/} $N^*=(G^*,w^*)$ as follows. Vertices of $G^*$ are the faces of $N$. 
Edges of $G^*$ correspond to the edges of $N$ that connect either two internal vertices of different colors, 
or an internal vertex with a boundary vertex; note that there might be several edges between the same pair of 
vertices in $G^*$. An edge $e^*$ of $G^*$ corresponding to $e$ is directed in such a way that 
the white endpoint of $e$ (if it exists) lies to the left of $e^*$ and 
the black endpoint of $e$ (if it exists) lies to the right of $e$. 
The weight $w^*(e^*)$ equals $1$ if both endpoints of $e$ are internal vertices, and $1/2$ if one of the 
endpoints of $e$ is a boundary vertex. 

\begin{proposition}\label{PSviay}
The restriction of the standard Poisson bracket on $\FF_N$ to the space $\FF_N^f$ is given by
\[
\{y_f,y_{f'}\}=\left(\sum_{e^*: f\to f'} w^*(e^*)-
\sum_{e^*: f'\to f} w^*(e^*)\right)y_fy_{f'}.
\]
\end{proposition}

For networks in a disk, the above proposition is a special case of Lemma~5.3 of \cite{GSV3}. 
For other surfaces the proof is literally the same.

In what follows, we will deal with networks of two kinds: acyclic networks in a disk with the same number of nonalternating
sources and sinks, and networks in an annulus obtained from the networks of the first kind by a certain construction, 
to be described below. 

In the former case we assume that $n$ sources are numbered clockwise and are followed by $n$ sinks numbered counterclockwise.
The weight of a path in this case is exactly the product of edge weights involved. The boundary measurements are organized into a
$n\times n$ matrix $X$ in such a way that $X_{ij}$ is the boundary measurement between the $i$th source and $j$th sink.
One can concatenate two networks of this kind 
by gluing the sinks of the former to the sources of the latter.
If $X_1$,  $X_2$ are  the
matrices associated with the two networks, then the matrix associated with their concatenation
is $X_1 X_2$. This fact can be used to visualize parametrization (\ref{factorGuv}).
Indeed, an $n\times n$ diagonal  matrix $\diag(d_1,\ldots,d_n)$ and elementary
bidiagonal matrices $E^-_i(l)$ and  $E^+_j(u)$ defined by~\eqref{elembi}
correspond to building blocks shown on Figure~\ref{fig:factor} a, b and c, respectively; 
all weights not shown explicitly are equal to~1. Note that building blocks themselves are not
networks, since their edge weights do not comply with the rules introduced above. 
However, as we will see below, objects glued from building blocks comply with all the 
rules.

\begin{figure}[ht]
\begin{center}
\includegraphics[height=3.0cm]{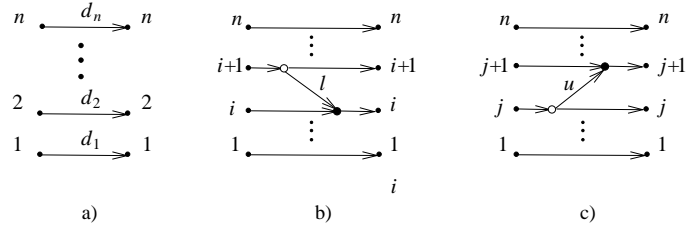}
\caption{Three building blocks used in matrix factorization}
\label{fig:factor}
\end{center}
\end{figure}

The concatenation of  $n(n-1)$ building blocks of the second and the third types and one building 
block of the first type, in an appropriately chosen order and with each building block having its own 
nontrivial weights, describes a generic element of $GL_n$ (see, e.g. \cite{Fallat}). The structure of the obtained
network is given by Figure~\ref{fig:genfactor}. Here and in what follows we use the gauge group action
to decrease the number of parameters of networks in question. In particular, 
this network has $2n(n-1)$ internal vertices, therefore, 
one can use the gauge group action to change the weights of $2n(n-1)$ edges to $1$. It is convenient to choose 
these edges to be all the horizontal edges except for one middle edge in each horizontal chain. The weights on the remaining edges are Laurent monomials
in the initial weights of the network. For example, if the endpoints of an edge $e$ belong to levels $i$ and $i+1$, then 
$u_e=w_ew_{P_{i+1}}/w_{P_i}$, where $P_i$ and $P_{i+1}$ are the horizontal paths from the endpoints of $e$ to the sinks 
$i$ and $i+1$, respectively.  

\begin{figure}[ht]
\begin{center}
\includegraphics[height=3.2cm]{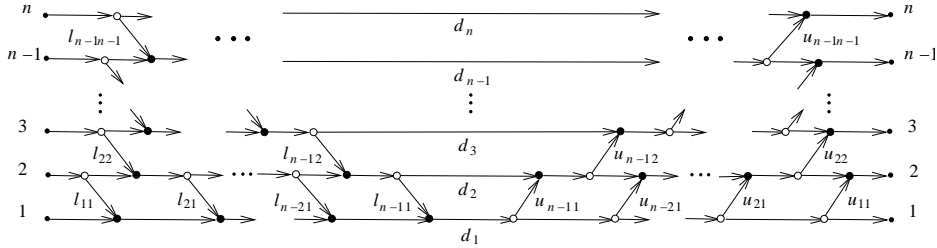}
\caption{Generic planar network; the weights of edges are Laurent monomials of the initial weights}
\label{fig:genfactor}
\end{center}
\end{figure}

The following result, which is a special case of Theorem~4.1 from \cite{GSV3}, explains why we call 
the bracket in consideration standard.

\begin{theorem} \label{PSGL}
For any network $N$ as above with $n$ sources and $n$ sinks the map from $\EE_{N}$ to the space of $n\times n$ matrices
given by the boundary measurement matrix is Poisson with respect to any standard Poisson bracket on $\EE_N$ and the standard
Sklyanin bracket on $GL_n$.
\end{theorem}  

\br {\rm Note that the definition of the R-matrix $R_{\alpha,\beta}$ in \cite{GSV3} contains a superfluous factor $1/2$.}
\er

Networks in an annulus that we study in this paper are obtained from the above described networks in a disk by 
a gluing procedure described in detail in Section 5. These networks have one source and one sink, both lying 
on the same connected component of the boundary. The other connected component of the boundary does not carry 
boundary vertices, and hence for our networks $H^0(\partial S)=H^0(G\cup\partial S)$, which implies
$\FF_N=\FF_N^f$. Therefore, the standard Poisson bracket on 
$\FF_N$ is completely described by Proposition~\ref{PSviay}.

\section{Coxeter double Bruhat cells}

 We start this section with describing a particular instance of the Berenstein-Fomin-Zelevinsky 
parametrization \cite{BFZ, FZ1} in the case of  Coxeter double Bruhat cells in $GL_n$.

Denote $s_{[p,q]}=s_ps_{p+1}\cdots s_{q-1}$
for $1\le p< q \le n$ and recall that every Coxeter element $v\in S_n$ can be written in the form
\be
\label{factoru}
v=   
s_{[i_{k-1} , i_k]}\cdots s_{[i_{1} , i_2]}s_{[1 , i_1]}
\ee
for some subset $I=\{1=i_0 < i_1 < ...< i_k=n\}\subseteq[1,n]$. Besides, 
define $L=\{1=l_0 < l_1 < ...< l_{n-k}=n\}$ by $\{ l_1 < \ldots < l_{n-k-1}\} = [1,n] \setminus I$.

\bl\label{u_inv}
Let $v$ be given by~{\rm \eqref{factoru}}, then
\[
v^{-1} =  
s_{[l_{n-k-1} , l_{n-k}]}\cdots s_{[l_{1} , l_2]}s_{[1 , l_1]}.
\]
\el

\begin{proof} We use induction on $n$.  Denote the right-hand side of the above relation by $\bar v$. The 
index $n-1$  belongs either to $I$ or to $L$.
In the latter case $l_{n-k-1} = n-1$, and we have 
$v=s_{[i_{k-1} , n]}\cdots s_{[i_{1} , i_2]}s_{[1 , i_1]}\ =
s_{[i_{k-1} , n-1]}\cdots s_{[i_{1} , i_2]}s_{[1 , i_1]}\ s_{n-1}$ and
$\bar v= s_{n-1} s_{[l_{n-k-2} , l_{n-k-1} ]}\cdots s_{[l_{1} , l_2]}s_{[1 , l_1]}$. Then
$v=v's_{n-1}$ and $\bar v=s_{n-1}\bar v'$, where $v'$, $\bar v'$ are Coxeter elements in $S_{n-1}$ corresponding to
index sets $I \setminus \{n\}\cup \{n-1\}$ and $L\setminus \{n\}$, and hence
$ v \bar v = v' \bar v'=1$ by the induction hypothesis.
Otherwise, if $n-1$  belongs  to $I$, we interchange the roles of $v$ and $\bar v$ and use the same argument.
\end{proof}

\bl
\label{u-matrix}
The permutation matrix that corresponds to a Coxeter element $v$ is 
\[
\tilde v = \sum_{j=1}^k e_{i_{j-1} i_j} + \sum_{j=1}^{n-k} e_{l_j l_{j-1}}.
\]
\el

\begin{proof} We use the same inductive argument as in the proof of Lemma \ref{u_inv}. Assuming
that $n-1 \in L$, the relation   $v=v's_{n-1}$ and the induction hypothesis imply
$\tilde v = ( e_{1 i_1} + \ldots +  e_{i_{k-1} n-1} + e_{l_1 1} + \ldots + e_{n-1 l_{n-k-2}} + e_{nn} ) 
( e_{11} + \ldots +  e_{n-2 n-2} + e_{n n-1} + e_{ n-1 n} ) =
e_{1 i_1} + \ldots +  e_{i_{k-1} n} + e_{l_1 1} + \ldots + e_{n-1 l_{n-k-2}} + e_{n n-1}$ as claimed.
\end{proof}

Let now $(u,v)$ be a pair of Coxeter elements and 
\be\label{IL}
 \begin{split}
 &I^+=\{1=i^+_0 < i^+_1 < ...< i^+_{k^+}=n\}, \\
 &I^-=\{1=i^-_0 < i^-_1 < ...< i^-_{k^-}=n\},\\
 &L^+=\{1=l^+_0<  l^+_1 < \ldots < l^+_{n-k^+-1}<  l^+_{n-k^+}= n\},\\
 &L^-=\{1=l^-_0<  l^-_1 < \ldots < l^-_{n-k^--1}<  l^-_{n-k^-}= n\}
 \end{split}
 \ee
  be subsets of $[1,n]$ that correspond to $v$ and $u^{-1}$ in the way just described. 
For a set of complex parameters $c_1^-,\ldots, c_{n-1}^-; c_1^+,\ldots, c_{n-1}^+; d_1, \ldots, d_n$,
define matrices $D=\diag (d_1,\ldots, d_n)$, 
\begin{equation}\label{Cj}
C^+_j= \sum_{\alpha=i^+_{j-1}}^{i^+_j-1} c^+_\alpha e_{\alpha,\alpha+1},\ j\in [1, k^+],\quad 
C^-_j= \sum_{\alpha=i^-_{j-1}}^{i^-_j-1} c^-_\alpha e_{\alpha+1,\alpha},\ j\in [1,  k^-],
\end{equation}
and
$$
\bar C^+_j= \sum_{\alpha=l^+_{j-1}}^{l^+_j-1} c^+_\alpha e_{\alpha,\alpha+1},\ j\in [1, n-k^+],\quad 
\bar C^-_j= \sum_{\alpha=l^-_{j-1}}^{l^-_j-1} c^+_\alpha e_{\alpha+1,\alpha},\ j\in 1, n-k^-].
$$

\bl
\label{X-Xinv}
A generic  element $X\in G^{u,v}$ can be written as
\begin{equation}
X=  (\one -C^-_1)^{-1} \cdots  (\one - C^-_{k^-})^{-1} D (\one - C^+_{k^+})^{-1}\cdots (\one -
C^+_1)^{-1}, 
\label{factorI}
\end{equation}
and its inverse can be factored as
\begin{equation}
X^{-1}=  
(\one + \bar C^+_{n-k^+})^{-1}\cdots (\one +
\bar C^+_1)^{-1} D^{-1} (\one +
\bar C^-_1)^{-1} \cdots  (\one + \bar C^-_{k^-})^{-1}.
\label{factorinv}
\end{equation}
\el

\begin{proof}
It is easy to see that
\be\label{Xintoel}
\begin{split} 
(\one - C^+_{j})^{-1} &= E^+_{i^+_{j-1}}(c^+_{i^+_{j-1}})\cdots  E^+_{i^+_{j}-1}(c^+_{i^+_{j}-1}),\\   
(\one - C^-_{j})^{-1} &= E^-_{i^-_{j}-1}(c^-_{i^-_{j}-1})\cdots  E^-_{i^-_{j-1}}(c^-_{i^-_{j-1}}),\\
(\one + \bar C^+_{j})^{-1} &= E^+_{l^+_{j-1}}(-c^+_{l^+_{j-1}})\cdots  E^+_{l^+_{j}-1}(-c^+_{l^+_{j}-1}),\\  
(\one + \bar C^-_{j})^{-1} &= E^-_{l^-_{j}-1}(-c^-_{l^-_{j}-1})\cdots  E^-_{l^-_{j-1}}(-c^-_{l^-_{j-1}}).
\end{split}
\ee
 Then, by (\ref{factorGuv}) and (\ref{factoru}), a generic  $X\in G^{u,v}$
can be written as in (\ref{factorI}).
Next, the same reasoning as in the proof of Lemma \ref{u_inv} implies that
\begin{align*}
(\one -C^+_1) \cdots  (\one - C^+_{k^+})&= 
\left(\left(\prod_{s=i^+_{k-1}}^{i^+_k -1} E^+_s(c^+_s)\right)\cdots\left(\prod_{s=1}^{i^+_1 -1} E^+_s(c^+_s)\right) \right )^{-1}\\ 
&= \left (\prod_{s=l^+_{n-k-1}}^{l^+_{n-k} -1} E^+_s(-c^+_s)\right ) \cdots \left (\prod_{s=1}^{l^+_1 -1} E^+_s(-c^+_s)\right)\\
  &= (\one + \bar C^+_{n-k^+})^{-1}\cdots (\one +\bar C^+_1)^{-1}
\end{align*}
and, similarly,
$$
(\one - C^-_{k^-})\cdots (\one -
C^-_1) = (\one +
\bar C^-_1)^{-1} \cdots  (\one + \bar C^-_{k^-})^{-1}.
$$
Therefore,
\begin{align*}
X^{-1} &=  (\one -
C^+_1) \cdots  (\one - C^+_{k^+}) D^{-1} (\one - C^-_{k^-})\cdots (\one -C^-_1)\\
&= (\one + \bar C^+_{n-k^+})^{-1}\cdots (\one +\bar C^+_1)^{-1} D^{-1} (\one +
\bar C^-_1)^{-1} \cdots  (\one + \bar C^-_{k^-})^{-1}.
\end{align*}
\end{proof}

The network $N_{u,v}$ that corresponds to factorization (\ref{factorI}) is obtained by the concatenation (left to right)
of $2 n -1$ building blocks (as depicted in Fig.~\ref{fig:factor}) that correspond to elementary matrices  
\begin{align*}
&E^-_{i^-_{2}-1}(c^-_{i^-_{2}-1}), \ldots,  E^-_{1}(c^-_{1}), E^-_{i^-_{3}-1}(c^-_{i^-_{3}-1}),\ldots,  
E^-_{i^-_{2}}(c^-_{i^-_{2}}), \ldots, \\
& E^-_{n-1}(c^-_{n-1}), \ldots,   E^-_{i^-_{k^--1}}(c^-_{i^-_{k^--1}}), D,
E^+_{i^+_{k^+-1}}(c^+_{i^+_{k^+-1}}), \ldots,  E^+_{n-1}(c^+_{n-1}),\\
&\ldots, E^+_{i^+_{2}}(c^+_{i^+_{2}}),\ldots  E^+_{i^+_{3}-1}(c^+_{i^+_{3}-1}), 
E^+_{1}(c^+_{1})\cdots  E^+_{i^+_{2}-1}(c^+_{i^+_{2}-1}).
\end{align*}
This network has $4(n-1)$ internal vertices and $5n-4$ horizontal edges. Similarly to the case of
generic networks discussed in Section~\ref{networks}, one can use the gauge group action to change the
weights of all horizontal edges except for those belonging to block $D$ to~$1$.

One can use the network $N_{u,v}$ to derive expressions for factorization parameters $d_i$, $c^+_i$, $c^-_i$
in terms of matrix entries of $X$. These formulas are a simple particular case of the general formulas
by Berenstein-Fomin-Zelevinsky for restoring factorization parameters in double Bruhat cells \cite{BFZ, FZ1}.

For a matrix $A$, denote by $A_{r_1, \ldots, r_k}^{p_1, \ldots, p_k}$ its minor formed by rows numbered
$r_1<\cdots < r_k$ and columns $p_1 <\cdots < p_k $. 

\bl
\label{Xtocd}
For any $i\in [1,n]$,
\bea
\nonumber
d_i = \frac{X_{[1,i]}^{[1,i]} }{X_{[1,i-1]}^{[1,i-1]}};
\eea
for any $i\in [1,n-1]$,
\bea
\nonumber
c_i^- = \frac{X_{i^-_1, ..., i^-_{j}, i+1}^{i^-_0, ..., i^-_j} }{ X_{i^- 
_1, ..., i^-_j, i}^{i^-_0, ..., i^-_j}},\   i^-_j < i < i^-_{j+1}, & \qquad 
c_{i^-_j}^- = \frac{X_{i^-_1, ..., i^-_{j}, i^-_{j}+1}^{i^-_0, ..., i^-_j} X_{[1,i^-_j-1]}^{[1,i^-_j-1]}}{ X_{i^- 
_1, ..., i^-_j}^{i^-_0, ..., i^-_{j-1}} X_{[1,i^-_j]}^{[1,i^-_j]} }, 
\\ \nonumber
c_i^+ = \frac{X_{i^+_0, ..., i^+_j}^{i^+_1, ..., i^+_{j},i+1}}{ X_{i^ 
+_0, ..., i^+_j}^{i^+_1, ..., i^+_j,i}},\ i^+_j < i < i^+_{j+1}, & \qquad c_{i^+_j}^+ = \frac{X^{i^+_1, ..., i^+_{j}, i^+_{j}+1}_{i^+_0, ..., i^+_j} X_{[1,i^+_j-1]}^{[1,i^+_j-1]}}{ X^{i^+ 
_1, ..., i^+_j}_{i^+_0, ..., i^+_{j-1}} X_{[1,i^+_j]}^{[1,i^+_j]} }.
\eea
\el
\begin{proof} Since by (\ref{factorI}), $X_{[1,i]}^{[1,i]}=D_{[1,i]}^{[1,i]}=d_1\cdots d_i$, the first formula 
follows easily. Next, note that, for $i^-_j < i \leq i^-_{j+1}$, there is exactly one directed path in $N_{u,v}$ 
that joins the $i$th source with the $i_j^-$th sink. The weight of this path equals
$c^-_{i-1} \cdots c^-_{i_j} d_{i_j^-}$.
Thus,  there is a unique collection
of vertex-disjoint paths in $N_{u,v}$ joining sources $i^-_1,\ldots, i^-_{j}, i$ with sinks 
$i^-_0,\ldots, i^-_{j}$.
By Lindstr\"om's Lemma  \cite{KarlinMacGregor, Lindstrom}, this
implies that $ X^{i^-_1, ..., i^-_{j}, i}_{i^-_0, ..., i^-_j}$ is equal to the product of weights of these paths.
Clearly, for $ i^-_j <  i < i^-_{j+1}$,   
$ X^{i^-_1, ..., i^-_{j}, i+1}_{i^-_0, ..., i^-_j}= c^-_i  X^{i^-_1, ..., i^-_{j}, i}_{i^-_0, ..., i^-_j}$. Also,
$X^{i^-_1, ..., i^-_{j}, i^-_{j}+1 }_{i^-_0, ..., i^-_j}= c^-_{i^-_j} d^-_{i_j}  X^{i^-_1, ..., i^-_{j}}_{i^-_0, ..., i^-_{j-1}}$.
Formulae for $c^-_i$ follow from these relations. Formulas for $c^+_i$ are derived in a similar way.
\end{proof}


Let us now introduce some combinatorial data that will be useful in the following sections.

Let us fix a pair $(u,v)$ of Coxeter elements, and hence, fix the sets $I^\pm$ given by~\eqref{IL}.
For any $i\in [1,n]$ define integers $\varepsilon^\pm_i$ and $\zeta^\pm_i$
by setting
\begin{equation}
\varepsilon^\pm_i=\left \{ \begin{array}{cc} 0 & \mbox{if}\  i=i^\pm_j\
\mbox{for some}\  0 < j \leq k_\pm\\ 1 & \mbox{otherwise}\end{array}
\right.  
\label{eps}
\end{equation}
and
\begin{equation}
\zeta^\pm_i =i (1-\varepsilon^\pm_i) -\sum_{\beta=1}^{i -1} \varepsilon^\pm_\beta;
\label{nunu}
\end{equation}
note that by definition, $\varepsilon^\pm_1=1$, $\zeta^\pm_1=0$. (Here and in what follows a relation involving variables with superscripts $\pm$ is a shorthand for two similar relations: the one obtained by simultaneously replacing each $\pm$ by $+$, and the other, by $-$.) 
Further, put $M^\pm_i=\{\zeta^\pm_\alpha\ : \ \alpha=1,\ldots,i \}$ and 
\be
k^\pm_i=\max \{ j: i^\pm_j \leq i\}.                               
\label{k+/-}
\ee
Finally, put
\be
\varepsilon_i= \varepsilon_{i}^+ + \varepsilon_{i}^-               
\label{epsum}
\ee
and  
\be
 \varkappa_i= i+1 - \sum_{\beta=1}^i \varepsilon_\beta.            
 \label{kappa}
 \ee

\br\label{samenet}
{\rm It is easy to see that there exist distinct pairs $(u,v)$ and $(u',v')$ that produce the same 
$n$-tuple $\varepsilon$. The ambiguity occurs when $\varepsilon_i=1$ for some $i\in [2,n-1]$. 
By~\eqref{epsum}, this situation corresponds either to $\varepsilon_i^+=1$, $\varepsilon_i^-=0$, or to $\varepsilon_i^+=0$, $\varepsilon_i^-=1$. Consequently, the number of pairs $(u,v)$ with the identical $n$-tuple $\varepsilon$ equals $2$ power the number
of times $\varepsilon_i$ takes value $1$.
}
\er

\bl
\label{allcomb}
{\rm (i)} The $n$-tuples $\varepsilon^\pm=(\varepsilon^\pm_i)$ and $\zeta^\pm=(\zeta^\pm_i)$ uniquely
determine each other.

{\rm (ii)} For any $i\in [1,n]$,
 \begin{equation*}
\zeta^\pm_i = 
\left \{ \begin{array}{cc} j & \mbox{if}\  i=i_j^\pm\ \mbox{for
some}\  0< j \leq k_\pm\\ -\sum_{\beta=1}^{i -1} \varepsilon^\pm_\beta &
\mbox{otherwise}\end{array} \right ..
\end{equation*}

{\rm (iii)} For any $i\in [1,n]$,
 $$
 k^\pm_i=i-\sum_{\beta=1}^i \varepsilon^\pm_\beta.
$$

{\rm (iv)} For any $i\in [1,n]$,
\begin{equation*}
M^\pm_i= [ k^\pm_i -i+1, k^\pm_i ]= [1 - \sum_{\beta=1}^i
\varepsilon^\pm_\beta, i-\sum_{\beta=1}^i \varepsilon^\pm_\beta ].
\end{equation*}
\el

\begin{proof}
(i) Follows form the fact that the transformation $\varepsilon^\pm\mapsto \zeta^\pm$ defined by~\eqref{nunu} 
is given by a lower--triangular matrix with a non-zero diagonal.

(ii) By~\eqref{nunu}, the first equality is equivalent to
\be\label{twocount}
i^\pm_j=j+\sum_{\beta=1}^{i^\pm_j-1}\varepsilon^\pm_\beta.
\ee
By~\eqref{eps}, the latter can be interpreted as counting the first $i^\pm_j$ elements of $(\varepsilon^\pm_i)$: exactly $j$ of them 
are equal to $0$, and all the other are equal to $1$.

The second equality follows trivially from~\eqref{eps} and~\eqref{nunu}.   

(iii) For $i=i^\pm_j$, follows immediately from~\eqref{k+/-} and~\eqref{twocount}. For $i\ne i^\pm_j$, the same counting 
argument used in~\eqref{twocount} gives
$$
i=k^\pm_i+\sum_{\beta=1}^{i}\varepsilon^\pm_\beta.
$$

(iv) Follows from parts (ii) and (iii). 
\end{proof}

\br\label{cmvao}
{\rm 
(i) If $v=s_{n-1} \cdots s_1$, then $X$ is a lower Hessenberg matrix, and if 
$u=s_{1} \cdots s_{n-1}$, then $X$ is an upper Hessenberg matrix. 

(ii) If  
$v=s_{n-1} \cdots s_1$ {\em and} $u=s_{1} \cdots s_{n-1}$, then $G^{u,v}$ consists of
tridiagonal matrices with non-zero off-diagonal entries ({\em Jacobi matrices\/}).
In this case $I^+=I^-=[1,n]$, $\varepsilon_1^\pm=1$ and $\varepsilon_i^\pm=0$ for $i=2, \ldots, n$.

(iii) If
$u=v=s_{n-1} \cdots s_1$ (which leads to $I^+=[1,n], I^-=\{1,n\}$), then 
elements of $G^{u,v}$
have a structure of recursion operators arising in the theory of
orthogonal polynomials on the unit
circle (see, e.g. \cite{simon}).

(iv) The choice $u=v=( s_{1} s_3 \cdots ) (s_2 s_4 \cdots )$ (the so-called {\em
bipartite Coxeter element\/})
gives rise to a special kind of pentadiagonal matrices $X$ (called {\em CMV
matrices\/}), which serve
as an alternative version of recursion operators for orthogonal polynomials
on the unit
circle, see \cite{cmv, simon}.}
\er

\begin{examp}\label{runex}
{\rm
 Let $n=5$, $v =s_4 s_3 s_1 s_2$ and $u =s_3 s_2 s_1 s_4$. The network $N_{u,v}$ that corresponds
to factorization (\ref{factorI}) is shown in Figure~\ref{factorex}. 

\begin{figure}[ht]
\begin{center}
\includegraphics[height=4.0cm]{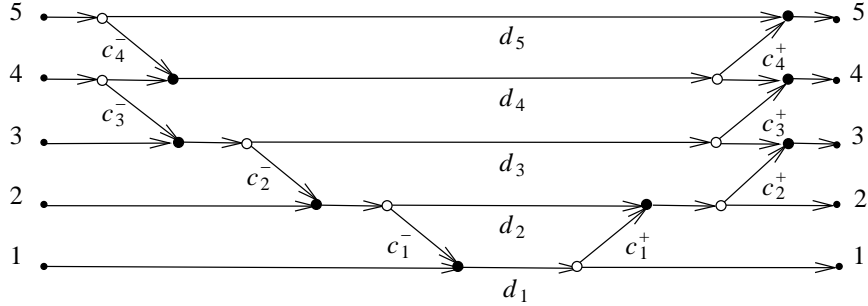}
\caption{Network representation for elements in $G^{s_3 s_2 s_1 s_4,s_4 s_3 s_1 s_2}$}
\label{factorex}
\end{center}
\end{figure}

A generic element $X\in G^{u,v}$ has a form
$$
X= (x_{ij})_{i,j=1}^5=\left (
\begin{array}{ccccc}
d_1 &  x_{11}c_1^+ & x_{12}c_2^+ & 0 & 0\\
c_1^-x_{11} & d_2+c_1^-x_{12} & x_{22}c_2^+ & 0 & 0\\
c_2^-x_{21} & c_2^-x_{22} & d_3+c_2^-x_{23} & d_3c_3^+ & 0\\
c_3^-x_{31} & c_3^-x_{32} & c_3^-x_{33} & d_4+c_3^-x_{34} & d_4c_4^+\\
0 & 0 & 0 & c_4^-d_4 & d_5+c_4^-x_{45}
\end{array}
\right ).
$$

One finds by a direct observation that $k^+=3$ and $I^+=\{i_0^+,i_1^+,i_2^+,i_3^+\}=\{1,3,4,5\}$, and hence
$L^+=\{l_0^+,l_1^+,l_2^+\}=\{1,2,5\}$. Next, $u^{-1}=s_4s_1s_2s_3$, therefore, $k^-=2$ and 
$I^-=\{i_0^-,i_1^-,i_2^-\}=\{1,4,5\}$, and hence $L^-=\{l_0^-,l_1^-,l_2^-,l_3^-\}=\{1,2,3,5\}$.
Further, 
$$
\varepsilon^+=(1,1,0,0,0), \qquad\varepsilon^-=(1,1,1,0,0),
$$
 and hence 
 $$
 \zeta^+=(0,-1,1,2,3),\qquad \zeta^-=(0,-1,-2,1,2).
 $$
Therefore, 
$$
(k^+_i)_{i=1}^5=(0,0,1,2,3), \qquad (k^-_i)_{i=1}^5=(0,0,0,1,2), 
$$
and hence 
\begin{align*}
(M^+_i)_{i=1}^5&=([0,0], [-1, 0], [-1,1], [-1,2], [-1,3]),\\ 
(M^-_i)_{i=1}^5&=([0,0], [-1, 0], [-2,0], [-2,1], [-2,2]).
\end{align*}
Finally, $\varepsilon=(2,2,1,0,0)$ and
$\varkappa=(0,-1,-1,0,1)$. By Remark~\ref{samenet}, there is one more pair of Coxeter elements such that produces the same $5$-tuples:
$v'=s_4s_1s_2s_3$ and $u'=s_2s_1s_3s_4$.

To illustrate Lemma~\ref{Xtocd}, we find $c_3^-$ and $c_4^-$. First, $1=i_0^-<3<i_1^-=4$, so
the corresponding formula in Lemma~\ref{Xtocd} gives $c_3^-=X^1_4/X^1_3$. Second, $4=i_1^-$, so the other formula in 
Lemma~\ref{Xtocd} gives $c_4^-=X^{1,4}_{4,5}X^{[1,3]}_{[1,3]}/X^1_4X^{[1,4]}_{[1.4]}$. It is easy to check that the 
right hand sides of both formulas indeed produce correct answers. 
We will use this example as our running example in the next section.
}
\end{examp}

We conclude this section with a proposition that explains how to recognize an
element of a Coxeter double Bruhat cell in $GL_n$. For any two subsets $R,P\subseteq [1,n]$ and a matrix $X\in GL_n$ 
we denote by $X(R,P)$ the submatrix of $X$ formed by rows $r\in R$ and columns $p\in P$.

\bp
An element $X\in GL_n$ belongs to a Coxeter double Bruhat cell if and only if the following conditions
hold for any $l\in [1,n-1]$:

 {\rm (i$_+$)} $\rank X([l+1,n],[1,l])=1$;

{\rm (i$_-$)} $\rank X([1,l],[l+1,n])=1$;

{\rm (ii$_+$)} $\rank X([1,n],[1,l])>1$ implies $X([l+1,n],[1,l-1])=0$;

{\rm (ii$_-$)} $\rank X([1,l],[1,n])>1$ implies $X([1,l-1],[l+1,n])=0$.
\label{rank-one}
\ep

\begin{proof} Let $X\in G^{u,v}$. Note that for any $p, r \in [1,n]$,  the rank of the submatrix 
$X([r,n],[1,p])$ does not change under right and left multiplication
of $X$ by elements of $\B_+$. Since $G^{u,v}\subset \B_+u\B_+$, this
means that we only need to check conditions (i$_+$) and (ii$_+$) for the permutation matrix $\tilde u$, 
for which it is clearly true in view of 
Lemma \ref{u-matrix}. Similarly, conditions (i$_-$) and (ii$_-$) reduce to considering $\tilde v$.

On the other hand, let $X$ satisfy condition (i$_+$) for any $l\in [1,n-1]$ and let $i_1^-$ be the largest
index such that $x_{i^-_1 1} \ne 0$. Condition (i$_+$) for $l=1$ implies $i_1^->1$. Further, condition (i$_+$) for $l=i_1^--1$
implies $X([i_1^-+1,n],[1,i_1^--1])=0$. Similarly, we can define
$i_2^-$ to be the largest index such that $x_{i_2^- i_1^-} \ne 0$ and conclude from condition (i$_+$) for 
$l=i_1^-$ that $i_2^- > i_1^-$,  and from condition (i$_+$) for $l=i_2^--1$ that $X([i_2^-+1,n],[i_1^-,i_2^--1])=0$.
Continuing in this manner, we construct a sequence
 $I_-=\{1=i^-_0 < i^-_1 < ...< i^-_{k^-}=n\}$ such that 
 \begin{equation}\label{lsp}
 x_{i^-_{s} i^-_{s-1}} \ne 0, \quad X([i_s^-+1,n],[i_{s-1}^-,i_s^--1])=0
 \end{equation}
 for $s \in [1, k^--1]$.  
 Multiplying $X$ on the right and on the left
 by appropriate elements of $\B_+$, we can reduce it to a matrix $X'=(x'_{ij})$ 
 satisfying~\eqref{lsp} and  such
 that  $x'_{i  i^-_{s-1}}= \delta_{i  i^-_{s}}$,  $x'_{i^-_{s} j}= \delta_{ i^-_{s-1} j }$. Moreover,  condition (ii$_+$) 
 imply that $x'_{ij} = 0$ for $i_{s-1} < j < i < i_s$. To summarize, the lower triangular part of $X'$ is 
 $e_{i^-_1 1} + e_{i^-_2 i^-_1} + \ldots + e_{i^-_{k^--1}  n}$.
 
 Now, let $s$ be the smallest index such that
 $i^-_s > s+1$. Then $i^-_1=2, \ldots, i^-_{s-1} = s$. 
 Consider the $(s+1)$st column of $X'$. Entries $x'_{i s+1}$ are zero for $i> i^-_s$ due to~\eqref{lsp}
 and for $1< i \leq i^-_s$ due to the properties of $X'$ described above. Since $X'$ is invertible, this means that 
 $x_{1 s+1}\ne 0$. Then the right multiplication by an invertible upper triangular matrix reduces
 $X'$ to a matrix $X''$ such that $x''_{1, s+1}=1$  and the rest of the first row entries are equal to zero, while 
 the lower triangular part of $X''$ and the entries in the strictly upper triangular part made 0 by previous reductions 
 are left unchanged. Comparing with Lemma \ref{u-matrix}, we see that the lower triangular part and the first $s$ rows 
 of $X''$ coincide
 with those of a permutation matrix corresponding to some Coxeter element $v$ of $S_n$. Continuing in the same fashion, 
 we can eventually reduce $X''$ through right multiplication by upper triangular matrices
 to the permutation matrix $\tilde u$, thus showing that $X \in \B_+ u \B_+$.  
 
 The same argument can be used to show that $X \in \B_- v \B_-$ for some Coxeter element $v$, based on conditions 
 (i$_-$) and (ii$_-$). This completes the proof.
 \end{proof}

\section{Inverse problem} 

\subsection{} 
In this section we show how an element $X$ of a Coxeter double Bruhat cell $G^{u,v}$ that admits factorization
(\ref{factorI}) can be restored from its Weyl function (\ref{weyl}) up to a conjugation by a diagonal matrix.

Recall various  useful representations
for the Weyl function $m(\lambda;X)$:
\begin{equation}
m(\lambda;X)=((\lambda\one-X)^{-1} e_1, e_1)=\frac{q(\lambda)}{p(\lambda)}= \sum_{j=0}^\infty \frac {h_j(X)} {\lambda^{j+1}}.
\label{weyl1}
\end{equation}
Here $e_i$ denotes the vector $\left ( \delta_{i\alpha}\right )_{\alpha=1}^n$ of the standard basis in $\mathbb{C}^n$,  
$(\ \cdot, \cdot )$ is the standard inner product,  
 $p(\lambda)$ is the characteristic polynomial of $X$, $q(\lambda)$ is the characteristic polynomial of the $(n-1)\times(n-1)$ submatrix of $X$ formed by deleting the first row and column, and
\begin{equation*}
h_j(X) = (X^j)_{11}= \ (X^{j} e_1,e_1), \quad j\in \mathbb Z,
\end{equation*}
is the {\em $j$th moment\/} of $X$. (Only moments with nonnegative  indices are present in (\ref{weyl1}), however, $h_j(X)$ for $j<0$, that we will need below, are also well-defined, since $X$ is invertible.) In what follows, when it does not lead to a confusion, we occasionally omit the argument and write
$h_j$ instead of $h_j(X)$.

To solve the inverse problem, we generalize the approach of \cite{FayGekh3}, where
only the cases of symmetric or Hessenberg $X$ were treated. The main idea stems from the classical moments problem \cite{akh}:  one considers the space
$\mathbb{C}[\lambda, \lambda^{-1}] /\det (\lambda - X) $ equipped with
the so-called {\em moment functional} - a bi-linear functional $\langle\ , \ \rangle$ on Laurent polynomials in one variable, uniquely defined by the property
 \be
 \label{momfun}
 \langle \lambda^i, \lambda^j \rangle = h_{i+j}.
 \ee
$X$ is then realized as a matrix of the  operator of multiplication by $\lambda$ relative to  appropriately 
selected  bases $ \{p_i^+(\lambda)\}_{i=0}^{n-1}$, $\{p_i^-(\lambda)\}_{i=0}^{n-1}$ bi-orthogonal with respect to 
the moment functional:
$$
\langle p_i^-(\lambda), p_j^+(\lambda) \rangle = \delta_{ij}. 
$$
For example, the classical tridiagonal case corresponds to the orthogonalization of the sequence $1, \lambda, \ldots, \lambda^{n-1}$.  Elements of $G^{s_{n-1}\cdots s_1, s_{n-1}\cdots s_1}$ (cf. Remark~\ref{cmvao}(iii)) result from the bi-orthogonalization
of sequences $1, \lambda, \ldots, \lambda^{n-1}$ and $\lambda^{-1}, \ldots, \lambda^{1-n}$, while 
CMV matrices (Remark~\ref{cmvao}(iv)) correspond to the bi-orthogonalization
of sequences $1, \lambda, \lambda^{-1}, \lambda^{2}, \ldots $ and   $1, \lambda^{-1}, \lambda, \lambda^{-2}, \ldots$.

For any $l\in \Z$, $i\in \mathbb N$ define 
Hankel matrices
$$
\H^{(l)}_i=(h_{\alpha +\beta + l - i - 1})_{\alpha,\beta=1}^i,
$$
and 
Hankel determinants
\begin{equation}
\Delta^{(l)}_i=\det \H^{(l)}_i;
\label{det}
\end{equation}
we assume that $\Delta^l_0=1$ for any $l\in \mathbb Z$.

\br\label{zerotoep} {\rm 
(i) Let $X$ be an $n\times n $ matrix, then it follows
from the Cayley-Hamilton theorem that, for $i > n$, the columns of $\H^{(l)}_i$ are linearly dependent and so $\Delta^{(l)}_i=0$.

(ii) In what follows we will frequently use the identity
\be
\Delta_{i+1}^{(l)}\Delta_{i-1}^{(l)} = \Delta_{i}^{(l-1)}\Delta_{i}^{(l+1)} -   \left( \Delta_{i}^{(l)}\right )^2,
\label{HankJac}
\ee
which is a particular case of Jacobi's determinantal identity. In particular, for $i=n$, (\ref{HankJac}) and the first part of the Remark imply 
$\Delta_{n}^{(l-1)}\Delta_{n}^{(l+1)} = \left( \Delta_{n}^{(l)}\right )^2$ for any $l$.
}
\er

The main result of this Section is 

\bt
 \label{invthm} If $X\in G^{u,v}$ admits factorization~{\rm\eqref{factorI}}, then
\be\label{cd}
\begin{split}
d_i &=\frac{\Delta_i^{(\varkappa_i + 1)}\Delta_{i-1}^{(\varkappa_{i-1})}}{\Delta_i^{(\varkappa_i)}\Delta_{i-1}^{(\varkappa_{i-1} + 1)}},\\ 
c_i^+c_i^- &= \frac{\Delta_{i-1}^{(\varkappa_{i-1})}\Delta_{i+1}^{(\varkappa_{i+1})}}{\left (\Delta_i^{(\varkappa_{i}+1)}\right )^2}
\left ( \frac{\Delta_{i+1}^{(\varkappa_{i+1} + 1)}}{\Delta_{i+1}^{(\varkappa_{i+1})}}  \right )^{\varepsilon_{i+1}} 
\left ( \frac{\Delta_{i-1}^{(\varkappa_{i-1} + 1)}}{\Delta_{i-1}^{(\varkappa_{i-1})}}  \right )^{2-\varepsilon_{i}}
\end{split}
\ee
for any $i\in [1,n]$.
\et

\br{\rm
\label{invrmrk}
 Formulae (\ref{cd}) allow us to restore an element $X\in G^{u,v}$ from its Weyl
function $m(\lambda; X)$ only modulo the diagonal conjugation. Indeed, it is
clear from (\ref{weyl}) that $m(\lambda; X)=m(\lambda;  T X T^{-1})$ for any
invertible diagonal matrix $T=\mbox{diag}(t_1, \ldots, t_n)$. On the other hand,
under the action $X \mapsto T X T^{-1}$, factorization parameters $d_i, c^\pm_i$ in 
(\ref{factorI}) are transformed as follows: $ d_i \mapsto d_i$,  $c^\pm_i \mapsto (t_i/t_{i+1})^{\pm 1} c^\pm_i$,
thus leaving the left-hand sides in (\ref{cd}) unchanged.
}
\er

\subsection{} The rest of the section is devoted to the proof of Theorem \ref{invthm}. The proof relies
on properties of polynomials of the form
\begin{equation}
\P^{(l)}_i(\lambda)=\det \left [\begin{array}{cccc} 
h_{l-i+1} &h_{l-i+2} &\cdots & h_{l+1}
\\ \cdots &\cdots &\cdots &\cdots
\\
h_l & h_{l+1} & \cdots &h_{l+i} 
\\ 1 & \lambda
&\cdots & \lambda^{i}
\end{array}\right ].
\label{polytoep}
\end{equation}

To prove  the first equality in (\ref{cd}) we need two auxiliary lemmas.

\bl Let $m\in [1,n-1]$ and $X_m$ be the $m\times m$ submatrix of $X\in G^{u,v}$ obtained by deleting $n-m$ last rows and columns. 
Then
\begin{equation}\label{submomrange}
h_\alpha (X_m)= h_\alpha(X)
 \end{equation}
for $\alpha\in [\varkappa_m  - m +1, \varkappa_m  + m]$. 
\label{submoment}
\el

\begin{proof} It is enough to prove the claim for $X\in G^{u,v}$ that admits 
factorization (\ref{factorI}). It is clear that $X_m$ does not depend on
parameters $c^\pm_m,\ldots, c^\pm_{n-1}$, $d_{m+1},\ldots, d_n$. Moreover, 
$X_m\in G^{u_m,v_m}$, where $u_m$ and $v_m$ are obtained from $u$ and $v$, respectively, 
by deleting all transpositions $s_i$ with $i\ge m$. Consequently, 
the network
$N_{u_m,v_m}$ can be obtained from the network $N_{u,v}$ 
by deleting all the edges above the horizontal line joining the $m$th source
 with the $m$th sink. Note also that if $\alpha > 0$ then $h_\alpha(X)$ is the sum
 of path weights over all paths from the first source to the first sink in the network obtained
 by the concatenation of $\alpha$ copies $N_{u,v}$. Thus,
 $h_\alpha (X_m)= h_\alpha(X)$ as long as none of the paths involved reaches above the $m$th horizontal
 level. The smallest positive power of $X$ such that in the corresponding network there is a path
 joining  the first source to the first sink  and reaching above the $m$th horizontal
 level is $r=r^+ + r^-$, where $r^\pm=\min\{j: i_j^\pm \geq m+1\}$. By (\ref{k+/-}), 
 $r^\pm= k_m^\pm +1$. Therefore,
\eqref{submomrange} holds for $\alpha\in [0, k_m^+ + k_m^- +1 ]$. By Lemma~\ref{allcomb}(iii),~\eqref{epsum} and~\eqref{kappa}, 
the latter interval coincides with $[0,\varkappa_m+m]$.
 
Next, consider the network $\bar N_{u^{-1},v^{-1}}$ that represents $X^{-1}$ corresponding to
factorization (\ref{factorinv}). Note that this network differs from $N_{u^{-1},v^{-1}}$. In
particular, in $\bar N$ all ``north-east'' edges are to the left of any ``south-east'' edge.
Once again,  the network $\bar N_{u_m^{-1},v_m^{-1}}$ is obtained from the network $\bar N_{u^{-1},v^{-1}}$  by deleting all the edges above the horizontal line joining the $m$th sink
 with the $m$th source. The smallest positive power of $X^{-1}$ such that in the corresponding network obtained by concatenation of copies of $\bar N_{u^{-1},v^{-1}}$ there is a path
 joining  the first source to the first sink  and reaching above the $m$th horizontal
 level is $\bar r = \bar r^++\bar r^--1$, where $\bar r^+=\min\{j: l_j^+ \geq m+1\}$ and $\bar r^-=\min\{j: l_j^- \geq m+1\}$.
The difference in the formulas for $r$ and $\bar r$ stems from the difference in the structure of the networks $N$ and $\bar N$: the latter already contains paths from the first source to the first sink that reach above the first horizontal level. Consequently,
it is possible that $\bar r=1$ for some $m>1$, whereas $r>1$ for any $m>1$. 
 
One can define combinatorial parameters $\bar\varepsilon^\pm_i$ and $\bar k^\pm_i$ similarly to~\eqref{eps} and~\eqref{k+/-} based on the sets $L^\pm$ rather than on $I^\pm$ (cp.~\eqref{IL}). It follows immediately from definitions that $\bar\varepsilon^\pm_i=1-\varepsilon^\pm_i$ for
$i\in [2,n-1]$ and $\bar\varepsilon^\pm_1=\varepsilon^\pm_1=1$. One can prove, similarly to Lemma~\ref{allcomb}(iii), that 
$\bar k^\pm_i=i-\sum_{\beta=1}^i\bar\varepsilon^\pm_\beta$, which translates to $\bar k^\pm_i=\sum_{\beta=1}^i\varepsilon^\pm_\beta-1$.
Since $\bar r^\pm=\bar k^\pm_m+1$, we get $\bar r=\sum_{\beta=1}^m\varepsilon_\beta-1$, and hence, by~\eqref{kappa}, $\bar r=m-\varkappa_m$. If $\bar r=1$, then $\varkappa_m-m+1=0$, and the interval $[0,\varkappa_m+m]$ coincides with $[\varkappa_m-m+1,\varkappa_m+m]$. Otherwise we can concatenate up to $\bar r-1=m-1-\varkappa_m$ networks $\bar N_{u^{-1},v^{-1}}$,
 and hence \eqref{submomrange} holds additionally for $\alpha\in [\varkappa_m-m+1,-1]$.
\end{proof}

\bl\label{charpoly}
Let $m\in [1,n-1]$, then
\begin{equation}
\det (\lambda - X_m) = \frac{1}{\Delta_m^{(\varkappa_m)}} \P^{( \varkappa_m )}_{m}(\lambda).
\label{charpolydet}
\end{equation}
In particular,
\begin{equation}
d_1\cdots d_m = \frac{\Delta_m^{(\varkappa_m + 1)}}{\Delta_m^{(\varkappa_m)}}.
\label{d1m}
\end{equation}
\el

\begin{proof} Let $\det (\lambda - X_m) = \lambda^m
+\sum_{i=0}^{m-1} a_{mi} \lambda^i$. Then the Hamilton-Cayley
theorem implies
\begin{equation*}
h_{\alpha+m} (X_m) + \sum_{i=0}^{m-1} a_{mi} h_{\alpha+i}(X_m)=0
\end{equation*}
for any $\alpha\in\mathbb{Z}$.
By Lemma \ref{submoment}, this relation remains valid if we replace
$h_{\alpha+i}(X_m)$ with $h_{\alpha+i}=h_{\alpha+i}(X)$ for
$i=0,\ldots, m$, as long as $\varkappa_m - m +1 \leq \alpha \leq \varkappa_m$. 
This
means that, after the right multiplication of the matrix used in
the definition (\ref{polytoep}) of $\P^{(\varkappa_m )}_{m}$ by the unipotent matrix
$\one + \sum_{\beta=0}^{m-1} a_{m\beta}
e_{\beta+1,m+1}$, one gets a matrix of the form
$$
\left (\begin{array}{cc} \H^{(\varkappa_m )}_{m} & 0\\
1\ \lambda\ \cdots\ \lambda^{m-1} & \det (\lambda - X_m)\end{array}\right ),
$$
and (\ref{charpolydet}) follows. Since $\det X_m = d_1 \cdots d_m$, (\ref{d1m}) drops out immediately from (\ref{charpolydet}) and 
\eqref{polytoep} after substitution $\lambda=0$.
\end{proof} 

\br\label{detiden} {\rm 
Combining Remark~\ref{zerotoep}(ii) with
(\ref{d1m}) for $m=n$ and taking into account that $\det X=d_1\cdots d_n$, we see that  for any $l$
\[
\frac{\Delta_{n}^{(l+1)}}{\Delta_{n}^{(l)}} = \det X,
\]
which implies that for any $l$
\be
\label{shift}
\Delta_{n}^{(l)} = \Delta_{n}^{(n-1)} \det X^{l+1-n}.
\ee
}
\er

 Now, the first formula in (\ref{cd}) is an easy consequence of (\ref{d1m}). To be in  a position 
 to prove the second formula in (\ref{cd}), we first need the following statement. For any $i\in [1,n]$ define subspaces
 $$
 \L^+_i=\Span \{e^T_1,\dots,e^T_i\},\qquad \L^-_i=\Span \{e_1,\dots, e_i\}.
 $$
 Besides,  put
 \begin{equation}
{\displaystyle \gamma^\pm_{i} 
= (-1)^{(i-1)\varepsilon^\pm_i}d_i^{-\varepsilon^\pm_i} \prod_{j=1}^{i-1}
c_j^\pm d_j^{\bar\varepsilon^\pm_j-\varepsilon^\pm_i},\quad i\in [2,n],} \label{gamma}
\end{equation}
where $\bar\varepsilon^\pm_j$ are defined in the proof of Lemma~\ref{submoment}, and $\gamma^\pm_1=1$.
 
 \begin{lemma}
\label{flaglemma} For any $i\in [1, n]$ one has
\begin{equation}
\gamma^+_{i} e^T_{i} = e^T_1 X^{\zeta^+_i} \mod \L^+_{i-1}  \label{flag1}
\end{equation}
and
\begin{equation}
\gamma^-_{i} e_{i} =  X^{\zeta^-_i} e_1 \mod \L^-_{i-1}.    \label{flag2}
\end{equation}
In particular,
$$
\L^+_i=\Span\{e^T_1 X^{\zeta^+_1},\dots,e^T_1 X^{\zeta^+_i}\}, \qquad
\L^-_i=\Span\{ X^{\zeta^-_1}e_1,\dots, X^{\zeta^-_i}e_1\}.
$$
\label{lemmaflag}
\end{lemma}

 \begin{proof} A proof for~\eqref{flag1} was given in \cite{FayGekh3}. We present it here in order to keep the paper self-contained. The case of~\eqref{flag2} can be treated similarly.
 
For any $X$ given by (\ref{factorI}),  
consider
an upper triangular matrix
$$
V=D (\one - C^+_{k^+})^{-1} (\one - C^+_{k^+-1})^{-1}\cdots (\one -
C^+_1)^{-1}.
$$
Note that $V$ is the upper triangular factor in
the Gauss factorization of $X$.

By (\ref{Cj}), $e_r^T C^+_j=0$ for $r< i^+_{j-1}$ and $r\geq i^+_j$, and hence
$$
e_r^T(1-C^+_j)^{-1}=  
\left \{ \begin{array}{cc}
e_r^T, &\quad r<i_{j-1}^+,\\
 e_r^T \mod \L^+_{r-1}, &\quad r \ge i_j^+.
\end{array}\right.
$$
Thus, for $j\in [1,k^+]$,
\begin{align*}
e_{i^+_{j-1}}^T V&=d_{i^+_{j-1}}e_{i^+_{j-1}}^T (\one - C^+_{k^+})^{-1} \cdots
(\one - C^+_1)^{-1}\\
&=d_{i^+_{j-1}}e_{i^+_{j-1}}^T (\one - C^+_j)^{-1} \mod \L^+_{i^+_j-1}\\
 &= d_{i^+_{j-1}}c^+_{i^+_{j-1}}\cdots c^+_{i^+_{j}-1}e^T_{i^+_{j}}\
\mod \L^+_{i^+_j-1}.
\end{align*}
A similar argument shows that $ e_r^T V \in  \L^+_{i^+_j-1}$ for $r <
i^+_{j-1}$. This implies
\begin{align*} 
e_{1}^T V^j&=
\left
(\prod_{\beta=0}^{j-1} d_{i^+_{\beta}}c^+_{i^+_{\beta}}\cdots
c^+_{i^+_{\beta+1}-1}\right ) e_{i^+_{j}}^T \mod \L^+_{i^+_j-1}\\ &
= \left (\prod_{r=1}^{i^+_j-1} c^+_r
d^{\bar\varepsilon_r^+}_r \right ) e^T_{i^+_{j}} \mod
\L^+_{i^+_j-1}. 
\end{align*}
Besides, $e^T_{i^+_{j-1}}XV^{-1}=e^T_{i^+_{j-1}}\mod \L^+_{i^+_{j-1}}$, hence the above relation can be re-written as
\begin{equation}\label{aux2}
e_{1}^T X^j=\left (\prod_{r=1}^{i^+_j-1} c^+_r
d^{\bar\varepsilon_r^+}_r \right ) e^T_{i^+_{j}} \mod
\L^+_{i^+_j-1}. 
\end{equation}

On the other hand,  for $l\in [i^+_{j-1},i^+_j-1]$, define
$m\ge 0$ 
so that 
$l+m+1$ is the smallest index greater than $l$ that 
belongs to the index set $L^+$. Then
\begin{align*}
e_l^T V^{-1}&= e_l^T(\one - C^+_j)\cdots(\one -
C^+_{k^+})D^{-1} \mod \L^+_{l+m} \\
&=
\left ((-1)^{m+1} c^+_l\cdots c^+_{l+m} d_{l+m+1}^{-1}\right )
e_{l+m+1}^T \mod \L^+_{l+m}. 
\end{align*}
The latter equality implies
\begin{equation}
e_1^T V^{-\alpha} 
= \left (
(-1)^{l^+_\alpha-1} c^+_1\cdots c^+_{l^+_\alpha-1} d^{-1}_{l^+_1}\cdots
d^{-1}_{l^+_\alpha}\right ) e_{l^+_\alpha}^T \mod
\L^+_{l^+_\alpha-1} \label{aux4}
\end{equation}
for any $l^+_\alpha\in L^+$ distinct from $1$ and $n$.
Note now that $l_\alpha=i$ if and only if
$\alpha=\sum_{\beta=1}^{i-1} \varepsilon^+_i$ (this can be considered as an analog or~\eqref{twocount}). Furthermore,
$d^{-1}_{l^+_1}\cdots d^{-1}_{l^+_\alpha}= \prod_{j=2}^i
d_j^{-\varepsilon_j^+}$. Thus, one can re-write (\ref{aux4}) as
$$
e_1^T V^{-\sum_{j=1}^{i-1} \varepsilon^+_i}=(-1)^{i-1}
\prod_{j=1}^{i-1} c^+_j d_{j+1}^{-\varepsilon^+_{j+1}}
e_i^T \mod \L_{i-1}^+.
$$
Together with $e^T_lVX^{-1}= e^T_l\mod \L^+_{l+m}$ this leads to
$$
e_1^T X^{-\sum_{j=1}^{i-1} \varepsilon^+_i}=(-1)^{i-1}
\prod_{j=1}^{i-1} c^+_j d_{j+1}^{-\varepsilon^+_{j+1}}
e_i^T \mod \L_{i-1}^+.
$$
Combining this relation with (\ref{aux2}) and Lemma~\ref{allcomb}(ii), one
gets~\eqref{flag1}.
 \end{proof}

\begin{examp}{\rm  We illustrate (\ref{flag1}) using Example~\ref{runex} and Fig.~\ref{factorex}. If $j>0$ then to find $i$
such that $e_1^T X^j = \gamma^+_i e^T_i \mod \L^+_{i-1}$ it is enough to find the highest sink
that can be reached by a path starting from the source 1 in the network obtained by concatenation of
$j$ copies of $N_{u,v}$. Thus, we conclude from Fig.~\ref{factorex}, that 
\begin{align*}
e_1^T X &= d_1 c_1^+ c_2^+ e^T_3 \mod \L^+_{2},\quad e_1^T X^2 = d_1 c_1^+ c_2^+ d_3 c_3^+ e^T_4 \mod \L^+_{3},\\ 
e_1^T X^3 &= d_1 c_1^+ c_2^+ d_3 c_3^+ d_4 c_4^+ e^T_5 \mod \L^+_{4}. 
\end{align*}
Similarly, using the network  $\bar N_{u^{-1},v^{-1}}$ shown in Fig.~\ref{factorinvex}, one observes
that $e_1^T X^{-1} = -c_1^+ d_2^{-1}  e^T_2 \mod \L^+_{1}$. These relations are in agreement
with (\ref{flag1}).

\begin{figure}[ht]
\begin{center}
\includegraphics[height=4.0cm]{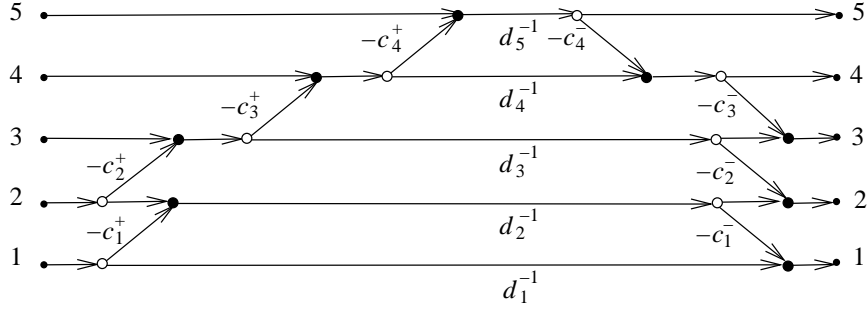}
\caption{Network $\bar N_{u^{-1},v^{-1}}$ for the double Bruhat cell $G^{u,v}$  from Example~\ref{runex}}
\label{factorinvex}
\end{center}
\end{figure}
}
\end{examp}

Define Laurent polynomials 
 $$
 p_{i}^\pm(\lambda) 
 = \frac{(-1)^{(i-1)\varepsilon_i^\pm}}{\gamma_{i}^\pm\Delta_{i-1}^{(\varkappa_{i-1})}}\lambda^{k^\pm_i - i + 1} \P^{(\varkappa_{i-1}-\varepsilon_i^\pm)}_{i-1}(\lambda),\quad i\in [1,n].
 $$

\begin{coroll} 
 \label{polycorr}
 {\rm (i)} One has
 $$ 
 e^T_1 p_{i}^+(X)= e^T_{i},\quad p_{i}^-(X)e_1= e_{i},\quad i\in [1,n].
 $$
 
 {\rm (ii)} 
 For any eigenvalue $\lambda$ of $X$, the column-vector
$(p_i^+(\lambda))_{i=1}^n$ and the row-vector $(p_i^-(\lambda))_{i=1}^n$ are,
respectively, right and left eigenvectors of $X$ corresponding to $\lambda$.
 \end{coroll}

\begin{proof} (i) We will only give a proof for $p_{i}^+(\lambda)$.  By Lemma \ref{flaglemma}, $e^T_1X^{\zeta_\alpha^+}$, $\alpha=1,\dots,i-1$, form a basis of $\L^+_{i-1}$, hence, taking into account~\eqref{flag1}, we get
 $$
 \gamma_i^+ e_i^T = e^T_1\left 
 ( X^{\zeta_i^+} +  \sum_{\alpha=1}^{i-1} \pi_{\alpha}X^{\zeta^+_\alpha}
 \right )
 $$ 
 for some coefficients $\pi_\alpha$. By Lemma~\ref{allcomb}(iv), this can be re-written as
 $$
 \gamma_i^+ e_i^T = e^T_1 X^{k^+_i - i + 1} \sum_{\alpha=1}^{i} \tilde \pi_{\alpha}X^{\alpha-1},
 $$ 
 where either $\tilde \pi_i=1$ (if $\varepsilon_i^+=0$), or
 $\tilde \pi_1=1$ (if $\varepsilon_i^+=1$). Define a polynomial $p(\lambda)=\sum_{\alpha=1}^{i} \tilde \pi_{\alpha}\lambda^{\alpha-1}$.
 By Lemma \ref{flaglemma}, vectors $X^\alpha e_1$, $\alpha\in M^-_{i-1}$, span the
 subspace $\L^-_{i-1}$. Therefore, by Lemma~\ref{allcomb}(iv),
 $\left (X^{k^+_i - i + 1}  p(X) X^\alpha e_1, e_1\right ) = 0 $ for $\alpha \in 
 [k_{i-1}^- - i +2, k_{i-1}^-]$. This system of linear equations 
 determines $p(\lambda)$ uniquely as
 $$
 p(\lambda)= \frac{(-1)^{(i-1)\varepsilon_i^+}}{\Delta_{i-1}^{(k_i^+ + k_{i-1}^- -i +1 +\varepsilon_i^+)}}
 \P^{(k_i^+ + k_{i-1}^- -i +1)}_{i-1}(\lambda), 
 $$ 
 which by~\eqref{kappa} and Lemma~\ref{allcomb}(iii) gives
 $$
p(\lambda)= \frac{(-1)^{(i-1)\varepsilon_i^+}}{\Delta_{i-1}^{(\varkappa_{i-1})}} \P^{(\varkappa_{i-1}-\varepsilon_i^+)}_{i-1}(\lambda)=\gamma^+_i\lambda^{-k^+_i + i - 1}p^+_i(\lambda).
$$ 
It remains to notice that
  $\gamma_i^+ e_i^T =  e_1^T X^{k^+_i - i + 1}  p(X)$, and 
  the result follows.

(ii) Let $z^\lambda = (z^\lambda_i)_{i=1}^n$ be a right eigenvector
of $X$ corresponding to an eigenvalue $\lambda$. Then
$
e_1^T p_i^+(X) z^\lambda = e_i^T z^\lambda,
$
which means that $z^\lambda_i =p_i^+(\lambda)  e_1^T  z^\lambda =
p_i^+(\lambda)  z_1^\lambda$.
Therefore, $z_1^\lambda\ne 0$, $p_i^+(\lambda) =
\frac{z^\lambda_i}{z^\lambda_1}$ and
$(p_i^+(\lambda))_{i=1}^n$ is a right eigenvector of $X$. The case of
$(p_i^-(\lambda))_{i=1}^n$
can be treated in the same way.
 \end{proof}

 \br 
 {\rm The statement of Corollary \ref{polycorr}{i} is equivalent to saying that Laurent polynomials 
 $p_i^{\pm}(\lambda)$ form a bi-orthonormal family with respect to  the  moment functional (\ref{momfun}) obtained by the Gram process 
 applied to the sequences $1, \lambda^{\zeta_1^+},  \lambda^{\zeta_2^+},\dots$ and $1, \lambda^{\zeta_1^-}, 
  \lambda^{\zeta_2^-},\dots$. Indeed, 
 $$
 \langle p_i^+(\lambda), p_j^-(\lambda) \rangle = e_1^T p_i^+(X) p_j^-(X) e_1 = (e_i, e_j)=  \delta_{ij}.
 $$
 }
 \er 

 For $l\in [2,n]$ define
\be
\label{Gamma_l}
\Gamma_l= \prod_{i=2}^l d_i^{-\varepsilon_i} \prod_{j=1}^{i-1}
c_j d_j^{\bar\varepsilon_j-\varepsilon_i},
\ee
where $\bar\varepsilon_j=\bar\varepsilon_j^++\bar\varepsilon_j^-$, $j\in [1,n]$, similarly to~\eqref{epsum}.
 
 \begin{coroll}\label{corshift} 
 For any $k\in \Z$,
\be
\label{Deltashift}
\Delta_l^{(\varkappa_{l}+k)} = \Gamma_l\ \left ( X^k \right )_{[1,l]}^{[1,l]}.
\ee
\end{coroll}

\begin{proof} By Lemma~\ref{flaglemma}, 
\begin{equation*}
\begin{split}
\left ( X^k \right )_{[1,l]}^{[1,l]} = 
\det \left ( e_i^T X^k e_j \right )_{i,j=1}^l &= \det \left ( \frac{1}{\gamma_i^+ \gamma_j^-}e_1^T X^{k+\zeta_i +\zeta_j} e_1 \right )_{i,j=1}^l\\
&= \left(\prod_{i=2}^l \frac{1}{\gamma_i^+ \gamma_i^-}\right)\det \left (e_1^T X^{k+\zeta_i +\zeta_j} e_1 \right )_{i,j=1}^l. 
\end{split}
\end{equation*}
By Lemma \ref{allcomb}(iii),~(iv) and (\ref{kappa}), the determinant in the last expression above is equal, up to a sign, to
$\Delta_l^{(\varkappa_{l}+k)}$.  By Lemma \ref{allcomb}(ii),
the sign is determined as 
$$ 
\prod_{i=2}^l (-1)^{(i-1)\varepsilon_i^+}
\prod_{j=2}^l (-1)^{(j-1)\varepsilon_j^-} = (-1)^{\sum_{j=2}^l (i-1)\varepsilon_i}. 
$$
On the other hand, (\ref{gamma}) implies that $\prod_{i=2}^l \gamma_i^+ \gamma_i^- = (-1)^{\sum_{j=2}^l (i-1)\epsilon_i} \Gamma_l$,
 and (\ref{Deltashift}) follows.
\end{proof}
 
Finally, we can complete the proof of Theorem \ref{invthm}. 
To  prove the second 
 relation in (\ref{cd}), observe that by Lemma \ref{flaglemma} and Corollary \ref{polycorr}(i),
 $$
 \gamma_i^+ = \left ( X^{\zeta_i^+} p^-_{i} (X) e_1, e_1 \right )= 
  \frac{(-1)^{(i-1)\varepsilon_i^-}}{\gamma_{i}^-\Delta_{i-1}^{(\varkappa_{i-1})}}
 \left (  X^{\zeta_i^+ + k^-_i - i + 1} \P^{(\varkappa_{i-1}-\varepsilon^-_{i})}_{i-1}(X) e_1, e_1 \right ).
 $$
 Since by (\ref{nunu}), \eqref{epsum}, (\ref{kappa}) and Lemma~\ref{allcomb}(iii), 
 $\zeta_i^++k_i^--i+1=\varkappa_i-(i-1)\varepsilon_i^+$, the above equality gives
 $$
 \gamma_i^+ = 
  \frac{(-1)^{(i-1)\varepsilon_i^-}}{\gamma_{i}^-\Delta_{i-1}^{(\varkappa_{i-1})}}
(-1)^{(i-1)\varepsilon_i^+}  \Delta_i^{(\varkappa_{i})},
 $$
 and so
 $$
  \gamma_i^+ \gamma_{i}^- = (-1)^{(i-1)\varepsilon_i} 
  \frac{ \Delta_i^{(\varkappa_{i})}}{\Delta_{i-1}^{(\varkappa_{i-1})}}.
 $$
Consider the ratio 
$$
\frac{\gamma_{i+1}^+ \gamma_{i+1}^-}{\gamma_{i}^+ \gamma_{i}^-}=(-1)^{i\varepsilon_{i+1}-(i-1)\varepsilon_i}
\frac{\Delta_{i-1}^{(\varkappa_{i-1})}\Delta_{i+1}^{(\varkappa_{i+1})}}{\left (\Delta_i^{(\varkappa_{i})}\right )^2}.
$$
 Taking into account (\ref{gamma}), we obtain
$$
c_i^+ c_{i}^- =\frac
{\Delta_{i-1}^{(\varkappa_{i-1})}\Delta_{i+1}^{(\varkappa_{i+1})}}{\left (\Delta_i^{(\varkappa_{i})}\right )^2}  
\frac{d_{i+1}^{ \varepsilon_{i+1}}}{d_i^{2-\varepsilon_i}}(d_1d_2\cdots d_{i-1})^{\varepsilon_{i}-\varepsilon_{i+1}},
$$
which together with the first relation in \eqref{cd} gives the second one.

\section{ Cluster algebra}

\subsection{} Let $N_{u,v}$ be  the network associated with $X\in GL_n$ and the factorization scheme (\ref{factorGuv}). We will now construct  a network $N_{u,v}^\circ$ in an annulus as follows: 

(i) For each $i\in [1,n]$,  add an edge that is directed from the $i$th sink on the right
to the $i$th source on the left in such a way that moving from the $i$th source  to the $i$th sink in $N_{u,v}$ and then returning to the $i$th source in along the new edge, one traverses a closed contour in the counter-clockwise direction. These $n$ new edges do not intersect and to each of them we assign weight $1$. 


(ii) Place the resulting network in the interior of an annulus  in such a way that the cut
(as defined in Section \ref{networks}) intersects $n$ new edges, and the inner boundary of the annulus is inside the domain bounded
by the top horizontal path in $N_{u,v}$ and the $n$th new edge. 


(iii) Place one source and one sink  on the outer boundary of the annulus, the former slightly to the right and the latter 
slightly to the left of the cut.  Split the first (the outermost) new edge into three similarly directed edges
by adding two vertices, a black one slightly to the right and a white one 
slightly to the left of the cut. Add an edge  with weight $w_{in}$ directed from the source to the new black vertex and another edge with weight $w_{out}$ directed from the 
new white vertex to the sink. 

It is important to note that the gauge group is rich enough to assure the possibility of assigning weights as described above,
with unit weights at prescribed edges.  

\begin{examp}\label{runexan}
{\rm
The network $N_{u,v}^\circ$ that corresponds to $N_{u,v}$ discussed in Example~\ref{runex} is shown in Figure~\ref{annex}.} 

\begin{figure}[ht]
\begin{center}
\includegraphics[height=9.0cm]{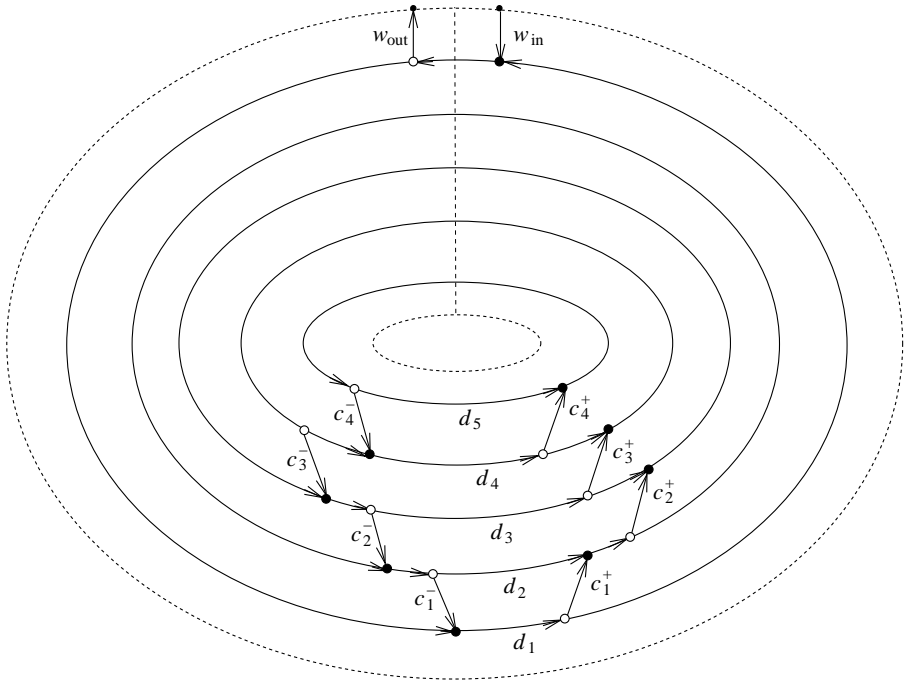}
\caption{Network $N_{u,v}^\circ$ for $N_{u,v}$ from Example~\ref{runex}}
\label{annex}
\end{center}
\end{figure}

\end{examp}

Now let $X$ be an element in a Coxeter double Bruhat cell $G^{u,v}$, $N_{u,v}$ be the network that corresponds
to the factorization (\ref{factorI}) and $N_{u,v}^\circ$
be the corresponding network in an annulus. Then  $N_{u,v}^\circ$ has $2(n-1)$ bounded faces $f_{0 i}$, $f_{1 i}$, $i\in [1,n-1]$, which we enumerate as follows: each face $f_{0 i}$ contains a piece of the cut and each
 face $f_{1 i}$ does not, and the value of $i$ is assigned according to the natural bottom to top
 order inherited from $N_{u,v}$.  There are also three unbounded faces: two of them, adjacent to the outer
 boundary of the annulus, will be denoted $f_{0 0}, f_{1 0} $, where the first index is determined using the same convention as for bounded faces. The third unbounded face is adjacent to the inner boundary. It will be denoted by $f_{0n}$. 
 
 Recall that faces of
 $N_{u,v}^\circ$ correspond to the vertices of the directed dual network $(N_{u,v}^\circ)^*$ (as defined in Section \ref{networks}).
To describe adjacency  properties of  $(N_{u,v}^\circ)^*$,
let us first consider inner vertices of $N_{u,v}^\circ$.  
There are altogether $4n- 2$ inner vertices. 
For every $i\in [1,n-1]$, the $i$th level contains two black and two white vertices. 
One of the black vertices is an endpoint of an edge directed from the $(i+1)$th level; it is denoted $v_b^-(i)$. 
The other one is an endpoint of an edge directed from the $(i-1)$th level (or from the source, for $i=1$); it is denoted $v_b^+(i)$. 
Similarly, white vertices are start points of the edges directed towards the $(i+1)$th and the $(i-1)$th levels (or towards the sink,
for $i=1$); they are denoted $v_w^+(i)$ and $v_w^-(i)$, respectively. The $n$th level contains only $v_b^+(n)$ and $v_w^-(n)$.

Now we can describe faces $f_{0 i}$, $i\in [0,n]$, and $f_{1 i}$, $i\in [0,n-1]$, by listing their vertices in the counterclockwise 
order. Below we use the following convention: if a vertex  appears in the description of a face with the exponent $0$, this
means that this vertex  does not belong to the boundary of the face. With this in mind, we obtain
\begin{equation*}
\begin{split}
 f_{1 i} &= \left (  v_b^-(i)  v_w^-(i)^{\varepsilon_i^-}  v_b^+(i)^{\varepsilon_i^+}  v_w^+(i) v_b^+(i+1) 
v_w^+(i+1)^{\bar\varepsilon_{i+1}^+}  v_b^-(i+1)^{\bar\varepsilon_{i+1}^-}  v_w^-(i+1)\right ),\\ 
 f_{0 i} &= \left (  v_w^+(i)  v_b^+(i)^{\bar\varepsilon_i^+}  v_w^-(i)^{\bar\varepsilon_i^-} v_b^-(i) 
v_w^-(i+1)  v_b^-(i+1)^{\varepsilon_{i+1}^-} v_w^+(i+1)^{\varepsilon_{i+1}^+}  v_b^+(i+1)\right )
\end{split}
\end{equation*}
for $i\in [2,n-2]$ and
\begin{equation*}
\begin{split}
 f_{1 0} &= \left ( \mbox{source}\   v_b^+(1) v_w^+(1) v_b^-(1)  v_w^-(1)\  \mbox{sink}\right ),\\  
 f_{0 0} &= \left ( \mbox{sink}\  v_w^-(1)) v_b^+(1)\  \mbox{source}\right ),\\ 
 f_{1 1} &= \left (  v_b^-(1) v_w^+(1) v_b^+(2)  v_w^+(2)^{\bar\varepsilon_2^+}  v_b^-(2)^{\bar\varepsilon_2^-} v_w^-(2) \right ),\\  
 f_{0 1} &= \left (  v_w^+(1) v_b^+(1) v_w^-(1)  v_b^-(1) v_w^-(2) v_b^-(2)^{\varepsilon_2^-}  v_w^+(2)^{\varepsilon_2^+} v_b^+(2) \right ), \\
 f_{1 n-1} &= \left (    v_b^-(n-1)  v_w^-(n-1)^{\varepsilon_{n-1}^-}  v_b^+(n-1)^{\varepsilon_{n-1}^+} v_w^+(n-1)  v_b^+(n)  v_w^-(n)\right ),\\ 
 f_{0 n-1} &= \left (  v_w^+(n-1)  v_b^+(n-1)^{\bar\varepsilon_{n-1}^+}  v_w^-(n-1)^{\bar\varepsilon_{n-1}^-} v_b^-(n-1)  v_w^-(n)  v_b^+(n)\right ), \\  
 f_{0 n} &= \left (   v_b^+(n)v_w^-(n) \right ).
\end{split}
\end{equation*}

\begin{examp}\label{runexan2}
{\rm
Vertices and faces of $N_{u,v}^\circ$ from Example~\ref{runexan} are shown in Figure~\ref{ann2}.
Consider face $f_{13}$. As we have seen before in Example~\ref{runex}, $\varepsilon^+_3=\varepsilon^+_4=\varepsilon^-_4=0$ and $\varepsilon^-_3=1$, so~the above description yields $f_{1 3} = \left (  v_b^-(3)  v_w^-(3) v_w^+(3) v_b^+(4)
v_w^+(4)  v_b^-(4) v_w^-(4)\right )$.
} 
\end{examp}

\begin{figure}[ht]
\begin{center}
\includegraphics[height=9.0cm]{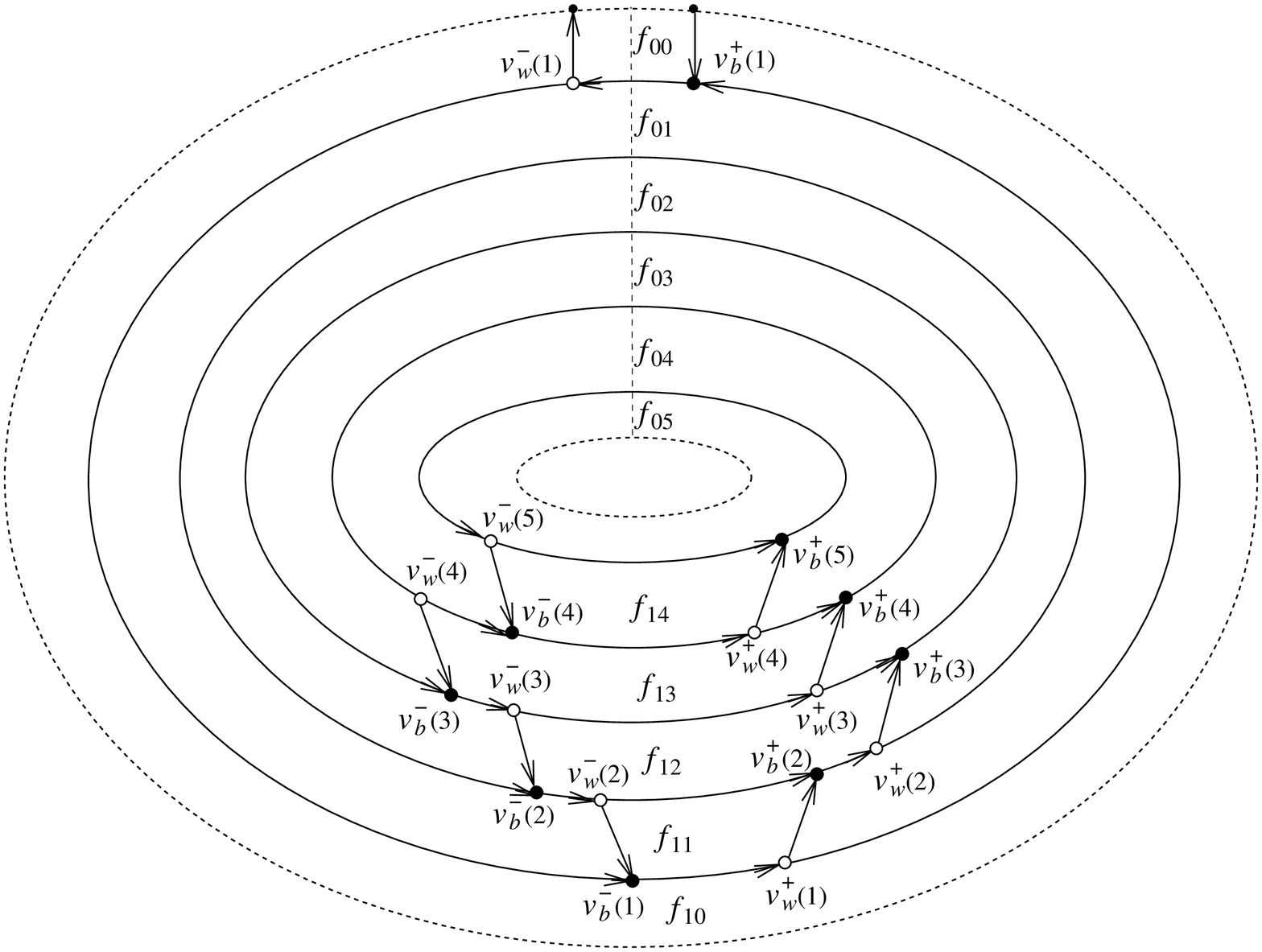}
\caption{Vertices and faces of $N_{u,v}^\circ$ for Example~\ref{runexan}}
\label{ann2}
\end{center}
\end{figure}

In our description of $(N_{u,v}^\circ)^*$ below we use the following convention: whenever we say that there are $\alpha<0$ edges directed 
from vertex $f$ to vertex $f'$, it means that there are $|\alpha|$ edges directed from $f'$ to $f$. It is easy to see that for
any $i\in [0,n-1]$,  faces $f_{0 i}$ and  $f_{1 i}$ have two common edges: $v_w^-(i+1)\to v_b^-(i)$ and
$v_w^+(i)\to v_b^+(i+1)$. The startpoints of each one of the edges are white, the endpoints are black, and in both cases face $f_{1i}$ lies to the right of the edge. This means that in $(N_{u,v}^\circ)^*$  there are two edges directed from $f_{0 i}$ to  $f_{1 i}$. Similarly, the description above shows that $(N_{u,v}^\circ)^*$ has
\begin{itemize}
\item[(i)] $1-\varepsilon_{i+1}$ edges directed from $f_{1 i+1}$ to $f_{1 i}$,
\item[(ii)] $2- \varepsilon_{i+1}$ edges directed from $f_{1 i}$ to $f_{0 i+1}$,
\item[(iii)] $1- \varepsilon_{i+1}$ edges directed from $f_{0 i+1}$ to $f_{0 i}$,
\item[(iv)] $\varepsilon_{i+1}$ edges directed from $f_{1 i+1}$ to $f_{0 i}$
 \end{itemize}
for $i\in [1, n-2]$,
one edge directed from $f_{1 n-1}$ to $f_{0 n}$ and one edge directed from $f_{0 n}$ to $f_{0 n-1}$. Finally,
for $|i-j| >1$ and $a, b \in \{ 0,1\}$, vertices $f_{a i}, f_{b j}$ in $(N_{u,v}^\circ)^*$ are not 
connected by edges.

Next, we associate with every face $f_{st}$ in $N_{u,v}^\circ$ a face weight $y_{st}$. We will see below that 
 $y_{st}$ are related to parameters $c^\pm_i$, $d_i$ via a monomial transformation. At this point, however, let us examine the standard  Poisson bracket on $\mathcal F_{N_{u,v}^\circ}$. As was explained in Section~\ref{networks}, this bracket is completely described by Proposition~\ref{PSviay}, which together with the above description of  $(N_{u,v}^\circ)^*$ implies the following Poisson relations for face weights: 
\be\label{taubracks}
\begin{split}
&\{y_{0i}, y_{1i}\} = 2 y_{0i} y_{1i}, \quad \{y_{1i}, y_{1i+1}\} = -  (1-\varepsilon_{i+1})y_{1i} y_{1i+1},\\
&\{y_{1i}, y_{0 i+1}\} = (2-\varepsilon_{i+1}) y_{1i} y_{0 i+1}, \quad 
\{y_{0i}, y_{0 i+1}\} =  - (1-\varepsilon_{i+1}) y_{0i} y_{0 i+1},\\ 
&\{y_{0i}, y_{1i+1}\} = -  \varepsilon_{i+1} y_{0i} y_{1i+1} 
\end{split}
\ee
for $i\in [2,n-2]$ and
\be\label{taubracks0}
\begin{split}
&\{y_{00}, y_{1 0}\} =  y_{00} y_{1 0},\quad \{y_{10}, y_{0 1}\} = 2 y_{10} y_{0 1},\\ 
&\{y_{10}, y_{1 1}\} = - y_{10} y_{1 1},\quad \{y_{00}, y_{0 1}\} = -  y_{00} y_{0 1},\\ 
&\{y_{0n}, y_{0 n-1}\} =  y_{0n} y_{0 n-1 },\quad \{y_{0n}, y_{1n-1}\} = - y_{0n} y_{1 n-1}, 
\end{split}
\ee
and the rest of the brackets are zero.

Denote by $M(\lambda)$  the {\it boundary measurement\/} for the network $N_{u,v}^\circ$, and put $H_0=w_{in} w_{out}$. 
We have the following

\bp Let $X\in G^{u,v}$ be given by {\rm \eqref{factorI}}, then 
$$
M(\lambda) =  H_0 \left ( (\lambda \one + X)^{-1}e_1,e_1 \right )= -H_0 m(-\lambda; X).
$$ 
\label{Weylresponse}
\ep

\begin{proof} Clearly, $M(\lambda)$ is a power series in $\lambda^{-1}$ with the coefficient of $\lambda^{-k}$ equal to  $(-1)^{k-1}$ times the sum of weights of all paths from the source  to the sink that cross the cut exactly $k$ times. 
Moreover, the leading term of  $M(\lambda)$ is  $H_0 \lambda^{-1}$. 
Denote $M(\lambda)= \sum_{k=0}^\infty (-1)^{k}H_k  \lambda^{-k-1}$.
Since the weight of every path has a factor $H_0$, 
computing $H_k/H_0$ is equivalent to computing the boundary measurement
between the first source and the first sink in the planar network obtained by concatenation
of $k$ copies of $N_{u,v}$. Therefore, $H_k/H_0= (X^ke_1,e_1)=h_k$, and
$$
M(\lambda)= H_0\sum_{k=0}^\infty (-1)^{k}h_k  \lambda^{-k-1}
= H_0\left( \lambda^{-1}\sum_{k=0}^\infty(-\lambda X)^ke_1,e_1\right) = H_0 \left ( (\lambda \one + X)^{-1}e_1,e_1 \right ).
$$
 \end{proof}

 Let $\RR_{n}$ denote the space of rational functions of the form $Q/P$,
where $P$ is a  monic polynomial of degree
$n$, $Q$ is a polynomial of degree at most $n-1$, $P$ and $Q$ are co-prime and $P(0)\ne 0$. 

\bp\label{BMdesc}
The space of boundary measurements associated with the network  $N_{u,v}^\circ$ is dense in $\RR_n$.
\ep

\begin{proof}
Let us first prove that any boundary measurement indeed belongs to $\RR_n$. 
By Proposition~\ref{Weylresponse}, the roots of $P$ are exactly the eigenvalues of $-X$, hence the degree of $P$ equals to $n$. Next, since  $M(\lambda)= \sum_{k=0}^\infty (-1)^{k}H_k  \lambda^{-k-1}$, the value of $M(\lambda)$ at infinity equals zero, and hence the degree of $Q$ is at most $n-1$. 

By Proposition~\ref{Weylresponse} and~\eqref{weyl1}, the coprimality statement is equivalent to saying that $X$ and its
submatrix obtained by deleting the first row and column have no common
eigenvalues. Suppose this is not true, and $\tilde\lambda$ is a common
eigenvalue. Denote $\tilde X = X -\tilde\lambda \one$. Then $\det \tilde X =
\tilde X_{[2,n]}^{[2,n]} = 0$,
and, by the Jacobi determinantal identity,
$$ 
\tilde X_{[1,n-1]}^{[2,n]} \tilde X_{[2,n]}^{[1,n-1]} = 
\tilde X_{[2,n-1]}^{[2,n-1]}\det \tilde X + \tilde X_{[2,n]}^{[2,n]}\tilde X_{[1,n-1]}^{[1,n-1]}=0.
$$
 Thus either $\tilde X_{[2,n]}^{[1,n-1]}$ or  $\tilde X_{[1,n-1]}^{[2,n]}$ is
zero. Assume the latter is true (the other case can be treated similarly).
Consider the classical adjoint $\wh{X}$ of $\tilde X$. Since $\tilde X$ is
degenerate, $\wh{X}$ has rank one. Since $\wh{X}_{11}=\wh{X}_{1n}=0$, either the first row or the first column of $\wh{X}$ has
all zero entries. Assume the latter
is true. This means
that every $(n-1)\times (n-1)$
minor based on the last $n-1$ rows of $\tilde X$ equals zero, and so the $n \times
(n-1)$ submatrix of $\tilde X$
obtained by deleting the first column does not have the full rank. Then there
is a non-zero vector $w$ with
$w_1=0$ such that $\tilde X w = 0$. Therefore, $w$ is linearly independent with the eigenvector 
$(p_i^+(\lambda))_{i=1}^n$ of $X$ constructed  in Corollary~\ref{polycorr}(ii),
whose first component is
equal to $1$. We conclude that the dimension of the eigenspace of $X$ corresponding to
$\tilde\lambda$ is greater than one. However, due to Lemma
\ref{flaglemma} and invertibility of $X$, $e_1$ is a cyclic vector for $X$,
which implies that all eigenspaces of $X$ are one-dimensional. This
completes the proof of coprimality by contradiction. The case when the first row of $\wh{X}$ is zero can be treated similarly.

To prove that $P(0)\ne 0$, we denote $P(\lambda)=\lambda^n +p_{n-1}\lambda^{n-1}+ \cdots + p_0$. Then relation $Q(\lambda)=M(\lambda) P(\lambda)$ yields 
\be 
\label{p-h}
\sum_{i=0}^{n} (-1)^i p_i H_{k+i} = 0,\quad  k\ge 0,
\ee
with $p_n=1$. Relations~\eqref{p-h} for $k\in [0,n-1]$ provide a system of linear equations for $p_0,\dots, p_{n-1}$. The determinant of this system equals $H_0^n\Delta_n^{(n-1)}$. It is well-known (see, e.g. Theorem 8.7.1 in \cite{fuhr}), that
the co-primality of $P$ and $Q$ is equivalent to the non-vanishing of $\Delta_n^{(n-1)}$. So, $p_0=P(0)$ can be restored uniquely as
$(-1)^n\Delta_n^{(n)}/\Delta_n^{(n-1)}$, which by Remark~\ref{detiden} is equal to $(-1)^n\det X$. It remains to recall that $\det X=d_1\cdots d_n\ne 0$.

The density statement follows easily from Theorem~\ref{invthm}: given $M(\lambda)$, one builds Hankel determinants~\eqref{det} and makes use of formulas~\eqref{cd} to restore $X$, provided $H_0$ and all determinants in the denominator do not vanish.
\end{proof}

\br\label{handH}
{\rm
(i) Since $p_0\ne 0$, equations (\ref{p-h}) extended to $k=-1,-2, \ldots$ can be used as a recursive definition of $H_{-1}, H_{-2}, \ldots $.

(ii) Since $m(-\lambda; X) = \frac{q(\lambda)}{p(\lambda)}$, where $p(\lambda) =
\sum_{i=0}^n (-1)^i p_i \lambda^i $ is the characteristic polynomial of $-X$, the Cayley-Hamilton
theorem implies that for any $k
\in \mathbb{Z}$, 
$$
\sum_{i=0}^n (-1)^i p_i h_{k+i} = ( \sum_{i=0}^n (-1)^i p_i X^{k+i} e_1,
e_1) = (X^k p(-X) e_1, e_1)=0. 
$$
Therefore, 
\be\label{hH}
H_k=h_kH_0
\ee
for any $k\in \Z$.

(iii) Denote $Q(\lambda)=q_{n-1}\lambda^n + \cdots + q_0$. Similarly to~\eqref{p-h} one gets
\be 
\label{q-h}
(-1)^{j+1}q_j = \sum_{i=j+1}^{n} (-1)^i p_{i} H_{i-j-1}, \quad j\in [0,n-1].
\ee
}
\er

The following proposition is a particular case of Theorem 3.1 in \cite{GSV4}.

 \bp 
 \label{brackmeas}
 The standard  Poisson bracket on $\mathcal F_{N_{u,v}^\circ}$ induces a Poisson bracket on $\RR_n$.
 This bracket is given by
\begin{equation}
{\displaystyle
\{M(\lambda), M(\mu)\} = -  \left ( \lambda\ M(\lambda)
- \mu\ M(\mu)\right )
\frac{M(\lambda)  - M(\mu)
}{\lambda-\mu}. }
\label{brackrat}
\end{equation}
\ep

\br
{\rm Using Proposition~\ref{Weylresponse}, one can deduce from (\ref{brackrat}) the Poisson brackets for the Weyl function
$m(\lambda)=m(\lambda; X)$:
\begin{equation}
{\displaystyle
\{m(\lambda), m(\mu)\} = - \left ( \lambda m(\lambda)
- \mu m(\mu)\right )
\left ( \frac{m(\lambda)  - m(\mu)
}{\lambda-\mu} + m(\lambda) m(\mu)\right ). }
\label{brackm}
\end{equation}
(The derivation of (\ref{brackm}) from (\ref{brackrat}) can be found in \cite{FayGekh2}, Proposition~3.)
Thus a combination of Theorem 2.1 and Propositions~\ref{Weylresponse} and~\ref{brackmeas} provides a network-based proof of the fact that the standard Poisson--Lie structure on $GL_n$ induces the Poisson bracket  (\ref{brackm}) on Weyl functions. This fact plays a useful role in the study of a multi-Hamiltonian structure of Toda flows.
}
\er

\subsection{} 
 To compute face weights in terms of factorization parameters $c_i^\pm, d_i$ , we introduce new notation that makes
 formulas (\ref{cd}) more convenient. First of all, for any $l\in \mathbb Z$, $i\in \mathbb N$ define, similarly to~\eqref{det}, 
Hankel determinants
\begin{equation}
\DDelta^{(l)}_i=\det (H_{\alpha +\beta + l - i - 1})_{\alpha,\beta=1}^i;
\label{ddet}
\end{equation}
we assume that $\DDelta^l_0=1$ for any $l\in \mathbb Z$. It follows from \eqref{hH} that $\DDelta^{(l)}_i=H_0^i\Delta^{(l)}_i$.
 
Let us fix a pair of Coxeter elements $(u, v)$ and denote
 $\varepsilon= (\varepsilon_i)_{i=1}^{n}$ and
 \be
x_{0i}=x_{0i}(\varepsilon)= \DDelta_{i}^{(\varkappa_{i})}, \quad x_{1i}=x_{1i}(\varepsilon)= \DDelta_{i}^{(\varkappa_{i}+1)},\quad 
c_i= c_i^+ c_{i}^-. 
\label{shorthand}
\ee
Then formulae (\ref{cd}) become
\be
\displaystyle{ d_i =\frac{x_{1i} x_{0 i-1}}{x_{0i} x_{1 i-1}},\quad 
c_i =\frac{x_{0 i-1}x_{0 i+1}}{x_{1 i}^2}
\left ( \frac{x_{1 i+1}}{x_{0 i+1}}  \right )^{\varepsilon_{i+1}} 
\left (\frac{x_{1 i-1}}{x_{0 i-1}}  \right )^{2-\varepsilon_{i}}.
} 
\label{cdshort}
\ee
 
 We now compute the  face weights  for $N_{u,v}^\circ$:
\be\label{faceweights}
\begin{split}
y_{0 i} &= c_i^{-1} = x_{0 i-1}^{1-\varepsilon_{i}}x_{0 i+1}^{\varepsilon_{i+1}-1}
x_{1i}^2 x_{1 i-1}^{\varepsilon_i-2} x_{1 i+1}^{-\varepsilon_{i+1}}, \\ 
y_{1 i} &= \frac{c_i d_i}{d_{i+1}} = 
x_{1 i-1}^{1-\varepsilon_{i}}x_{1 i+1}^{\varepsilon_{i+1}-1}
x_{0i}^{-2} x_{0 i-1}^{\varepsilon_i} x_{0 i+1}^{2-\varepsilon_{i+1}},
\end{split}
\ee
for $i\in [1, n-1]$ and
\be\label{faceweights0}
y_{0 0} = \frac{1}{H_0} = x_{01}^{-1}, \quad y_{1 0} = \frac{H_0^2}{H_1}= x_{01}^{2} x_{11}^{-1}, \quad 
y_{0n}= d_n = x_{1 n} x_{0 n-1}
x_{0n}^{-1} x_{1 n-1}^{-1}.
\ee
 
 We will re-write (\ref{faceweights}) for $i=n-1$ in a slightly different way. Recall that $\varepsilon_{n}=0$. Due to (\ref{shift}), 
\begin{equation}\label{perver}
\begin{split} 
 x_{0 n}^{2} x_{1 n}^{-1}&=
\DDelta_{n}^{(\varkappa_n-1)}= \DDelta_{n}^{(n-1)} (\det X)^{\varkappa_n -n},\\ 
x_{0 n}& =
\DDelta_{n}^{(\varkappa_n)}= \DDelta_{n}^{(n-1)} (\det X)^{\varkappa_n -n +1 }. 
\end{split}
\end{equation}
Thus, 
(\ref{faceweights}) yields
\begin{equation*}
\begin{split}
y_{0 n-1} &= x_{0 n-2}^{1-\varepsilon_{n-1}}
x_{1n-1}^2 x_{1 n-2}^{\varepsilon_{n-1}-2}\left (\DDelta_{n}^{(n-1)}\right )^{-1} \left (\frac{1}{\det X}\right )^{\varkappa_n -n +1}, \\ 
y_{1 n-1} &= 
x_{1 n-2}^{1-\varepsilon_{n-1}}
x_{0 n-1}^{-2} x_{0 n-2}^{\varepsilon_{n-1}} \DDelta_{n}^{(n-1)}\left (\frac{1}{\det X}\right )^{n- \varkappa_n}.
\end{split}
\end{equation*}
 
 Define 
 \be\label{initclust}
 \x =\x (\varepsilon)=(x_i)_{i=1}^{2n}= \left( x_{01},  x_{11}, \ldots, x_{0 n-1},  x_{1 n-1}, \DDelta_{n}^{(n-1)}, 
 \frac{\DDelta_n^{(n-2)}}{\DDelta_n^{(n-1)}}=\frac{1}{\det X} \right)
\ee 
and $\y=\y (\varepsilon)= (y_i)_{i=1}^{2n}=  ( y_{00},  y_{10}, \ldots, y_{0 n-1},  y_{1n-1} )$. Then $ y_i = \prod_{i=1}^{2n} x_j^{a_{ij}}$, where
$A= (a_{ij})_{i=1}^{2n}$ is an $n\times n$ block lower--triangular matrix with $2\times 2$ blocks:
\begin{equation*}
A= \left (
\begin{array} {ccccc}
\ V_1 & 0 & 0 & 0 & 0\\
\ U & V_2 & 0 & 0 & 0\\
-V_2^T & U & V_3 & 0 & 0\\
 0 & \ddots & \ddots & \ddots & 0\\
 0 & 0 &-V_{n-1}^T & U & V_n 
\end{array}
\right )
\end{equation*}
with
\be
\label{Xtofaceblocks}
\begin{gathered}
U= \left (
\begin{array} {cc}
 0 & 2 \\
-2 & 0 
\end{array}
\right ), \quad
V_1= \left (
\begin{array} {cc}
 -1 & 0 \\
2 & -1 
\end{array}
\right ), \quad
V_n= \left (
\begin{array} {cc}
 -1 & \varkappa_n -n +1 \\
1  & n- \varkappa_n 
\end{array}
\right ),\\
V_i= \left (
\begin{array} {cc}
\varepsilon_i -1 & -\varepsilon_i \\
2 -\varepsilon_i  & \varepsilon_i  -1 
\end{array}
\right ),\quad  i\in [2, n-1].
\end{gathered}
\ee
The matrix $A$ is invertible, since $\det V_i=1$, $i\in [1,n-1]$ and $\det V_n=-1$.
 
 \br{\rm
Note that the expression for $x_{2n}$ in terms of face weights is independent of $\varepsilon$:
 $$
 x_{2n}= (y_{00}y_{10})^n  (y_{01}y_{11})^{n-1} \cdots (y_{0n-1}y_{1n-1}).
 $$
  }
  \er

 \subsection{} Given a pair of Coxeter elements $(u,v)$, we want to define a cluster algebra with the
compatible Poisson bracket given by~\eqref{brackrat}. 
To this end, we use the strategy developed in \cite{GSV1}. The first step
consists in finding a coordinate system on $\RR_n$ 
such that written in terms of their logarithms, the Poisson bracket~\eqref{brackrat} becomes constant.
Having in mind Proposition~\ref{lostprop}, we require this coordinate system to be given by a collection of regular functions on $\RR_n$.
Clearly, $H_i$, $i\ge 0$, are regular on $\RR_n$, and hence so are $\DDelta_n^{(n-1)}$ and $\DDelta_n^{(n)}$. 
Besides, it was explained in the proof of Lemma~\ref{BMdesc} that $\DDelta_n^{(n-1)}$ and $\DDelta_n^{(n)}$ 
do not vanish on $\RR_n$, hence $(\DDelta_n^{(n-1)})^{-1}$ and $(\DDelta_n^{(n)})^{-1}$ are regular as well. Consequently, by
Remark~\ref{handH}(i), $H_i$ are regular functions on $\RR_n$ for $i<0$, and hence so are Hankel determinants~\eqref{ddet} for any $l\in \Z$, $i\in \mathbb{N}$.
Therefore, components of $\x(\varepsilon)$ are regular functions on $\RR_n$ and they are connected by
an invertible monomial transformation to face weights $\y(\varepsilon)$
that satisfy Poisson relations (\ref{taubracks}), \eqref{taubracks0} of the required kind.
Therefore we can use $\x(\varepsilon)$ as an initial cluster. Now, following \cite{GSV1}, we have to compute
the matrix that defines cluster transformations, based on the coefficient matrix of the bracket \eqref{brackrat}.
 
Define a $2n\times 2n$ matrix
\be\label{bepsilon}
B(\varepsilon) =- \left (
\begin{array} {cccc}
\  U & V_2 & 0 & 0 \\
-V_2^T & U & V_3 & 0 \\
 0 & \ddots & \ddots & \ddots \\
  0 & 0 & -V_{n}^T & -\frac12 U 
\end{array}
\right )
\ee
 with $2\times 2$ block coefficients given by (\ref{Xtofaceblocks}). 
Denote by $\tilde B(\varepsilon)$ the $(2n-2)\times 2n$ submatrix of $B(\varepsilon)$ formed by the first $2n-2$ 
rows and consider the cluster algebra  $\A_\varepsilon$ of rank $2n-2$ with the initial seed
 $\Sigma(\varepsilon)=(\x(\varepsilon),\tilde B(\varepsilon))$,  
so that $x_i$, $i\in [1,2n-2]$, are 
cluster variables and $x_{2n-1}$,  $x_{2n}$ are stable variables.

\begin{lemma}
\label{cmptbl}
Poisson structure {\rm \eqref{brackrat}} is compatible with the cluster algebra $\A_\varepsilon$.
\end{lemma}
 
\begin{proof} 
 Let us first revisit standard Poisson structure on $\mathcal F_{N_{u,v}^\circ}$ described by
(\ref{taubracks}),~\eqref{taubracks0}. 
It is easy to see that in terms of the components of the vector $\y=\y(\varepsilon)$, 
this bracket can be written as $\{ y_i, y_j \}=  \omega_{ij} y_i y_j$, where
the matrix $\Omega=\Omega(\varepsilon)= (\omega_{ij})_{i,j=1}^{2n}$ is given by
\be
\label{Omega}
\Omega = \left (
\begin{array} {cccc}
\ \frac{1}{2} U & V_1 & 0 & 0 \\
-V_1^T & U & V_2 & 0 \\
 0 & \ddots & \ddots & \ddots \\
  0 & 0 & -V_{n-1}^T & U 
\end{array}
\right ). 
\ee
Therefore, the matrix of coefficients of the Poisson bracket (\ref{brackrat}) written in coordinates $\x(\varepsilon)$ is 
$\Omega^\x=A^{-1} \Omega (A^T)^{-1}$. 

Note that $\Omega$ defined by~\eqref{Omega} is invertible.
 To see that, observe that the block-entries of
$\Omega$ satisfy relations
$ V^T_i U V_i = U$,  $i\in [1, n-1]$, and $U^2=4\one_2$, which implies that $\Omega$ can be
factored as
\be\label{omegafactor}
\Omega = \left (
\begin{array} {cccc}
\one_2 & 0 & 0 & 0 \\
\frac{1}{2} V_1^T U & \one_2 & 0 & 0 \\
0 & \ddots & \ddots & \ddots \\
 0 & 0 & \frac{1}{2} V_{n-1}^T U & \one_2
\end{array}
\right )\
\left (
\begin{array} {cccc}
\ \frac{1}{2} U & V_1 & 0 & 0 \\
0 & \frac{1}{2} U & V_2 & 0 \\
0 & \ddots & \ddots & \ddots \\
 0 & 0 & 0 & \frac{1}{2} U
\end{array}
\right ).
\ee
Therefore, $\det \Omega =1$, and hence $\Omega^\x$ is invertible. 

To find its inverse $A^T\Omega^{-1}A$, observe that if we define 
$$
J = \left (
\begin{array} {cccc}
0 & \one_2 & 0 & 0 \\
0 & 0& \one_2 & 0 \\
 0 & \ddots & \ddots & \ddots \\
  0 & 0 & 0 & 0 
\end{array}
\right ),
$$
then $A = \Omega  J^T+V_n\otimes E_{nn}$ with $ E_{nn} = ( \delta_{i n} \delta_{j n})_{i,j=1}^n$. We then have
$$
A^T\Omega^{-1}A=  J \Omega^T J^T - J \left ( V_n\otimes E_{nn} \right ) + \left ( V_n^T\otimes E_{nn} \right ) J^T +
 \left ( V_n^T\otimes E_{nn} \right ) \Omega^{-1} \left ( V_n\otimes E_{nn} \right )=B(\varepsilon), 
$$
since by \eqref{omegafactor}, the lower-right $2\times 2$ block of  $\Omega^{-1}$ equals $-\frac12 U$ and $V_n^TUV_n=-U$.

So,  $B(\varepsilon)$ is non-degenerate and skew-symmetric. 
Thus we can invoke Theorem 1.4
of \cite{GSV1}. According to equation (1.5) in the proof of this theorem, compatibility will follow from the condition 
$\tilde B(\varepsilon)  \Omega^\x= (D\ \ 0 )$, where $D$ is a $(2n-2)\times (2n-2)$ diagonal matrix.
Since $B^{-1}(\varepsilon)= \Omega^\x$, this condition is obviously satisfied with $D=\one_{2n-2}$.
\end{proof}

Our goal is to prove

\bt 
{\rm (i)} The cluster algebra $\A_\varepsilon$ does  not depend on $\varepsilon$.

{\rm (ii)} The localization of the complex form of $\A_\varepsilon$ with respect to 
the stable variables $x_{2n-1}$, $x_{2n}$  is isomorphic to the ring
of regular functions on  $\RR_n$.
\label{clusterrat}
\et

\begin{proof}
First, we will compute cluster transformations~\eqref{exchange} of the initial cluster $\x(\varepsilon)$
in directions $(0i)$ and $(1i)$.
The transformed variables are denoted $\bar x_{0i}$ and $\bar x_{1i}$, respectively.
By~\eqref{bepsilon}, for $i\in [1,n-2]$ the transformations in question are determined  by the matrix
\[
\left(\begin{array}{cccccc}
\varepsilon_i-1 & 2- \varepsilon_i & 0  & -2 & 1-\varepsilon_{i+1} & \varepsilon_{i+1}\\
-\varepsilon_i  & \varepsilon_i-1  & 2  &  0 & \varepsilon_{i+1}-2 & 1-\varepsilon_{i+1}
\end{array}\right).
\]
Therefore, we have to consider the following cases.

{\bf Case 1:} $\varepsilon_{i}=\varepsilon_{i+1}= 0$. Then by (\ref{kappa}), $\varkappa_{i+1}= \varkappa_{i} + 1$ and $\varkappa_{i-1} =\varkappa_{i} - 1$,  so $x_{0 i}= \DDelta_{i}^{(\varkappa_i)}$ is transformed into
$$
\bar{x}_{0 i} = \frac{\DDelta_{i-1}^{(\varkappa_i-1)}\left ( \DDelta_{i}^{(\varkappa_i+1)} \right )^2 +  \DDelta_{i+1}^{(\varkappa_i+1)}\left ( \DDelta_{i-1}^{(\varkappa_i)} \right )^2}{\DDelta_{i}^{(\varkappa_i)}}.
$$
Using (\ref{HankJac}) and \eqref{hH}, we re-write the numerator as
\begin{equation*}
\begin{split}
\DDelta_{i-1}^{(\varkappa_i-1)}&\left ( \DDelta_{i}^{(\varkappa_i)} \DDelta_{i}^{(\varkappa_i+2)} - 
\DDelta_{i-1}^{(\varkappa_i+1)} \DDelta_{i+1}^{(\varkappa_i+1)} \right ) + 
\DDelta_{i+1}^{(\varkappa_i+1)}\left ( \DDelta_{i-1}^{(\varkappa_i)} \right )^2\\ 
&= \DDelta_{i}^{(\varkappa_i)}  \DDelta_{i-1}^{(\varkappa_i-1)} \DDelta_{i}^{(\varkappa_i+2)} +
\DDelta_{i+1}^{(\varkappa_i+1)} \left ( \left ( \DDelta_{i-1}^{(\varkappa_i)} \right )^2 - 
\DDelta_{i-1}^{(\varkappa_i-1)} \DDelta_{i-1}^{(\varkappa_i+1)} \right ) \\ 
&=\DDelta_{i}^{(\varkappa_i)} \left (  \DDelta_{i-1}^{(\varkappa_i-1)} \DDelta_{i}^{(\varkappa_i+2)}
- \DDelta_{i-2}^{(\varkappa_i)} \DDelta_{i+1}^{(\varkappa_i+1)} \right ),
\end{split}
\end{equation*}
and so
\be\label{10i}
\bar{x}_{0 i} = \DDelta_{i-1}^{(\varkappa_i-1)} \DDelta_{i}^{(\varkappa_i+2)}
- \DDelta_{i-2}^{(\varkappa_i)} \DDelta_{i+1}^{(\varkappa_i+1)}.
\ee

Similarly, $x_{1 i}= \DDelta_{i}^{(\varkappa_i+1)}$ is transformed into
$$
\bar{x}_{1 i} = \frac{\DDelta_{i-1}^{(\varkappa_i)}\left ( \DDelta_{i+1}^{(\varkappa_i+1)} \right )^2 +  \DDelta_{i+1}^{(\varkappa_i+2)}\left ( \DDelta_{i}^{(\varkappa_i)} \right )^2}{\DDelta_{i}^{(\varkappa_i+1)}},
$$
which can be re-written as
\be\label{11i}
\bar{x}_{1 i} = \DDelta_{i}^{(\varkappa_i-1)} \DDelta_{i+1}^{(\varkappa_i+2)}
- \DDelta_{i-1}^{(\varkappa_i)}\DDelta_{i+2}^{(\varkappa_i+1)}.
\ee

{\bf Case 2:} $\varepsilon_{i}=\varepsilon_{i+1}= 2$. This case is similar to Case 1. 
We have $\varkappa_{i+1} = \varkappa_{i} - 1$ and $\varkappa_{i-1} =  \varkappa_{i} +1$,
hence
\be\label{20i}
\bar{x}_{0 i} = \frac{\DDelta_{i-1}^{(\varkappa_i+1)}\left ( \DDelta_{i+1}^{(\varkappa_i)} \right )^2 +  \DDelta_{i+1}^{(\varkappa_i -1)}\left ( \DDelta_{i}^{(\varkappa_i+1)} \right )^2}{\DDelta_{i}^{(\varkappa_i)}}=
\DDelta_i^{(\varkappa_{i}+2)}\DDelta_{i+1}^{(\varkappa_i-1)}-\DDelta_{i-1}^{(\varkappa_i+1)}\DDelta_{i+2}^{(\varkappa_i)}
\ee
and 
\be\label{21i}
\bar{x}_{1 i} = \frac{\DDelta_{i+1}^{(\varkappa_i)}\left ( \DDelta_{i-1}^{(\varkappa_i+1)}\right )^2+
\DDelta_{i-1}^{(\varkappa_i+2)}\left ( \DDelta_{i}^{(\varkappa_i)} \right )^2}{\DDelta_{i}^{(\varkappa_i+1)}}
=
\DDelta_{i-1}^{(\varkappa_i+2)}\DDelta_i^{(\varkappa_{i}-1)}-\DDelta_{i-2}^{(\varkappa_i+1)}\DDelta_{i+1}^{(\varkappa_i)}.
\ee

{\bf Case 3:} $\varepsilon_{i}=0$, $\varepsilon_{i+1}= 2$. We have $\varkappa_{i+1} =\varkappa_{i-1} =  \varkappa_{i} - 1$, 
and so $x_{0i}$ is transformed into
$$
\bar{x}_{0 i} = \frac{\left ( \DDelta_{i-1}^{(\varkappa_i)} \DDelta_{i+1}^{(\varkappa_i)} \right )^2 + 
 \DDelta_{i-1}^{(\varkappa_i-1)}\ \DDelta_{i+1}^{(\varkappa_i-1)}\left ( \DDelta_{i}^{(\varkappa_i+1)} \right )^2}{\DDelta_{i}^{(\varkappa_i)}}.
$$
The numerator of the above expression can be re-written as
\begin{equation*}
\begin{split}
\left ( \DDelta_{i-1}^{(\varkappa_i)}\right. &\left.\DDelta_{i+1}^{(\varkappa_i)} \right )^2 + 
\left (  \DDelta_{i}^{(\varkappa_i)}\ \DDelta_{i}^{(\varkappa_i-2)} - \left ( \DDelta_{i}^{(\varkappa_i-1)} \right )^2\right )
\left ( \DDelta_{i}^{(\varkappa_i+1)} \right )^2\\
&= \left ( 
\left ( \DDelta_{i-1}^{(\varkappa_i)} \DDelta_{i+1}^{(\varkappa_i)} \right )^2 - 
\left ( \DDelta_{i}^{(\varkappa_i+1)} \DDelta_{i}^{(\varkappa_i-1)} \right )^2  
\right )
+    \DDelta_{i}^{(\varkappa_i)}\ \DDelta_{i}^{(\varkappa_i-2)}  
\left ( \DDelta_{i}^{(\varkappa_i+1)} \right )^2 \\
&=\DDelta_{i}^{(\varkappa_i)} 
\left (\DDelta_{i}^{(\varkappa_i-2)} 
\left (\DDelta_{i}^{(\varkappa_i+1)} \right )^2 -
\DDelta_{i}^{(\varkappa_i)} 
\left (\DDelta_{i-1}^{(\varkappa_i)} \DDelta_{i+1}^{(\varkappa_i)}  +  \DDelta_{i}^{(\varkappa_i+1)} \DDelta_{i}^{(\varkappa_i-1)}  \right ) 
\right ), 
\end{split}
\end{equation*}
and so
\be\label{30i}
\bar{x}_{0 i} =   \DDelta_{i}^{(\varkappa_i-2)} \left ( \DDelta_{i}^{(\varkappa_i+1)} \right )^2 -
\DDelta_{i}^{(\varkappa_i)} \left (  \DDelta_{i-1}^{(\varkappa_i)} \DDelta_{i+1}^{(\varkappa_i)}  +  \DDelta_{i}^{(\varkappa_i+1)} \DDelta_{i}^{(\varkappa_i-1)}  \right ).
\ee

On the other hand,
\be\label{31i}
\bar{x}_{1 i} = \frac{\DDelta_{i-1}^{(\varkappa_i)} \DDelta_{i+1}^{(\varkappa_i)}  +  \left ( \DDelta_{i}^{(\varkappa_i)} \right )^2}{\DDelta_{i}^{(\varkappa_i+1)}}= \DDelta_{i}^{(\varkappa_i-1)}.
\ee

{\bf Case 4:} $\varepsilon_{i}=2$, $\varepsilon_{i+1}= 0$. This case is similar to Case 3. 
We have $\varkappa_{i+1} =\varkappa_{i-1} =  \varkappa_{i} +1$,
hence
\be\label{40i}
\bar{x}_{0 i} = \frac{\DDelta_{i-1}^{(\varkappa_i+1)} \DDelta_{i+1}^{(\varkappa_i+1)}  +  \left ( \DDelta_{i}^{(\varkappa_i+1)} \right )^2}{\DDelta_{i}^{(\varkappa_i)}}=\DDelta_{i}^{(\varkappa_i+2)} 
\ee
and
\be\label{41i}
\begin{split}
\bar{x}_{1 i} &=\frac{\left ( \DDelta_{i-1}^{(\varkappa_i+1)} \DDelta_{i+1}^{(\varkappa_i+1)} \right )^2 + 
 \DDelta_{i-1}^{(\varkappa_i+2)}\ \DDelta_{i+1}^{(\varkappa_i+2)}\left ( \DDelta_{i}^{(\varkappa_i)} \right )^2}{\DDelta_{i}^{(\varkappa_i+1)}}\\
&=   \DDelta_{i}^{(\varkappa_i+3)} \left ( \DDelta_{i}^{(\varkappa_i)} \right )^2 -
\DDelta_{i}^{(\varkappa_i+1)} \left (  \DDelta_{i-1}^{(\varkappa_i+1)} \DDelta_{i+1}^{(\varkappa_i+1)}  +  \DDelta_{i}^{(\varkappa_i)} \DDelta_{i}^{(\varkappa_i+2)}  \right ).
\end{split}
\ee

{\bf Case 5:} $\varepsilon_{i}=1$, $\varepsilon_{i+1}= 2$. We have $\varkappa_{i+1}  =  \varkappa_{i} - 1$,  $\varkappa_{i-1}=\varkappa_{i}$, so $x_{0i}$ is transformed via~\eqref{20i}
and $x_{1i}$ via~\eqref{31i}.
 
 
{\bf Case 6:} $\varepsilon_{i}=2$, $\varepsilon_{i+1}= 1$. 
We have $\varkappa_{i+1} =\varkappa_{i}$, $\varkappa_{i-1} =  \varkappa_{i} +1$,
so $x_{0i}$ is transformed via~\eqref{40i} and $x_{1i}$ via~\eqref{21i}.

{\bf Case 7:} $\varepsilon_{i}=0$, $\varepsilon_{i+1}= 1$. 
We have $\varkappa_{i+1} =\varkappa_{i}$, $\varkappa_{i-1} =  \varkappa_{i} -1$,
so $x_{0i}$ is transformed via~\eqref{10i} and $x_{1i}$ via~\eqref{31i}.

{\bf Case 8:} $\varepsilon_{i}=1$, $\varepsilon_{i+1}= 0$. 
We have $\varkappa_{i+1} =\varkappa_{i}+1$, $\varkappa_{i-1} =  \varkappa_{i}$,
so $x_{0i}$ is transformed via~\eqref{40i} and $x_{1i}$ via~\eqref{11i}.

{\bf Case 9:} $\varepsilon_{i}=\varepsilon_{i+1}= 1$. We have $\varkappa_{i+1}=\varkappa_{i-1}=\varkappa_i$,
so $x_{0i}$ is transformed via~\eqref{40i} and $x_{1i}$ via~\eqref{31i}.

Now, let $i=n-1$. In this situation transformations of the initial cluster are determined by the matrix
\[
\left(\begin{array}{cccccc}
\varepsilon_{n-1}-1 & 2- \varepsilon_{n-1} & 0  & -2 & 1 & n-\varkappa_n-1\\
-\varepsilon_{n-1}  & \varepsilon_{n-1}-1  & 2  &  0 & -1 & \varkappa_n-n
\end{array}\right).
\]
Note that $\varkappa_n$ does not exceed $n-1$, so the last two elements in the first row are always nonnegative, and the last two elements in the second row are always negative. By~\eqref{perver}, they contribute to the corresponding relations
$\DDelta_{n}^{(\varkappa_n)}$ and $\DDelta_{n}^{(\varkappa_n-1)}$, respectively.
Therefore, we  have to consider the following cases.

{\bf Case 10:} $\varepsilon_{n-1}=0$. Then $\varkappa_{n-1} = \varkappa_{n} - 1= \varkappa_{n-2} + 1$, and  $x_{0 n-1}$
is transformed into
\begin{equation*}
\bar{x}_{0 n-1} 
=  \frac{\left ( \DDelta_{n-1}^{(\varkappa_n)} \right )^2 \DDelta_{n -2}^{(\varkappa_n -2 )} +   \DDelta_{n}^{(\varkappa_n)}\left ( \DDelta_{n-2}^{(\varkappa_n-1)} \right )^2}{\DDelta_{n-1}^{(\varkappa_n-1)}}; 
\end{equation*}
Similarly to Case 1, this gives~\eqref{10i} for $i=n-1$.

On the other hand,
\begin{equation*}
\bar{x}_{1 n-1} 
= \frac{\left ( \DDelta_{n-1}^{(\varkappa_n-1)} \right )^2 +   \DDelta_{n}^{(\varkappa_n-1)}\DDelta_{n-2}^{(\varkappa_n-1)} }{\DDelta_{n-1}^{(\varkappa_n)}}, 
\end{equation*}
which gives \eqref{31i} for $i=n-1$. We thus see that the transformations in this case are exactly the same as in Case 7.

{\bf Case 11:} $\varepsilon_{n-1}=1$. Then $\varkappa_{n-1} = \varkappa_{n} - 1= \varkappa_{n-2} $, and hence 
\begin{equation*}
\bar{x}_{0 n-1} 
= \frac{\left ( \DDelta_{n-1}^{(\varkappa_n)} \right )^2 +   \DDelta_{n}^{(\varkappa_n)}\DDelta_{n-2}^{(\varkappa_n)} }{\DDelta_{n-1}^{(\varkappa_n-1)}}, 
\end{equation*}
which gives \eqref{40i} for $i=n-1$. 

Similarly,
\begin{equation*}
\bar{x}_{1 n-1} 
= \frac{\left ( \DDelta_{n-1}^{(\varkappa_n-1)} \right )^2 +   \DDelta_{n}^{(\varkappa_n-1)}\DDelta_{n-2}^{(\varkappa_n-1)} }{\DDelta_{n-1}^{(\varkappa_n)}}, 
\end{equation*}
which gives \eqref{31i} for $i=n-1$. We thus see that the transformations in this case are exactly the same as in Case 9.

{\bf Case 12:} $\varepsilon_{n-1}=2$. Then $\varkappa_{n-1} = \varkappa_{n} - 1= \varkappa_{n-2} -1 $, and  
$x_{0 n-1}$ transforms exactly as in the previous case.

On the other hand,
$$
\bar{x}_{0 n-1} 
= \frac{\left ( \DDelta_{n-1}^{(\varkappa_n-1)} \right )^2 \DDelta_{n -2}^{(\varkappa_n +1 )} +   \DDelta_{n}^{(\varkappa_n-1)}\left ( \DDelta_{n-2}^{(\varkappa_n)} \right )^2}{\DDelta_{n-1}^{(\varkappa_n)}}. 
$$
Similarly to Case 2, this gives~\eqref{21i} for $i=n-1$. We thus see that the transformations in this case are exactly the same as in Case 6.

%

Let $(u',v')$ be an arbitrary pair of Coxeter elements, $\varepsilon'$ be the corresponding $n$-tuple built by~\eqref{eps} and~\eqref{epsum}.

\bl 
\label{Lem4.2}
For any Coxeter elements $u'$, $v'$, the seed $\Sigma(\varepsilon')=(\x(\varepsilon'),\tilde B(\varepsilon'))$ belongs to $\A_\varepsilon$. 
\el

\begin{proof} First we will show that, in certain cases, mutations of the seed $\Sigma(\varepsilon)$ 
transform it into a seed equivalent to $\Sigma(\varepsilon')$ for an appropriately chosen $\varepsilon'$. 
These situations are listed
in the table below. In this table, only the entries at which   $\varepsilon$ and $\varepsilon'$ differ are specified.
In the first four rows $i$ is assumed to be less than $n-1$. The second column describes the direction
of the seed mutation: under the mutation in direction $(s, i)$, the cluster variable $x_{s i}$ is being transformed. 
It should also be understood that each mutation is followed by the permutation of variables
with indices $(0,i)$ and $(1,i)$ in the new cluster, which results in permuting columns and rows
$2i-1$ and $2i$ in the matrix obtained via the corresponding matrix mutation. In particular,
if  $\x(\varepsilon')$ is obtained
from $\x(\varepsilon)$ via the cluster transformation in direction $(s, i)$, then $\x(\varepsilon)$ is obtained
from  $\x(\varepsilon')$ via the cluster transformation in direction $(1-s, i)$.

\begin{center}
\begin{tabular}{| l | c | r |}
  \hline
  $\varepsilon$ & Direction &  $\varepsilon'$ \\
    \hline 
  $\varepsilon_{i}=0, \varepsilon_{i+1}= 2$ & $(1,i)$ & $\varepsilon'_{i}=1, \varepsilon'_{i+1}= 1$ \\ 
  \hline
  $\varepsilon_{i}=2, \varepsilon_{i+1}= 0$ & $(0,i)$ & $\varepsilon'_{i}=1, \varepsilon'_{i+1}= 1$ \\ 
  \hline
  $\varepsilon_{i}=1, \varepsilon_{i+1}= 0$ & $(0,i)$ & $\varepsilon'_{i}=0, \varepsilon'_{i+1}= 1$ \\ 
  \hline
   $\varepsilon_{i}=2, \varepsilon_{i+1}= 1$ & $(0,i)$ & $\varepsilon'_{i}=1, \varepsilon'_{i+1}= 2$ \\ 
  \hline
  $\varepsilon_{n-1}= 0$ & $(0,n-1)$ & $\varepsilon'_{n-1}=1$ \\ 
  \hline
  $\varepsilon_{n-1}= 1$ & $(1,n-1)$ & $\varepsilon'_{n-1}=2$ \\ 
  \hline
\end{tabular}
\end{center}

We will only provide justification for rows one and five of the table. The remaining cases can be treated similarly. If $\varepsilon_{i}=0, \varepsilon_{i+1}= 2$, let $\varepsilon'$ be defined
by $\varepsilon'_{i}=1, \varepsilon'_{i+1}= 1$ and $\varepsilon'_{j}=\varepsilon_{j}$ for $j\ne i, i+1$
Then it is easy to check that the matrix
mutation in the direction $(1,i)$ followed by the permutation of rows and columns $2 i-1$ and $2 i$ 
transforms  $B(\varepsilon)$ into $B(\varepsilon')$. Note also that when $\varepsilon$ is replaced by
$\varepsilon'$, the corresponding sequence $\varkappa=(\varkappa_i)_{i=1}^n$ transforms into a sequence $\varkappa'$ that differs
from $\varkappa$ only in the component $\varkappa'_i=\varkappa_i - 1$. This means, that $\x'=\x(\varepsilon')$ differs from
$\x(\varepsilon)$ only in components $x_{0i}(\varepsilon')=\DDelta_i^{(\varkappa_i - 1)} = \bar x_{1i}(\varepsilon)$ (cf. \eqref{30i} in Case 3) and $x_{1i}(\varepsilon')=\DDelta_i^{(\varkappa_i)} = x_{0 i}(\varepsilon)$. Thus, we see that the seed mutation
in direction $(1,i)$ of the initial seed of $\A_\varepsilon$ transforms it into a seed equivalent to
the initial seed of $\A_{\varepsilon'}$. 

Now consider the case $\varepsilon_{n-1}= 0$, $\varepsilon'_{n-1}=1$. Then  $\varkappa_{n-1} = \varkappa_{n} -1$, $\varkappa_{n-1}' = \varkappa_{n-1} -1$ and $\varkappa_n' = \varkappa_n -1$. The fact that the matrix
mutation in direction $(1,n-1)$ followed by the permutation of rows and columns $2 n-1$ and $2 n$ 
transforms  $B(\varepsilon)$ into $B(\varepsilon')$ becomes easy to check once we recall that, by (\ref{kappa}),
$n -1 - \varkappa_n$ is always nonnegative. As was shown in Case 11 above, 
$\bar x_{1n-1}(\varepsilon)= \DDelta_{n-1}^{(\varkappa_n - 2)} = \DDelta_{n-1}^{(\varkappa'_{n - 1})}=x_{0 n-1}(\varepsilon')$. Also $\bar x_{0n-1}(\varepsilon)= \DDelta_{n-1}^{(\varkappa_{n-1})} = \DDelta_{n-1}^{(\varkappa'_{n - 1}+1)}=x_{1 n-1}(\varepsilon')$, which completes the check.

To complete the proof of the lemma, it suffices to show that, for any $\varepsilon'$, the seed  
$\Sigma(\varepsilon')$ is a seed in $\A_{\varepsilon^{(0)}}$ for $\varepsilon^{(0)}= (2,0,\ldots, 0)$. 
This can be done by induction
on $\sum_{i=2}^{n-1} \varepsilon'_i$. Indeed, if $\varepsilon'_{n-1}\ne 0$, then   
$\Sigma(\varepsilon')$ can be obtained via a single mutation from   
$\Sigma(\varepsilon)$, where $\varepsilon$ differs from $\varepsilon'$ only in the $(n-1)$st component: $\varepsilon_{n-1}=\varepsilon'_{n-1}-1$ (see the last two rows of the above table). Otherwise, if $i\in [2,  n-2] $ is the largest index
such that $\varepsilon'_{i}\ne 0$, then, using the table again, we see that 
$\Sigma(\varepsilon')$ can be obtained via a sequence of mutations from   
$\Sigma(\varepsilon)$, where $\varepsilon= (\varepsilon'_1,\ldots, \varepsilon'_{i-1}, \varepsilon'_{i}-1, 0, \ldots, 0)$. The intermediate  transformations of the $n$-tuple $\varepsilon$ in this case are
$(\varepsilon'_1,\ldots, \varepsilon'_{i-1}, \varepsilon'_{i}-1, 0, \ldots, 0, 1,0)$, $(\varepsilon'_1,\ldots, \varepsilon'_{i-1}, \varepsilon'_{i}-1, 0, \ldots, 1, 0, 0)$, $\ldots$, $(\varepsilon'_1,\ldots, \varepsilon'_{i-1}, \varepsilon'_{i}-1, 1,0,  \ldots, 0,0)$.
\end{proof}

The first statement of Theorem~\ref{clusterrat} follows immediately.
We can now drop the dependence on $\varepsilon$ in the cluster algebra $\A_\varepsilon$ and denote it 
simply by $\A$.

\bl
\label{specvar}
For any $j\in [1, n-1]$, $k \in \Z$, the function
\be
\label{specclust}
x(j,k)= \DDelta_j^{(k)} x_{2n}^{\max(0, k+ j +1 - 2n)}
\ee
is a cluster variable in $\A$.
\el

\begin{proof} Consider the cluster $\x=\x(\varepsilon)$ that corresponds to $\varepsilon=(2, 0, \ldots, 0)$. 
In this case $x_{2j-1}=\DDelta_j^{(j-1)}$, $j\in [1,n]$, and
$x_{2j}=\DDelta_j^{(j)}$, $j\in [1, n-1]$.
The matrix $B(\varepsilon)$ can be conveniently represented by a planar graph $\Gamma$, whose vertices are represented by
nodes of a $2\times n$ rectangular grid. Vertices in the top row (listed left to right) correspond to cluster
variables $x_1, x_3, \ldots, x_{2n-1}$, and vertices in the bottom row correspond to cluster
variables $x_2, x_4, \ldots$, $x_{2n}$. We will label the $j$th vertex  in the $s$th row by  $(s,j)$, $s=0,1$, $j\in [2,n]$
($s =0$ corresponds to the top row, and $s=1$ to the bottom row).
In accordance with (\ref{bepsilon}), $\Gamma$ has  edges $(i,j) \to (s,j-1)$ for $s=0,1$ and any $j\in [2, n-1]$, edges 
$(0,n) \to (0,n-1)$,  $(1,n-1) \to (1,n)$, $(1,n-1) \to (0, n)$, $(1,n) \to (0, n)$ and
double edges $(0,j) \to (1,j)$ for $j\in [1, n-1]$ and $(1,j) \to (0,j+1)$ for $j\in [1, n-2]$: 
\be\label{ratgraph}
\begin{xy}
*!C\xybox{
\xymatrix{
\bullet\ar@{=>}[d]&\bullet\ar[l]\ar@{=>}[d]  &\bullet\ar[l]\ar@{=>}[d] 
&\bullet\ar[l]|{{\phantom{.}\cdots\phantom{.}}}\ar@{=>}[d] &\bullet\ar[l]\\
\bullet\ar@{=>}[ur] & \bullet\ar[l]\ar@{=>}[ur] & \bullet\ar[l]\ar@{=>}[ur]|{\stackrel{\cdots}{\cdots}}     
&\bullet\ar[l]|{\ \cdots\ }\ar[ur]\ar[r]&\circ\ar[u]
} }\end{xy}
\ee

We marked the vertex that corresponds to the stable variable $x_{2n}$ differently, as it plays a special
role in what follows. In particular, we will occasionally perturb a two-row structure of transformations of the graph $\Gamma$ by ``moving around'' the white vertex. Note also that, in view of the definition 
(\ref{specclust}),   the cluster variable associated with the vertex $(s, j)$ of $\Gamma$ is
$x(j, s+j -1)$ for $j\in [1, n-1]$.

Denote by $T_p$ the cluster transformation in direction $p$. Let us consider the result of the composition 
$$
T = T_{2n-2}\circ T_{2n-4}\circ\cdots\circ T_4\circ T_2 \circ T_{2n-3}\circ\cdots\circ T_3\circ T_1.
$$
An application of $T_1$ transforms $x_1=H_0$ into $\tilde x_1=\frac{1}{H_0} ( H_1^2 + \DDelta_2^{(1)}) = H_2$
and the graph $\Gamma$ into 
$$
\begin{xy}
*!C\xybox{
\xymatrix{
\bullet\ar[r]&\bullet\ar@{=>}[d]  &\bullet\ar[l]\ar@{=>}[d] &\bullet\ar[l]|{\ \cdots\ }\ar@{=>}[d] &\bullet\ar[l]\\
\bullet\ar@{=>}[u]& \bullet\ar[l]\ar@{=>}[ur] & \bullet\ar[l]\ar@{=>}[ur]|{\stackrel{\cdots}{\cdots}} 
&\bullet\ar[l]|{\ \cdots\ }\ar[ur]\ar[r]&\circ\ar[u]
} }\end{xy}
$$
Next, an application of $T_3$ transforms $x_3$ into $\tilde x_3=\DDelta_2^{(3)}$ (here we use (\ref{HankJac}) with $i=l=2$)
and the graph $\Gamma$ into 
$$
\begin{xy}
*!C\xybox{
\xymatrix{
\bullet\ar@{=>}[dr]&\bullet\ar[l]\ar[r]  &\bullet\ar@{=>}[d] &\bullet\ar[l]|{\ \cdots\ }\ar@{=>}[d] &\bullet\ar[l]\\
\bullet\ar@{=>}[u]& \bullet\ar[l]\ar@{=>}[u] & \bullet\ar[l]\ar@{=>}[ur]|{\stackrel{\cdots}{\cdots}}
&\bullet\ar[l]|{\ \cdots\ }\ar[ur]\ar[r]&\circ\ar[u]
} }\end{xy}
$$
Continuing in the same fashion and using on the $j$th step relation (\ref{HankJac}) with $i=l=j$, we conclude that an application of $T_{2n-3}\circ\cdots\circ T_3\circ T_1$ to the initial cluster transforms $\Gamma$ into 
$$
\begin{xy}
*!C\xybox{
\xymatrix{
\bullet\ar@{=>}[dr]&\bullet\ar[l]\ar@{=>}[dr] &\bullet\ar[l] \ar@{=>}[dr]|{\stackrel{\cdots}{\cdots}}
&\bullet\ar[l]|{\ \cdots\ }\ar[r] &\bullet\ar[dl]\\
\bullet\ar@{=>}[u]& \bullet\ar[l]\ar@{=>}[u] & \bullet\ar[l]\ar@{=>}[u]&\bullet\ar[l]|{\ \cdots\ }\ar@{=>}[u]\ar[r]&\circ\ar[u]
} }\end{xy}
$$
 with the variable $x_{2i-1}$ replaced with $\tilde x_{2i-1}=\DDelta_i^{(i+1)}$ for all $i\in [1,n-1]$.
Similarly, the subsequent application of $T_{2n-4}\circ\cdots\circ T_4\circ T_2$ transforms $\Gamma$ into
$$
\begin{xy}
*!C\xybox{
\xymatrix{
\bullet\ar@{=>}[d]
&\bullet\ar[l]\ar@{=>}[d]  &\bullet\ar[l]\ar@{=>}[d]&\bullet\ar[l]|{\ \cdots\ }\ar@{=>}[d] 
&\bullet\ar[l]\ar[r] &\bullet\ar[dl]\\
\bullet\ar@{=>}[ur] & \bullet\ar[l]\ar@{=>}[ur] & \bullet\ar[l]\ar@{=>}[ur]|{\stackrel{\cdots}{\cdots}}
&\bullet\ar[l]|{\ \cdots\ }\ar[r]&\bullet\ar@{=>}[u]\ar[r]&\circ\ar[u]
} }\end{xy}
$$
and replaces $x_{2i}$ with $\tilde x_{2i}=\DDelta_i^{(i+2)}$ for all $i\in [1,n-2]$. Finally, $T_{2n-2}$ transforms $x_{2n-2}=\DDelta_{n-1}^{(n-1)}$ into 
\begin{multline*}
\tilde x_{2n-2} = \left (\DDelta_{n-1}^{(n-1)}\right )^{-1} \left ( x_{2n}  \left (\DDelta_{n-1}^{(n)}\right )^{2}  +\DDelta_{n-2}^{(n)}\DDelta_{n}^{(n-1)}\right )\\ 
=\left (\DDelta_{n-1}^{(n-1)}\right )^{-1} x_{2n}  \left (  \left (\DDelta_{n-1}^{(n)}\right )^{2}  +\DDelta_{n-2}^{(n)}\DDelta_{n}^{(n)}\right ) = x_{2n} \DDelta_{n-1}^{(n+1)},
\end{multline*}
where we used (\ref{shift}). The corresponding transformation of the graph $\Gamma$ is
\be
\label{ratgraph1}
\begin{xy}
*!C\xybox{
\xymatrix{
\bullet\ar@{=>}[d]
&\bullet\ar[l]\ar@{=>}[d]  &\bullet\ar[l]\ar@{=>}[d]&\bullet\ar[l]|{\ \cdots\ }\ar@{=>}[d] 
&\bullet\ar[l]\ar@{=>}[d] &\bullet\ar[l]\\
\bullet\ar@{=>}[ur] & \bullet\ar[l]\ar@{=>}[ur] & \bullet\ar[l]\ar@{=>}[ur]|{\stackrel{\cdots}{\cdots}}
&\bullet\ar[l]|{\ \cdots\ }\ar@{=>}[ur]\ar[dr]&\bullet\ar[l]\ar[ur]&\\
&&&&\circ\ar[u]&
} }\end{xy}
\ee
To  summarize, $T$ results in the
the transformation $x(j,s+j -1)\mapsto x(j,i+j +1)$ for $s=0,1$, $j\in [1, n-1]$, and in replacing the initial graph $\Gamma$
(see (\ref{ratgraph})) with $T(\Gamma)$ (see (\ref{ratgraph1})). Observe also that the subgraphs of $\Gamma$ and $T(\Gamma)$ spanned by black vertices coincide.

Arguing in exactly the same fashion, we deduce that for $r=1,\ldots, n-2$, an application of $T^r$ results
in a cluster $T^r(\x)$ with the corresponding graph $T^r(\Gamma)$ such that (i) the  subgraphs of $\Gamma$ and $T^r(\Gamma)$ 
spanned by black vertices coincide and (ii) the white vertex is connected by simple edges to vertices $(1, n-r)$, $(1, n-r-1)$ so as to form a cyclically oriented triangle. In particular, the graph associated with $T^{n-2}(\x)$ is
$$
\begin{xy}
*!C\xybox{
\xymatrix{
\bullet\ar@{=>}[d]
&\bullet\ar[l]\ar@{=>}[d]  &\bullet\ar[l]\ar@{=>}[d]&\bullet\ar[l]|{\ \cdots\ }\ar@{=>}[d] 
&\bullet\ar[l]\ar@{=>}[d] &\bullet\ar[l]\\
\bullet\ar@{=>}[ur]\ar[dr] & \bullet\ar[l]\ar@{=>}[ur] & \bullet\ar[l]\ar@{=>}[ur]|{\stackrel{\cdots}{\cdots}}
&\bullet\ar[l]|{\ \cdots\ }\ar@{=>}[ur]&\bullet\ar[l]\ar[ur]&\\
&\circ\ar[u]&&&&
} }\end{xy}
$$
Furthermore, the cluster variable in $T^r(\x)$ associated with the vertex $(s,j)$,
$s=0,1$, $j\in [1, n-1])$, in $T^r(\Gamma)$ is $x(j,s+j -1 + 2 r)$. This claim relies on repeated applications
of relations
\bea
\nonumber
&x(1, k-1) x(1, k+1) =  x_{2n}^{\delta_{k+2-2n,0}}x(1, k)^2 +  x(2, k), \\
\nonumber
&x(j, k-1) x(j, k+1) =  x_{2n}^{\delta_{k+j+1-2n,0}}x(j, k)^2 + x(j-1, k) x(j+1, k),\quad j\in[2,n-2], \\
&x(n-1, k-1) x(n-1, k+1) =  x_{2n}^{\delta_{k-n,0}}x(n-1, k)^2 +  x_{2n}^{\max(0,k+1-n)} x(n-2, k-1) \DDelta_n^{(n-1)},
\nonumber
\eea
which, in turn, follow easily from (\ref{specclust}), (\ref{HankJac}), (\ref{shift}).

The same pattern of transformations for cluster variables
remains valid also for $r\ge n-1$. However,  the graph associated 
 with  $T^{n-1}(\x)$   has a form
$$
\begin{xy}
*!C\xybox{
\xymatrix{
&\bullet\ar@{=>}[d]
&\bullet\ar[l]\ar@{=>}[d]  &\bullet\ar[l]\ar@{=>}[d]&\bullet\ar[l]|{\ \cdots\ }\ar@{=>}[d] 
&\bullet\ar[l]\ar@{=>}[d] &\bullet\ar[l]\\
\circ\ar[r]&\bullet\ar@{=>}[ur] & \bullet\ar[l]\ar@{=>}[ur] & \bullet\ar[l]\ar@{=>}[ur]|{\stackrel{\cdots}{\cdots}}
&\bullet\ar[l]|{\ \cdots\ }\ar@{=>}[ur]&\bullet\ar[l]\ar[ur]&
} }\end{xy}
$$
and the graph associated 
with  $T^r(\x)$  for $r > n-1$ has a form
$$
\begin{xy}
*!C\xybox{
\xymatrix{
&\bullet\ar@{=>}[d]
&\bullet\ar[l]\ar@{=>}[d]  &\bullet\ar[l]\ar@{=>}[d]&\bullet\ar[l]|{\ \cdots\ }\ar@{=>}[d] 
&\bullet\ar[l]\ar@{=>}[d] &\bullet\ar[l]\\
\circ\ar[ur]^{b_r}&\bullet\ar[l]^{a_r}\ar@{=>}[ur] & \bullet\ar[l]\ar@{=>}[ur] & \bullet\ar[l]\ar@{=>}[ur]|{\stackrel{\cdots}{\cdots}}
&\bullet\ar[l]|{\ \cdots\ }\ar@{=>}[ur]&\bullet\ar[l]\ar[ur]&
} }\end{xy}
$$
where multiplicities $a_r, b_r$ are given by $a_r = 2(r-n)+1$, $b_r= 2 (r-n) +2$. Thus we have shown that $x(j, l+ j-1)$ is a cluster variable in $\A$ for any $j\in [1,n-1]$, and $l \ge 0$.

To recover $x(j, 1- j - l)$ for  $j\in [1,n-1]$, $l \ge 0$, we act in a similar way, starting with the cluster corresponding
to $\varepsilon=(2, 2, \ldots, 2, 0)$ and repeatedly applying a composition of cluster transformations
$$
T_{2n-3}\circ\cdots\circ T_3\circ T_1\circ T_{2n-2}\circ\cdots\circ T_4\circ T_2.
$$
Thus, $x(j, k)$ is a cluster variable in $\A$ for $k \in \Z \setminus [2-j, j-2]$. 
To complete the proof, it suffices to notice that by~\eqref{kappa}, the range of possible values for $\varkappa_j$ is $[1-j, j-1]$,
 and thus  for any $k \in [2-j, j-2]$ there exists a cluster $\x(\varepsilon)$ given by~\eqref{initclust} such that $x(j, k)$ is one of its variables. Therefore,  this variable belongs to $\A$ by Lemma~\ref{Lem4.2}.
\end{proof}

Since $x(1,k)=H_k$ for $k\leq 2n - 2$, the following statements readily apparent from Lemma~\ref{specvar}.

\begin{coroll} 
\label{corH}
For any $k\leq 2n - 2$, $H_k$ is a cluster variable in $\A$. 
\end{coroll}

To prove the second statement of Theorem~\ref{clusterrat}, we would like to apply Proposition~\ref{lostprop} in the situation when $V=\RR_n$, $\A$ is the cluster algebra discussed above and $\x$ is given by~\eqref{initclust}. Clearly $\RR_n$ is Zariski open in $\C^{2n}$, as the complement to the union of hypersurfaces $\DDelta_n^{(n-1)}=0$ and $p_0=0$. Condition (ii) is satisfied by construction, and condition (iii) follows from Cases 1--12 discussed above. It remains to check condition (i).

Observe that the ring of regular functions on $\RR_n$ is generated by $2(n+1)$ functions $p_0,\dots, p_{n-1}$, $q_0,\dots, q_{n-1}$, $p_0^{-1},
 (\DDelta_n^{(n-1)})^{-1}$. Clearly, the last two generators belong to $\A^*$. Recall that coefficients $p_i$ satisfy
  relations~\eqref{p-h}. These relations for $k\in [-2,n-3]$ provide a system of linear equations, and by Corollary~\ref{corH}, 
the coefficients of this system belong to $\A^*$. The determinant of the system is $\DDelta_n^{(n-3)}=\DDelta_n^{(n-1)}/p_0^2$, so it does not vanish on $\RR_n$. Therefore, coefficients $p_0,\dots, p_{n-1}$ belong to $\A^*$. Finally, coefficients $q_0,\dots, q_{n-1}$
belong to $\A^*$ due to relations~\eqref{q-h} that involve $p_0,\dots, p_{n-1}$ and $H_k$ for $k<n$.
\end{proof}

\subsection{}\label{difcon}
The cluster algebra $\A$ built above is tightly connected to the cluster algebra studied in \cite{kedem, dk}. We denote the latter
$\A_2$, since to get it from $\A$ one has to fix the values of both stable variables $x_{2n-1}$ and $x_{2n}$ at $1$. Another cluster
algebra, an intermediate between $\A$ and $\A_2$, is obtained by fixing the value of $x_{2n}$ at $1$; it is denoted $\A_1$.

The exchange matrix of $\A_2$ is  obtained
from $\tilde B(\varepsilon)$ by deleting the last two rows. If we take $\varepsilon=(2,1,\dots,1,0)$ and rearrange the cluster
variables as $x_{01},\dots,x_{0 n-1},x_{11},\dots, x_{1 n-1}$, the exchange matrix will be given by
$$
\left( \begin{array}{cc}
0 & -C\\
C & 0
\end{array}
\right),
$$
where $C$ is the Cartan matrix for $A_{n-1}$. This gives precisely the initial cluster considered in \cite{kedem,dk}. Other clusters
related to $Q$-systems (in what follows we call them {\it $Q$-clusters\/}) are obtained from the initial one by using the exchange relation, which is identical to~\eqref{HankJac}. It is shown in Lemma 1.3 in \cite{dk} that
$Q$-clusters correspond bijectively to {\it Motzkin paths}, that is, integer sequences
$\{m_1,\dots,m_{n-1}\}$ such that $|m_j-m_{j+1}|\le 1$. It follows immediately from~\eqref{kappa} that $\{\varkappa_1,\dots,\varkappa_{n-1}\}$ is a Motzkin path starting at $0$. It is easy to check that this gives a bijection between 
$Q$-clusters corresponding to Motzkin paths starting at $0$ and clusters $\x(\varepsilon)$ studied above: the former are
truncations $\x_2(\varepsilon)$ obtained from
$\x(\varepsilon)$ by deleting the stable coordinates. Any other Motzkin path is a translate
of a Motzkin path starting at $0$. The corresponding $Q$-clusters are described by the following statement.

\bl \label{truncat}
Let $\x_1(\varepsilon)=\left((\DDelta_i^{(\varkappa_i)}, \DDelta_i^{(\varkappa_i+1)})_{i=1}^{n-1},\DDelta_n^{(n-1)}\right)$ 
be a cluster in $\A_1$ obtained by the truncation of $\x(\varepsilon)$, and $r\in\Z$.

{\rm (i)} The $r$-shift  
$\x_1^r(\varepsilon)=\left((\DDelta_i^{(\varkappa_i+r)}, \DDelta_i^{(\varkappa_i+r+1)})_{i=1}^{n-1},\DDelta_n^{(n-1+r)}\right)$ 
is a cluster in $\A_1$, and its exchange matrix coincides with
that of $\x_1(\varepsilon)$.

{\rm (ii)}
Let $\x_2(\varepsilon)=(\DDelta_i^{(\varkappa_i)}, \DDelta_i^{(\varkappa_i+1)})_{i=1}^{n-1}$ be the further 
truncation of $\x_1(\varepsilon)$,
then its $r$-shift  $\x_2^r(\varepsilon)=(\DDelta_i^{(\varkappa_i+r)}, \DDelta_i^{(\varkappa_i+r+1)})_{i=1}^{n-1}$ is a $Q$-cluster corresponding to the Motzkin path $\{\varkappa_1+r,\dots,\varkappa_{n-1}+r\}$ and its exchange matrix coincides with
that of $\x_2(\varepsilon)$.
\el

\begin{proof} (i) For $r=1$ and $\varepsilon=(2,0,\dots,0)$ the proof consists in an application of $T_{2n-3} \circ \cdots \circ T_1$
to the graph $\Gamma$ shown on~\eqref{ratgraph} with the white vertex deleted (see the proof of Lemma~\ref{specvar}, and  take 
into an account that $x_{2n}=1$ implies via~\eqref{shift} that $\DDelta_n^{(n-1+r)}= \DDelta_n^{(n-1)}$ for any $r\in \Z$).
To extend this results to any other value  $\varepsilon'$ it suffices to use the cluster transformation taking $\x(\varepsilon)$
to $\x(\varepsilon')$. The case $r>1$ follows by induction. The case $r=-1$ is treated similarly to the case $r=1$ with
$T_{2n-3} \circ \cdots \circ T_1$ replaced by $T_2 \circ \cdots \circ  T_{2n-2}$, and the case $r<-1$ follows by backward induction.

(ii) Follows immediately from (i).
\end{proof}


Consequently, all cluster variables in all $Q$-clusters ($R_{\alpha,m_\alpha}$ in the notation of \cite{dk}) 
form a subset of $\{ x_t(j,k),\ j\in [1,n-1],\ k\in \Z\}$, where $x_2(j,k)$
are obtained from $x(j,k)$ defined in~(\ref{specclust}) by setting both stable variables to~$1$. The correspondence is given by
$R_{\alpha,m_\alpha}\leftrightarrow x_2(\alpha,m_\alpha)$.

We conclude this section with a proposition that, in light of the above fact, 
implies the central positivity result (Theorem 9.15) in \cite{dk}.

\bp
\label{difra}
{\rm (i)} 
For any $\varepsilon$ and any $j\in [1,n-1]$, $k\in \Z$, $x(j,k)$ is a Laurent polynomial in $\x(\varepsilon)$ with non-negative
integer coefficients.

 {\rm (ii)} 
For any $\varepsilon$ and any $j\in [1,n-1]$, $k, r\in \Z$, $x_1(j,k)=x(j,k)|_{x_{2n}=1}$ is a Laurent polynomial in $\x^r_1(\varepsilon)$ with non-negative integer coefficients.
\ep

\begin{proof} (i)  Define parameters $c_i, d_i$ by (\ref{cdshort}). Pick a pair $(u,v)$ of Coxeter elements that correspond to $\varepsilon$ and consider the element $X\in G^{u,v}$ defined by (\ref{factorI}) with factorization parameters $d_i$ and  $c_i^-=c_i$,
 $c_i^+=1$. Then $H_i = H_0 h_i(X)$. This means that  for any $j\in [1,n-1]$, $k\in \Z$, 
$x(j,k)= H_0^j x_{2n}^{\max(0, k+ j +1 - 2n)}  \Delta_j^{(k)} (X)$, where by
$\Delta_j^{(k)} (X)$ we mean the determinant defined in (\ref{det}) built from $h_i(X)$.
By Corollary~\ref{corshift}, $\Delta_j^{(k)} $ is the 
product of a Laurent monomial in variables from $\x(\varepsilon)$ 9with coefficient~$1$) and the minor $\left ( X^{k-\varkappa_j} \right )_{[1,j]}^{[1,j]}$. If $k-\varkappa_j \ge 0$, then, by Lindstr\"om's lemma, this minor is equal to 
the sum of products of  path weights over all collections of non-intersecting paths leading from the $j$ lowest sources to 
the $j$ lowest sinks in the network obtained by concatenating $k-\varkappa_j$ copies of the network $N_{u,v}$. 
Thus $\left ( X^{k-\varkappa_j} \right )_{[1,j]}^{[1,j]}$ is a polynomial in factorization parameters $c_i, d_i$ with non-negative integer coefficients, 
and the claim follows, since cluster variables and factorization parameters are connected by a monomial transformation
with no coefficients. 
On the other hand, if $k-\varkappa_j < 0$, then, by a well-known determinantal identity, 
$\left ( X^{k-\varkappa_j} \right )_{[1,j]}^{[1,j]}= (\det X)^{k-\varkappa_j} \left ( X^{\varkappa_j-k} \right )_{[j+1,n]}^{[j+1,n]}= (d_1\cdots d_n)^{k-\varkappa_j} \left ( X^{\varkappa_j-k} \right )_{[j+1,n]}^{[j+1,n]}$, and 
the previous argument applies.

(ii) By Lemma \ref{truncat}(i), $\x_1^r(\varepsilon)$ is indeed a cluster in $\A_1$, and its exchange matrix corresponds to $\tilde 
B(\varepsilon)$. Define parameters $c^{(r)}_i, d^{(r)}_i$ by (\ref{cdshort}) with every Hankel determinant $\DDelta_i^{(l)}$  replaced by $\DDelta_i^{(l+r)}$. 
Pick a pair $(u,v)$ of Coxeter elements that correspond to $\varepsilon$ and consider the element $X\in G^{u,v}$ defined by (\ref{factorI}) with factorization parameters $d_i= d_i^{(r)}$ and  $c_i^-=c^{(r)}_i$, $c_i^+=1$. Then $H_{i +r}= H_r h_i(X)$ for $i\in [0,\ldots, 2n-1]$. Recursion (\ref{p-h}) together with Remark \ref{handH} imply that, in fact, $H_{i +r}= H_r h_i(X)$ for all $i\in \Z$. Therefore,
$x(j,k)= H_r^j x_{2n}^{\max(0, k+ j +1 - 2n)}  \Delta_j^{(k-r)} (X)$, and the rest of the proof is identical to (i).
\end{proof}

\br \label{tribute}
{\rm (i) In fact, we can refine Proposition~\ref{difra}(i) and prove Laurent positivity of $x(j,k)$ with respect to shifted clusters
as well. However, this proof needs additional tools in cluster algebra theory, and will be published elsewhere.

(ii) If factorization parameters in (\ref{factorI}) are positive, then the
matrix $X$ is totally nonnegative, and so are matrices $X^k$  for $k=1, 2,\ldots$ and $J X^k J$ with $J=\diag((-1)^i)_{i=1}^n$ for 
$k =-1, -2, \ldots$. This indicates, in particular, a connection between $Q$-systems and
totally nonnegative matrices and their network interpretation. This connection is explored  in \cite{DK2}, Section 7.

(iii) The quantization of the cluster algebra considered in this subsection is the subject of the forthcoming paper \cite{BK}.}
\er

\section{Coxeter--Toda lattices}

\subsection{}
The goal of this section is to establish a connection between the cluster
algebra $\A$ defined above and transformations of Coxeter--Toda flows.
First, consider the  Toda hierarchy defined by (\ref{Lax_intro}). Equations
on $X$ induce an evolution
of the corresponding Weyl function $m(\lambda; X)$, which can be most
conveniently described in terms of its Laurent coefficients $h_i$. The
following proposition is well known in the case of the usual (tridiagonal)
Toda flows.

\bp
\label{mom_evol}
If $X=X(t)$ satisfies the Lax equation {\rm (\ref{Lax_intro})}, then coefficients
$h_i(X)= (X^i e_1, e_1)$
of the Laurent expansion of the Weyl function $m(\lambda; X)$ evolve
according to equations
$$
\ddt h_i(X) = h_{i+k}(X) - h_k(X) h_i(X).
$$
\ep

\begin{proof} If $X$ satisfies (\ref{Lax_intro}), then so does $X^i$. By rewriting $X^k$ as $\pi_+(X^k)+\pi_-(X^k)+\pi_0(X^k)$, we get
\begin{equation*}
\begin{split}
\ddt h_i(X) &=\left ( \left [X^i, -\frac{1}{2} \left ( \pi_+(X^k) -
\pi_-(X^k)\right )\right ] e_1, e_1 \right )\\
&= \frac{1}{2} \left (
X^i\left ( X^k - 2 \pi_+(X^k) - \pi_0(X^k)\right ) e_1, e_1 \right ) -
\frac{1}{2} \left ( \left ( - X^k +2 \pi_-(X^k)  + \pi_0(X^k)\right ) X^i
e_1, e_1 \right )\\
& = \left ( X^{i+k} e_1, e_1 \right ) -\left ( \pi_0(X^k)
e_1, e_1 \right ) \left ( X^i e_1, e_1 \right ) = h_{i+k}(X) - h_k(X)
h_i(X).
\end{split}
\end{equation*}
\end{proof}

Now, let $(u,v)$ be a pair of Coxeter elements. Coxeter--Toda flows on
$G^{u,v} /\HH$ are induced by the restriction of the Toda hierarchy to
$G^{u,v}$. To get a more detailed description of Coxeter--Toda flows, we
choose parameters $c_i=c_i^+ c_i^-$, $d_i$ that correspond to the
factorization (\ref{factorI}) of a generic element in $G^{u,v}$ as
coordinates on the open dense set in  $G^{u,v} /\HH$. Indeed, $c_i, d_i$ are
invariant under conjugation by diagonal matrices (cf. Remark~\ref{invrmrk}) and are
clearly independent as functions on $G^{u,v} /\HH$.

\bl The standard Poisson--Lie structure on $GL_n$
induces the following Poisson brackets
for variables $c_i, d_i$:
\be
\{ c_i, c_{i+1}\} = (\varepsilon_{i+1} -1)  c_i c_{i+1}, \quad \{ d_i, d_j\} = 0,
\quad \{ c_i, d_i\} = -  c_i d_{i}, \quad \{ c_i, d_{i+1}\} =   c_i d_{i+1},
\label{brack_cd}
\ee
and the rest of the brackets are zero.
\el

\begin{proof} In view of Theorem \ref{PSGL}, it is sufficient to compute Poisson
brackets for $c_i, d_i$
induced by Poisson brackets (\ref{taubracks}), \eqref{taubracks0} for face weights of the
network $N_{u,v}^\circ$.
The first equation is an easy consequence of the equality $y_{0i}=c^{-1}_i$,
$i\in [1,  n-1]$,
(cf. (\ref{faceweights})) and Poisson relations for $y_{0i}$ described in
(\ref{taubracks}), \eqref{taubracks0}.

By  (\ref{faceweights}),~\eqref{faceweights0}, $y_{0i} y_{1i} = d_i/d_{i+1}$ for $i\in [0, n-1]$
(here $d_0=1$). Therefore, 
$$
\{\log d_i/d_{i+1},\log d_j/d_{j+1}\} = \{\log y_{0i} y_{1i},\log y_{0j} y_{1j}\},\quad i,j\in [0,n-1], 
$$
which equals the sum of the entries of the
$2\times 2$ block of $\Omega$ in rows $2i+1, 2i+2$ and columns $2j+1,2j+2$. By~\eqref{Omega},
each such block is proportional either to $U$, or to $V_k$, or to $V_k^T$, $k\in [1,n-1]$, given by~\eqref{Xtofaceblocks}.
It is easy to see that the sum of the entries for each of these matrices equals zero, and 
hence $\{d_i/d_{i+1},d_j/d_{j+1}\}=0$ for all $i,j\in[0,n-1]$. In particular, this holds for $i=0$, which can be 
re-written as $\{d_1, d_j/d_{j+1}\}=0$ for all $j\in[0,n-1]$. Taking into account that 
$d_j = d_1 (d_2/d_1) \cdots (d_i/d_{i-1})$, we get the second formula in~\eqref{brack_cd}.

Similarly,
$$
\{ \log c_{i},  \log d_{j+1}/d_{j}\} =\{ \log 1/c_{i},  \log d_j/d_{j+1}\}
 = \{ \log y_{0i} ,  \log y_{0j} y_{1j}\}, 
$$ 
for $i\in [1,n-1]$, $j\in [0,n-1]$,
which equals the sum of the two upper entries of the
$2\times 2$ block of $\Omega$ in rows $2i+1, 2i+2$ and columns $2j+1,2j+2$.
By~\eqref{Omega}, if such a block is nontrivial, it is equal
either to $U$, or to $V_k$, or to $-V_k^T$, $k\in [1,n-1]$, given by~\eqref{Xtofaceblocks}.
Since the sum of the two upper entries equals $2$ for $U$ and $-1$ in the other two cases,   
we conclude that
$\{ \log c_{i},  \log d_{j+1}/d_{j}\} = 2\delta_{i,j}-\delta_{i,j+1} -\delta_{i, j-1}$ for $i\in [1,n-1]$, 
$j\in [0,n-1]$. In particular, for $j=0$ one gets $\{\log c_i, \log d_1\}=-\delta_{i1}$ for $i\in [1,n-1]$.
Re-writing $d_j$ via $d_1$ and $d_{i+1}/d_i$ as before, one gets 
$\{ \log c_{i},  \log d_{j}\} =  - \delta_{i,j} + \delta_{i, j-1}$, $i\in [1,n-1]$, $j\in [1,n]$, which
is equivalent to the last two equations in (\ref{brack_cd}).
\end{proof}

\br{\rm We could have also computed brackets (\ref{brack_cd}) by
specializing general formulas
obtained in \cite{KZ} for Poisson brackets for factorization parameters of
an arbitrary double Bruhat cell
in a standard semisimple Poisson--Lie group.
}
\er

Due to their invariance under conjugation by elements of $\HH$, Hamiltonians
$F_k(X)= \frac{1}{k} \Trace X^k$ of the Toda flows, when restricted
to a Coxeter double Bruhat cell $G^{u,v}$, can be expressed as functions
of $c_i, d_i$, which, in turn, serve as Hamiltonians for Coxeter--Toda flows on
$G^{u,v} /\HH$. The easiest way to write down $F_k$ as a function  of $c_i,
d_i$ explicitly is to observe that $\Trace X^k$ is equal to the sum
of weights of all paths
that start and end at the same level in the planar network obtained by
concatenation of $k$ copies
of $N_{u,v}$. In the case $k=1$, we only need to use $N_{u,v}$ itself, which
leads to the following formula for $F_1$: define $I^-$ and $I^+$ by~\eqref{IL} and denote $I^- \cup I^+ = \{
1=i_1, \ldots, i_m=n\}$, then
\be
\label{F1}
F_1=F_1(c, d) = d_1 + \sum_{l=1}^{k-1} \sum_{j=i_l+1}^{i_{l+1}} \left (  d_j
+ c_{j-1} d_{j-1} + \ldots
c_{j-1} \cdots c_{i_l} d_{i_l} \right ).
\ee
One can use (\ref{F1}), (\ref{brack_cd}) to write equations of the first
Coxeter--Toda flow generated by
$F_1$ on $G^{u,v}/\HH$ as a system of evolution equations for $c_i, d_i$.

\begin{examp}\label{runexCTL}
{\rm
(i) For our running Example \ref{runex}, $I^-\cup I^+=\{1,3,4,5\}$, so~(\ref{F1}) becomes
$$
F_1 = d_1 + d_2 + c_1 d_1 + d_3 + c_2 d_2 + c_2 c_1 d_1 + d_4 + c_3 d_3 +
d_5 + c_4 d_4.
$$

(ii) Let $v=s_{n-1} \cdots s_1$, then $I^-\cup I^+=[1,n]$ and formula
(\ref{F1}) reads
$F_1 (c, d) = d_1 + d_2 + c_1d_1+\ldots + d_n + c_{n-1} d_{n-1}$. If, in
addition, $u=v^{-1}$, then $\varepsilon= ( 2, 0,\ldots, 0)$ and
$F_1$ and (\ref{brack_cd}) generate Hamiltonian equations
$$
\ddt d_i = \{ d_i, F_1\} = \{d_i, c_i d_i + c_{i-1}d_{i-1}\} = d_i ( c_i d_i
- c_{i-1}d_{i-1} ),
$$
$$
\ddt c_i = \{ c_i, F_1\} = \{c_i, d_i + d_{i+1} + c_{i-1}d_{i-1} + c_i d_i +
c_{i+1}d_{i+1}\} =
c_i ( d_{i+1} - d_i  +  c_{i-1}d_{i-1} - c_i d_i  ).
$$
Then a change of variables $r_{2i-1} = d_i$, $i\in [1, n]$, and  $r_{2i} =
c_i d_i$, $i\in [1, n-1]$, results in
the equations of the {\em open Volterra lattice\/}:
$$
\ddt r_i = r_i (r_{i+1} - r_{i-1}),\quad i\in [1, 2n-1]; \ r_{0} = r_{2n}= 0.
$$
Another change of variables, $a_i = c_i d_i^2,\ b_i = d_i + c_{i-1} d_{i-1}$, 
leads to equations of motion
of the Toda lattice that were presented in the introduction. Note that
$a_i, b_i$ are, resp., subdiagonal and diagonal matrix entries in a lower
Hessenberg representative of an element in $G^{u,v}/\HH$ defined by
parameters $c_i, d_i$.

(iii) If
$u=v=s_{n-1} \cdots s_1$, then $\varepsilon=\{2, 1, \ldots, 1,0\}$, and
Hamiltonian equations generated by
$F_1$ and (\ref{brack_cd}) produce the system
$$
\ddt d_i = d_i ( c_i d_i - c_{i-1}d_{i-1} ),\quad 
\ddt c_i =
c_i ( d_{i+1} - d_i  +  c_{i+1}d_{i+1} - c_i d_i  ).
$$
After the change of variables $\tilde c_i = c_i d_i $ this system turns into the
{\em relativistic Toda lattice}
$$
\ddt d_i = d_i ( \tilde c_i  - \tilde c_{i-1} ), \quad
\ddt {\tilde c}_i =
\tilde c_i ( d_{i+1} - d_i  +  \tilde c_{i+1} - \tilde c_{i-1} ).
$$
}
\end{examp}

Proposition \ref{mom_evol} combined with Theorem \ref{invthm} suggests a
method to solve Coxeter--Toda lattices explicitly, following the strategy that
was originally applied in \cite{moser} to the usual Toda lattice. In order
to find a solution with initial conditions $c_i(0), d_i(0)$ to the
Coxeter--Toda equation on $G^{u,v}/\HH$ generated by the Hamiltonian $F_k$,
we first define 
$$
m^0(\lambda)=m(\lambda;X(0))=\sum_{i=0}^\infty \frac{h_i^0}{\lambda^{i+1}}
$$
to be the Weyl function of any representative $X(0)\in
G^{u,v}$  of the element in $G^{u,v}/\HH$  with coordinates $c_i(0), d_i(0)$.
Let $M(\lambda;t) = \sum_{i=0}^\infty {H_i(t)}{\lambda^{-i-1}}$ be the
solution to
a linear system on $\mathcal{R}_n$  described in terms of Laurent
coefficients $H_i(t)$ by
$$
\ddt H_i(t) = H_{i+k} (t), \quad i=0,1,\ldots, 
$$
with initial conditions $H_i(0) = h_i^0$. For $i < 0$, define $H_i(t)$ via~\eqref{p-h}, 
where $(-1)^{n-i} p_i$ are
coefficients of the characteristic polynomial of $X(0)$.

\bp
\label{solveToda} The solution with initial conditions $c_i(0), d_i(0)$ to
the $k$th Coxeter--Toda equation on $G^{u,v}/\HH$ is given by formulas
{\rm (\ref{cd})} with $h_i=h_i(t) = {H_i(t)}/{H_0(t)}$, $i\in \mathbb{Z}$.
\ep

\begin{proof} An easy calculation shows that $h_i=h_i(t) =
{H_i(t)}/{H_0(t)}$, $i\ge 0$, give the solution to the system presented in
Proposition \ref{mom_evol} with initial conditions
$h_i(0)=h_i^0$. Thus the function $m(\lambda,t)=\sum_{i=0}^\infty
{h_i(t)}{\lambda^{-i-1}}$ evolves in the way prescribed by the $k$th Toda
flow and therefore coincides with $m(\lambda;X(t))$, where $X(t)$ is the
solution of (\ref{Lax_intro}) with the initial condition $X(0)$. Since
coefficients of the characteristic polynomial are preserved by Toda flows,
Remark~\ref{handH}(ii) implies that for $i <0$ we also
have $h_i(t) = h_i(X(t))$. Finally, since  the Moser map is
invertible on  $G^{u,v}/\HH$, we see that the system in Proposition \ref{mom_evol} is, in
fact, equivalent to the $k$th Toda flow on $G^{u,v}/\HH$ which completes the
proof.
\end{proof}

We see that for any pair of Coxeter elements $(u,v)$, the Coxeter--Toda flows are
equivalent to the same evolution of Weyl functions. We want to exploit this
fact to construct, for any two pairs $(u, v)$ and $(u', v')$ of Coxeter elements,
a transformation between
$G^{u,v}/\HH$ and $G^{u',v'}/\HH$ that is Poisson and maps the $k$th
Coxeter--Toda flow into the $k$th Coxeter--Toda flow. We call such a
transformation a {\em generalized B\"acklund--Darboux transformation}. The
term ``B\"acklund transformation'' has been used broadly over the years for
any transformation that maps solutions of one nonlinear equation into
solutions of another. To justify
the use of Darboux's name, we recall that traditionally a
B\"acklund--Darboux transformation consists
in interchanging factors in some natural factorization of the Lax operator
associated with a given integrable system. In the case of Coxeter--Toda
flows, the same number and type of elementary factors
appears in the Lax matrix associated with any Coxeter double Bruhat cell.
Hence we  use the term
``generalized B\"acklund--Darboux transformation'' even though in our case,
re-arrangement of factors is accompanied by a transformation of
factorization parameters.

Let us fix two pairs,  $(u, v)$ and $(u', v')$, of Coxeter elements and let
$\varepsilon=(\varepsilon_i)_{i=1}^n$,
$\varepsilon'=(\varepsilon'_i)_{i=1}^n$ be the corresponding $n$-tuples defined 
by~\eqref{eps},~\eqref{epsum}. We construct a map $\sigma_{u,v}^{u',v'}: G^{u,v}/\HH
\to G^{u',v'}/\HH$ using the following procedure.
Consider the cluster algebra $\A$ defined in Section 5. Fix a seed
$\Sigma(\varepsilon)=(\x(\varepsilon), \tilde B(\varepsilon))$ in $\A$, where $\x(\varepsilon)$ is given
by~\eqref{initclust} and $B(\varepsilon)$ by~\eqref{bepsilon}.
Let $T_\varepsilon^{\varepsilon'}$ be the sequence of cluster transformations
defined in the proof of Lemma~\ref{Lem4.2} that transforms
$\Sigma(\varepsilon)$  into the seed $\Sigma(\varepsilon')$. Next, for an element in $G^{u,v}/\HH$ with coordinates $c_i,
d_i$, consider its representative $X\in G^{u,v}$, the corresponding Weyl
function $m(\lambda; X)$ and the sequence of moments $h_i(X)$, $i \in
\mathbb{Z}$. Apply transformation $\tau_{u,v}$ by
assigning values to cluster variables in the cluster
$\x(\varepsilon)$ according to formulas~\eqref{ddet},~\eqref{shorthand},~\eqref{initclust} with $H_i$
replaced by $h_i(X)$. Then apply transformation $T_\varepsilon^{\varepsilon'}$
to $\x(\varepsilon)$ to obtain the cluster $\x(\varepsilon')$. Finally, apply transformation $\rho_{u',v'}$ by using equations~\eqref{cdshort}
with $\varepsilon$ replaced by $\varepsilon'$
and components of $\x(\varepsilon)$ replaced by those of $\x(\varepsilon')$ to
compute parameters $c'_i, d'_i$ that serve
as coordinates of an element in $G^{u',v'}/\HH$. This concludes the
construction of $\sigma_{u,v}^{u',v'}$.

\bt \label{GBD}
The map $\sigma_{u,v}^{u',v'}: G^{u,v}/\HH \to G^{u',v'}/\HH$ is a
birational transformation that preserves the Weyl function,
maps Coxeter--Toda flows on $G^{u,v}/\HH$ into matching Coxeter--Toda flows on
$G^{u',v'}/\HH$ and is Poisson with respect
to Poisson structures on $G^{u,v}/\HH$ and  $G^{u',v'}/\HH$ induced by the
standard Poisson--Lie bracket on $GL_n$.
\et

\begin{proof} Moments $h_j(X)$ are polynomial functions of $c_i, d_i$ for
$i\ge 0$ and rational functions of  $c_i, d_i$ for $i < 0$. Values we assign
to cluster variables in $\x(\varepsilon)$ are thus rational functions of  $c_i,
d_i$. This, combined with the rationality of $T_\varepsilon^{\varepsilon'}$ and
equations~\eqref{cdshort}, shows that the map $\sigma_{u,v}^{u',v'}$ is rational. It
is easy to see that its inverse is $\sigma^{u,v}_{u',v'}$ which implies
birationality. The claim that $\sigma_{u,v}^{u',v'}$ preserves the Weyl function is simply a
re-statement of Lemma~\ref{Lem4.2}, which implies that if clusters $\x(\varepsilon)$ and
$\x(\varepsilon')$ are obtained from a function $M(\lambda) \in \mathcal R_n$
according to ~\eqref{ddet},~\eqref{shorthand},~\eqref{initclust}, then $T_\varepsilon^{\varepsilon'}$
transforms $\x(\varepsilon)$ into $\x(\varepsilon')$. The rest of the statement of
the theorem is a consequence of the invariance of the Weyl function, since
Poisson
structures on $G^{u,v}/\HH$ and  $G^{u',v'}/\HH$  induce the same Poisson
bracket on $\RR_n$ compatible with $\A$
and, by Proposition \ref{mom_evol},  Coxeter--Toda flows generated by
Hamiltonians $F_k$ on $G^{u,v}/\HH$ and  $G^{u',v'}/\HH$  induce the same
evolution
of the Weyl function.

\end{proof}

To illustrate Theorem~\ref{GBD}, in the table below we list elementary generalized
B\"acklund--Darboux transformations that correspond to cluster
transformations from a fixed cluster $\x(\varepsilon)$ into an adjacent cluster
$\x(\varepsilon')$. The table can be viewed in parallel with the table in the
proof of Lemma~\ref{Lem4.2}. Expressions for transformed variables $c'_j$,
$d'_j$ are obtained by combining formulas for cluster transformations with
equations (\ref{cdshort}). Variables that are not listed are left unchanged.

\medskip

{
\offinterlineskip
\tabskip=0pt
\halign{ 
\vrule height2.75ex depth2.25ex width 0.6pt  #\tabskip=0.3em &  
\hfil $#$\hfil  & 
\strut\vrule # & 
\hfil $#$\hfil &
\vrule # &
\hfil #\hfil &
#\vrule &
\hfil #\hfil &
#\vrule width 0.6pt \tabskip=0pt\cr
\noalign{\hrule height 0.6pt}
& \varepsilon & & \varepsilon' &&  Transformation & & Inverse & \cr
\noalign{\hrule}
& \varepsilon_{i}=0 && \varepsilon'_{i}=1 && $ c'_i = \frac{c_i d_i}{d_{i+1}}$ && $c_i = \frac{c'_i
d'_{i+1}}{d'_{i}(1+c'_i)^2}$ &\cr
& \varepsilon_{i+1}= 2 && \varepsilon'_{i+1}=1 &&  $d'_i= \frac{ d_i d_{i+1}}{d_{i+1} + c_i d_i},\  d'_{i+1} = d_{i+1} + c_i d_i$
&& $d_{i} = d'_{i}(1 + c_i'),\ d_{i+1}= \frac{d'_{i+1}}{1 + c'_i}$ &\cr
\noalign{\hrule height 0.6pt}
& \varepsilon_{i}=2 && \varepsilon'_{i}=1 && $c'_i = \frac{c_i d_{i+1}}{d_{i}(1+c_i)^2}, \ c'_{i+1} = c_{i+1} (1+c_i)$ &&
$c_i = \frac{c'_i d'_{i}}{d'_{i+1}}, \ c_{i+1}= \frac{ c'_{i+1} d'_{i+1}}{d'_{i+1} + c'_i d'_i}$&\cr
& \varepsilon_{i+1}= 0 && \varepsilon'_{i+1}=1 && $d'_{i+1} = \frac{d_{i+1}}{1 + c_i }, \ d'_i= d_{i} (1+c_i)$ &&
 $d_{i+1} = d'_{i+1}  + d'_{i} c_i', \ d_{i}= \frac{  d'_i d'_{i+1}}{d'_{i+1}  + d'_{i} c_i' }$ &\cr
\noalign{\hrule height 0.6pt}
& \varepsilon_{i}=1 && \varepsilon'_{i}=0 && $c'_i = \frac{c_i d_{i+1}}{d_{i}(1+c_i)^2}, \quad c'_{i+1}=c_{i+1} ( 1 + c_i)$ &&
$c_i = \frac{c'_i d'_i}{d'_{i+1}},\quad  c_{i+1} = \frac{ c'_{i+1}d'_{i+1}}{d'_{i+1} + c'_i d'_i}$ &\cr
& \varepsilon_{i+1}= 0 && \varepsilon'_{i+1}=1 && $d'_{i} = d_{i}(1 + c_i),\quad d'_{i+1}= \frac{  d_{i+1}}{1 + c_i}$ &&
$d_{i+1} = d'_{i+1} + c'_i d'_i,\quad d_i= \frac{ d'_i d'_{i+1}}{d'_{i+1} + c'_i d'_i}$&\cr
\noalign{\hrule height 0.6pt}
& \varepsilon_{i}=2 && \varepsilon'_{i}=1 && $c'_i = \frac{c_i d_{i+1}}{d_{i}(1+c_i)^2},\quad  c'_{i-1}=c_{i-1} ( 1 + c_i)$ &&
$c_i = \frac{c'_i d'_i}{d'_{i+1}},\quad  c_{i-1} = \frac{ c'_{i-1}  d'_{i+1}}{d'_{i+1} + c'_i d'_i}$&\cr
& \varepsilon_{i+1}= 1 && \varepsilon'_{i+1}=2 && $d'_{i} = d_{i}(1 + c_i),\quad d'_{i+1}= \frac{  d_{i+1}}{1 + c_i}$ &&
$ d_{i+1} = d'_{i+1} + c'_i d'_i,\quad  d_i= \frac{ d'_i d'_{i+1}}{d'_{i+1} + c'_i d'_i}$&\cr
\noalign{\hrule height 0.6pt}
&&&&& $c'_{n-1} = \frac{c_{n-1} d_{n-1}}{d_{n}}$ && $c_{n-1} = \frac{c'_{n-1}
 d'_{n}}{d'_{i}(1+c'_{n-1})^2}$&\cr
& \varepsilon_{n-1}= 0 && \varepsilon'_{n-1}=1 && $d'_{n} =d_{n} + c_{n-1} d_{n-1}$ &&
 $d_{n-1} = d'_{n-1}(1 + c_{n-1}')$&\cr
 &&&&& $d'_{n-1}= \frac{ d_n d_{n-1}}{d_{n} + c_{n-1}d_{n-1}}$ &&$d_{n}= \frac{  d'_{n}}{1 + c'_{n-1}}$&\cr
 \noalign{\hrule height 0.6pt}
  &&&&& $c'_{n-1}= \frac{c_{n-1} d_{n-1}}{d_{n}}$ && $c_{n-1} = \frac{c'_{n-1} d'_{n}}{d'_{n-1}(1 +c'_{n-1})^2}$&\cr
 & \varepsilon_{n-1}= 1 && \varepsilon'_{n-1}=2 && $c'_{n-2} = \frac{ c_{n-2}
  d_{n}}{d_{n} + c_{n-1} d_{n-1}}$ &&  $c_{n-2}=c'_{n-2} ( 1 + c'_{n-1})$&\cr
  &&&&& $d'_{n} = d_{n} + c_{n-1} d_{n-1}$ && $d_{n-1} = d'_{n-1}(1 + c'_{n-1})$&\cr
  &&&&& $d'_{n-1}= \frac{ d_n d_{n-1}}{d_{n} + c_{n-1} d_{n-1}}$ && $d_{n}= \frac{  d'_{n}}{1 +c'_{n-1}}$&\cr
 \noalign{\hrule height 0.6pt} 
}}

\medskip

Elementary generalized B\"acklund--Darboux transformations can be conveniently interpreted in
terms of equivalent transformations of perfect networks introduced in \cite{Postnikov}. The three types of equivalent transformations are shown in Figure~\ref{equitrans}. Instead of trying to describe the
general case, we will provide an example.

\begin{figure}[ht]
\begin{center}
\includegraphics[width=9.0cm]{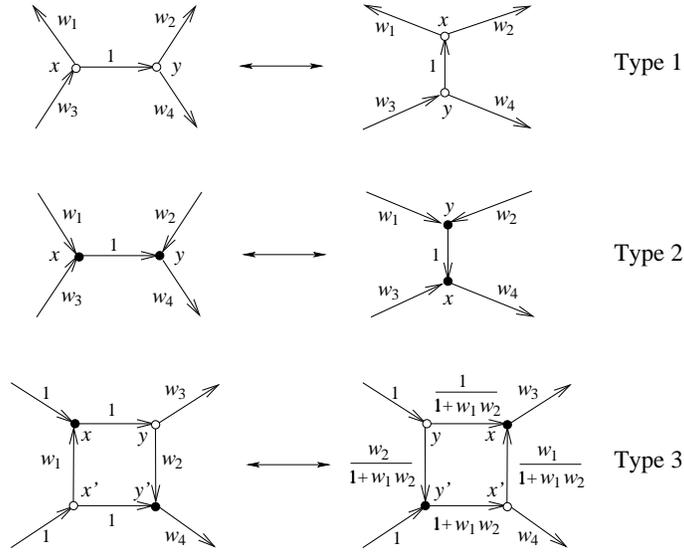}
\caption{Equivalent transformations of perfect networks}
\label{equitrans}
\end{center}
\end{figure}

\begin{examp}\label{runexGBD}
{\rm Consider the network from Example~\ref{runexan}. Recall that $\varepsilon=(2,2,1,0,0)$ 
and set $i=2$. So, $\varepsilon_2=2$ and $\varepsilon_3=1$, which corresponds to the fourth row of the 
above table. The corresponding transformation consists of the following steps:

(i) Type 2 transformation with $x=v_b^+(3)$, $y=v_b^-(3)$ and $w_1=w_4=1$, $w_2=c_2^+$, $w_3=c_3^-$.

(ii) Type 3 transformation with $x=v_b^+(3)$, $y=v_w^-(3)$, $x'=v_w^+(2)$, $y'=v_w^-(2)$ and 
$w_1=c_2^+$, $w_2=c_2^-$, $w_3=d_3$, $w_4=1$.

(iii) Type 1 transformation with $x=v_w^+(2)$, $y=v_w^-(2)$ and $w_1=c_2^+/(1+c_2)$, $w_2=1$,
$w_3=d_2$, $w_4=c_1^-$.

(iv) The gauge group action at $v_b^+(3)$ that takes the triple of weights $(d_3, c_2^+/(1+c_2), 
1/(1+c_2))$ to $(1, d_3c_2^+/(1+c_2), d_3/(1+c_2))$.

\begin{figure}[ht]
\begin{center}
\includegraphics[width=10.0cm]{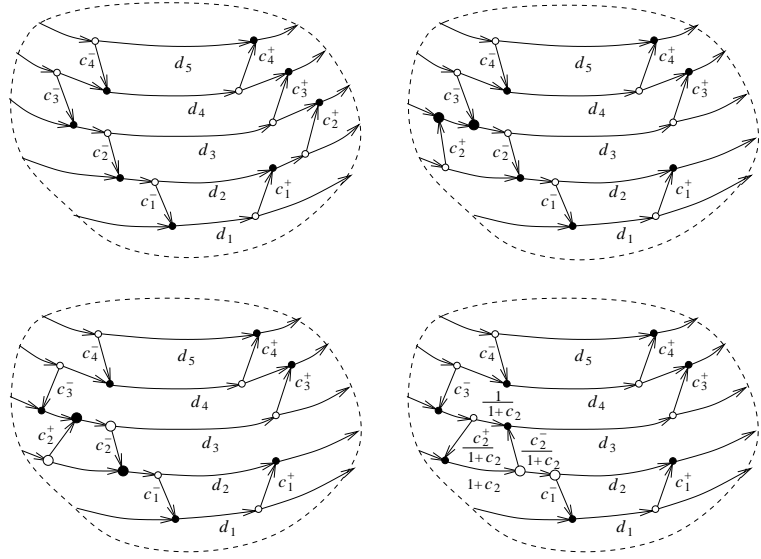}
\caption{Elementary generalized B\"acklund--Darboux transformation: steps (i) and (ii)}
\label{gbd1}
\end{center}
\end{figure}

(v) The gauge group action at $v_w^-(2)$ that takes the triple of weights $(1+c_2, 1, c_1^-)$  to $(1,1+c_2, c_1^-(1+c_2))$.

(vi) The gauge group action at $v_w^+(2)$ that takes the triple of weights $(1+c_2,
d_3c_2^+/(1+c_2), d_2)$ to $(d_2(1+c_2), d_3c_2^+/[d_2(1+c_2)], 1)$.

Thus, at the end we have $(c_2^-)'=c_2^-/(1+c_2)$, $(c_2^+)'=d_3c_2^+/[d_2(1+c_2)]$, and hence
$c_2'=d_3c_2/[d_2(1+c_2)]$. Besides, $(c_1^-)'=c_1^-(1+c_2)$, $c_1^+)'=c_1^+$, and hence
$c_1'=c_1(1+c_2)$. Finally, $d_2'=d_2(1+c_2)$ and $d_3'=d_3/(1+c_2)$. All these expressions coincide with those given in the fourth row of the table.

Transformations of the relevant part of the network during the first two steps are shown in 
Figure~\ref{gbd1}.

Transformations of the relevant part of the network during the remaining four steps are shown in 
Figure~\ref{gbd2}. 

\begin{figure}[ht]
\begin{center}
\includegraphics[width=10.0cm]{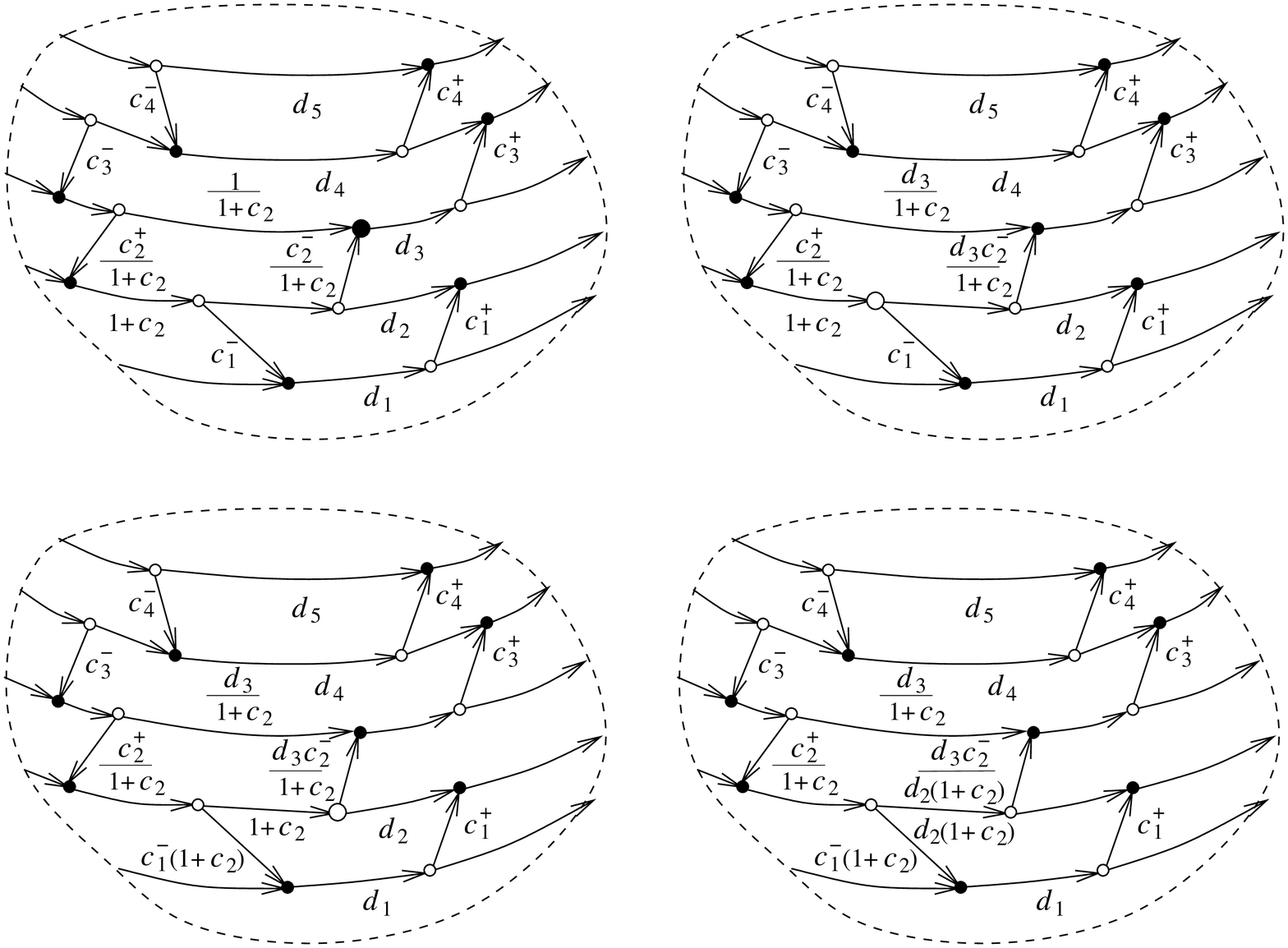}
\caption{Elementary generalized B\"acklund--Darboux transformation: steps (iii)-(vi)}
\label{gbd2}
\end{center}
\end{figure}
}
\end{examp}

We can make transformations $\sigma_{u,v}^{u',v'}$ more explicit by using Corollary~\ref{corshift}.
Below we write $A_{[i]}$ for the determinant of the leading principal $i\times i$ submatrix of a matrix $A$.
Pick an element $X=X(c,d) \in G^{u,v}$ defined by (\ref{factorI}) with factorization parameters
$d_i$ and $c_i^-= c_i$, $c_i^+ = 1$. 
 
\bp The map $\sigma_{u,v}^{u',v'} $ transforms coordinates $c_i, d_i$ on $ G^{u,v}/\HH$ into
coordinates $c'_i, d'_i$ on $ G^{u',v'}/\HH$ given by formulas
$$
d_i' = \frac{\left (X^{\delta\varkappa_i+1} \right )_{[i]} \left (X^{\delta\varkappa_{i-1} }\right )_{[i-1]} } 
 {\left (X^{\delta \varkappa_i} \right )_{[i]} \left (X^{\delta\varkappa_{i-1}+1} \right )_{[i-1]} }, $$
$$ c_i' = c_i d_i^2 \frac{ (X_{[i-1]})^{\varepsilon_i} }{ (X_{[i+1]})^{\varepsilon_{i+1}}}
 \frac{\left (X^{\delta\varkappa_{i-1}} \right )_{[i-1]} \left (X^{\delta \varkappa_{i+1} }\right )_{[i+1]} } 
 {\left (X^{\delta \varkappa_i+1} \right )_{[i]}^2 }
\left (  \frac{ \left (X^{\delta \varkappa_{i+1} +1}\right )_{[i+1]} } { \left (X^{\delta\varkappa_{i+1} }\right )_{[i+1]} } \right )^{\varepsilon_{i+1}}
\left (  \frac{ \left (X^{\delta \varkappa_{i-1} +1}\right )_{[i-1]} } { \left (X^{\delta\varkappa_{i-1} }\right )_{[i-1]} } \right )^{2-\varepsilon_{i}},
$$
where $\varepsilon$ and $\varkappa$ (resp. $\varepsilon'$ and $\varkappa'$) are $n$-tuples (\ref{epsum}), (\ref{kappa}) associated with $(u,v)$ (resp. $(u',v')$) and $\delta\varkappa_j=\varkappa'_j-\varkappa_j$ for $j\in [1,n]$.
\label{BDexact}
\ep

\begin{proof} The claim follows from formulas (\ref{cd}), (\ref{Gamma_l}), (\ref{Deltashift}) and an easy computation that shows that
$$
\frac{\Gamma_{i-1}\Gamma_{i+1}}{\Gamma_{i}^2} =  c_i d_i^2 \frac{ (d_1\cdots d_{i-1})^{\varepsilon_i} }{ (d_1\cdots d_{i+1})^{\varepsilon_{i+1}}}=
 c_i d_i^2 \frac{ (X_{[i-1]})^{\varepsilon_i} }{ (X_{[i+1]})^{\varepsilon_{i+1}}}.
$$
\end{proof}

\begin{examp}
{\rm Let $v=u^{-1}=u'=v'=s_{n-1}\cdots s_1$. Then $G^{u,v}/\HH$ is the set of Jacobi matrices (\ref{Jac}), which serves as the phase space for the finite nonperiodic Toda lattice, and $G^{u',v'}/\HH$ can be viewed as a phase space for the relativistic Toda lattice. Combining Theorem \ref{GBD}
with Example~\ref{runexCTL}, we obtain the following corollary of Proposition \ref{BDexact}.

\begin{coroll} If the entries $a_i, b_i$ of the Jacobi matrix $L$ evolve according to the equations
of the Toda lattice, then functions
$$
d'_i =  \frac{\left (L^{2-i} \right )_{[i]} \left (L^{2-i}\right )_{[i-1]} } 
 {\left (L^{1-i} \right )_{[i]} \left (L^{3-i} \right )_{[i-1]} },\qquad 
\tilde {c}'_i = a_i  \frac{\left (L^{-i} \right )_{[i+1]} \left (L^{3-i}\right )_{[i-1]} } 
 {\left (L^{1-i} \right )_{[i]} \left (L^{2-i} \right )_{[i]} }
$$
solve the relativistic Toda lattice.
\end{coroll}
\begin{proof} First, observe that an element $X$ featured in  Proposition \ref{BDexact} is a tridiagonal matrix whose nonzero off-diagonal entries are $ X_{i i+1} = d_i, X_{i+1 i} = c_i d_i $. The matrix $L$ associated with the same parameters $c_i, d_i$ is related to $X$ via $L = D X D^{-1}$ where
$D= \diag(1, d_1, \ldots, d_1\cdots d_{n-1})$. This means that $ (L^k)_{[i]} =  (X^k)_{[i]} $ for any $i,k$ and 
$a_i = L_{i+1 i} = c_i d^2_i $. Furthermore, $\varepsilon_i=0$ for $i\in [2,\ldots, n]$, $\varkappa_i = i-1$ and $\varkappa'_i = 0$ for $i\in [1,\ldots, n]$. The claim then
follows from  Example~\ref{runexCTL}(iii) and formulas of Proposition \ref{BDexact}.
\end{proof}

}
\end{examp}

\subsection{}
It is natural to ask if the classical Darboux transformation $ X= X_- X_0
X_+ \mapsto D(X) = X_0 X_+ X_- $ can  also be
interpreted in terms of the cluster algebra $\A$.  The transformation
$D$ constitutes a step in the {\em $LU$-algorithm}
for computing eigenvalues of a matrix $X$. A connection of the
$LU$-algorithm (as well as similar numerical algorithms, such as $QR$ and
Cholesky algorithms) to integrable systems of Toda type is well-documented,
see, e.g. \cite{dlt, w}. For an arbitrary semisimple Lie group, a restriction
of such a transformation to a Coxeter double Bruhat cell of type $G^{u,u}$ was
studied, under
the name of {\em factorization dynamics}, in \cite{Reshetikhin&Co}. We
collect some relevant simple facts about
the transformation $\D$ in the proposition below.

\bp
\label{darboux} 
Let $X\in \N_- \B_+$. Then

{\rm (i)} for any $i\in \mathbb{Z}$, $h_i(D(X))=h_{i+1}(X)/h_1(X)$;

{\rm (ii)} for any $u,v\in S_n$, if $X\in G^{u,v}$ then $D(X)\in
G^{u,v}$;

{\rm (iii)} $D$ descends to a rational Poisson map
$\D: G^{u,v}/\HH\to G^{u,v}/\HH$ that coincides with a time-one map of
the Hamiltonian flow
generated by the Hamiltonian $F(X)= \frac{1}{2} \Trace \log^2 X$.
\ep

\begin{proof} (i) For $i\ge 0$, we have
$ h_{i+1} (X)  =\left ( X_- (X_0 X_+ X_-)^{i} X_0 X_+ e_1, e_1 \right ) =
d_1 \left (D(X)^{i} e_1, e_1 \right )= h_1(X) h_{i} (D(X))$. 
The case $i < 0$ can be treated similarly.

(ii) It suffices to observe that if $Y_1 \in \N_-$ and $Y_2 \in \B_+$ than
both statements $Y_1 Y_2 \in G^{u,v}$ and $Y_2 Y_1 \in G^{u,v}$ are
equivalent to $Y_1 \in \B_+ v \B_+, Y_2 \in \B_- u \B_-$.

(iii) Claim (ii) implies that $D$ descends to a rational map from
$G^{u,v}/\HH$ to $G^{u,v}/\HH$. The rest of the claim is an immediate
corollary of general results in Section~7.1 in \cite{Reshetikhin&Co}.
\end{proof}

For a pair of Coxeter elements $(u,v)$, Proposition \ref{darboux}(i) allows us to
completely describe the action of $\D$ on $G^{u,v}/\HH$ in terms of a
simple map on $\RR_n$. Namely, define $\eta : \RR_n \to
\RR_n$ by $\eta(M(\lambda)) = 
\lambda M(\lambda) - H_0$.
Equivalently, $\eta$ can be described by
$\eta\left( \sum_{i=0}^\infty H_i\lambda^{-i-1} \right ) =
\sum_{i=0}^\infty H_{i+1}\lambda^{-i-1}$.
Then Proposition~\ref{darboux}(i) implies that on $G^{u,v}/\HH$ 
$$
\D =  \rho_{u,v} \circ \x_{u,v}\circ \eta \circ m_{u,v},
$$
where maps $\rho_{u,v}$, $\x_{u,v}$, $m_{u,v}$ were defined in the Introduction.

\br 
{\rm As we have seen in Section~\ref{difcon}, the shift $H_i \mapsto H_{i+1}$ plays an important role in the study of
$Q$-systems in \cite{dk}. }
\er

To tie together the cluster algebra $\A$ and the Darboux transformation $\D$, 
we have to descend to the cluster algebra $\A_1$ introduced in Section~\ref{difcon}. 
we will only need to fix the stable variable $x_{2n}$ to be equal to 1.
In view of (\ref{initclust}), this means that we are dealing with double
Bruhat cells in ${SL}_n$ rather than in $GL_n$. 
In order to emphasize a similarity between the classical Darboux transformation $\D$ and the generalized
B\"acklund--Darboux transformation $\sigma^{u',v'}_{u,v}$, we express the former similarly to~\eqref{gbdasclust}.

\bp 
$\D= \rho_{u,v} \circ T_{\D} \circ
\tau_{u,v}$, where $T_{\D} $
is a sequence of cluster transformations in $\A$. 
\ep

\begin{proof}
Note that in the graphical representation of the matrix
$B(\varepsilon)$ that we employed in the proof of Lemma \ref{specvar}, passing to the cluster
algebra $\A'$
amounts to erasing the white vertex and all corresponding edges in the
graph $\Gamma$. Consider the cluster corresponding to $\varepsilon=(2,0,\ldots,
0)$. By Lemma~\ref{truncat}(i),  the shift $H_i\mapsto H_{i+1}$ is achieved by
an application of $T_{2n-3} \circ \cdots \circ T_1$. 
This means that for $v=u^{-1}=s_{n-1}\cdots s_1$ we can choose
$T_{2n-3} \circ \cdots \circ T_1$ for $T_{\D}$. Then, for arbitrary pair of 
Coxeter elements $(u,v)$, $T_{\D} $ can be defined as
$$
T_{\D} = T^{u,v}_{w^{-1},w}\circ \left ( T_{2n-3} \circ
\cdots \circ T_1 \right ) \circ T_{u,v}^{w^{-1},w}
$$
with $w=s_{n-1}\cdots s_1$.
\end{proof}

\section*{Acknowledgments}

We wish to express gratitude to A.~Berenstein, P.~Di Francesco, R.~Kedem and A.~Zele\-vinsky for useful comments. 
A.~V.~would like to thank the University of Michigan, where he spent a sabbatical term in Spring 2009 and where this paper was finished. He is grateful to Sergey Fomin for warm hospitality and stimulating working conditions.
M.~S.~expresses his gratitude to the Stockholm University and the Royal Institute of Technology , 
where he worked on this manuscript in Fall 2008 during his sabbatical leave.
M.~G.~was supported in part by NSF Grant DMS \#0801204. 
M.~S.~was supported in part by NSF Grants DMS \#0800671 and PHY \#0555346.  
A.~V.~was supported in part by ISF Grant \#1032/08.

\bibliographystyle{unsrt}

\end{document}